\documentclass[a4paper,10pt]{article}
\usepackage{amsmath}
\usepackage{amsfonts}
\usepackage{amssymb}
\usepackage{mathrsfs}
\usepackage{pb-diagram}
\usepackage{epstopdf}
\usepackage{enumerate}

\usepackage{amscd,amsthm,curves,enumerate,latexsym,multibox}
\usepackage[all]{xypic}
\usepackage{epstopdf}
\newtheorem{lemma}{Lemma}[section]
\newtheorem{remark}[lemma]{Remark}
\newtheorem{theorem}[lemma]{Theorem}
\newtheorem{corollary}[lemma]{Corollary}
\newtheorem{definition}[lemma]{Definition}
\newtheorem{proposition}[lemma]{Proposition}
\newtheorem{axiom}[lemma]{Axiom}
\newtheorem{assumption}[lemma]{Assumption}
\begin{document}




\setlength{\baselineskip}{0.780cm} \pagestyle{empty}
\begin{center}
\Large{ \bf  Donaldson-Thomas Theory for \\ Calabi-Yau Four-folds}
\end{center}

\vspace{16mm}
\begin{center}
Yalong Cao
\end{center}

\vspace{16mm}
\begin{center}
A Thesis Submitted in Partial Fulfilment\\ of the Requirements for the Degree of\\
Master of Philosophy \\ in \\ Mathematics
\end{center}

\vspace{16mm}
\begin{center}
The Chinese University of Hong Kong \\
July 2013
\end{center}

\newpage
\setcounter{page}{1}
\pagestyle{myheadings}
\markright{Donaldson-Thomas theory for Calabi-Yau four-folds}

\noindent
{\Huge {\bf Abstract}}
\vspace{1.2cm}

\noindent

Let $X$ be a complex four-dimensional compact Calabi-Yau manifold equipped with a  K\"ahler form $\omega$ and a holomorphic four-form $\Omega$.
Under certain assumptions, we define Donaldson-Thomas type deformation invariants by studying the moduli space of the solutions of Donaldson-Thomas
equations on the given Calabi-Yau manifold. We also study sheaves counting on local
Calabi-Yau four-folds. We relate the sheaves countings over $K_{Y}$ with the Donaldson-Thomas invariants for the associated compact three-fold $Y$.
In some very special cases, we prove the DT/GW correspondence for $X$.
Finally, we compute the Donaldson-Thomas invariants of certain Calabi-Yau four-folds when the moduli spaces are smooth.

\newpage
\begin{center}
{\large {\bf Acknowledgments}}
\end{center}

\vspace{5mm}

I express my deep and sincere gratitude to my supervisor \emph{Naichung Conan Leung} for his invaluable insight, idea and encouragement.
It is him who brought me to the fascinated area of higher dimensional gauge theory and spent so much time discussing with me.
Without his guidance, the research program would not even start.

I am very grateful to several brilliant mathematicians for their academic discussions. During my first
semester in CUHK, I was so fortunate to have many discussions with \emph{Huailiang Chang} on obstruction theory and discussions with \emph{Weiping Li} on sheaf moduli spaces and his paper \cite{lq}. I also want to thank \emph{Zheng Hua} for explaining his works and useful discussions.

I also want to express my gratitude to my previous teachers \emph{Guangyuan Zhang}, \emph{Ming Xu} and \emph{Lieming Li} when I was an undergraduate student. They always gave me so many encouragement when I met difficulty during my self-learning in mathematics. Without them, I might already give up.

Let me also thank my fellow colleagues in the Institute of
Mathematical Science and Mathematics Department of The Chinese University of Hong Kong, including \emph{Kai-Leung Chan, Yunxia Chen, Nikolas Ziming Ma, Yi Zhang, Yin Li}. I also enjoyed so many discussions with my former academic brothers \emph{Kwokwai Chan}, \emph{Jiajing Zhang} and \emph{Changzheng Li}. I express my special gratitude to \emph{Yi Zhang} for helping me a lot in my application to the graduate school of CUHK, and special thanks to \emph{Yin Li} for teaching me so much on Tex-typing.

Finally, I would like to give special thanks to my family, to my mother, my grandma and grandpa who always give me unlimited support and care. This thesis is dedicated to them.

\newpage

\tableofcontents
\newpage

\section{Introduction}
In this thesis we study Donaldson-Thomas theory for Calabi-Yau four-folds.
Originally Floer studied Chern-Simons theory for oriented real three dimensional closed manifolds and
defined instanton Floer homology generalizing the Casson invariants. For oriented real four dimensional closed manifolds, Donaldson \cite{d},\cite{dk}
defined a polynomial invariant by studying the moduli space of anti-self dual connections on $SU(2)$ bundles over the given four manifold.
Thomas \cite{th} then studied complex analogue of Chern-Simons gauge theory on Calabi-Yau three-folds and defined the so called
Donaldson-Thomas invariants. As complex analogue of
Donaldson theory for Calabi-Yau four-folds, we study equations written by Donaldson and Thomas and define the corresponding invariants under certain assumptions. \\
${}$ \\
\textbf{Notation and convention}. Throughout this thesis, unless specified otherwise, $\big(X,\mathcal{O}(1)\big)$ will be
a polarized compact connected Calabi-Yau four-fold equipped with a K\"ahler form $\omega$ and a holomorphic four-form $\Omega$
such that $c_{1}(\mathcal{O}(1))=[\omega]$.

By Yau's celebrated theorem proving Calabi's conjecture,
there is a Ricci flat K\"ahler metric $g$ on $X$ such that $\Omega\wedge\overline{\Omega}=dvol$, where $dvol$ is the volume form of $g$. \\
${}$ \\
We denote $(E,h)$ to be a complex vector bundle with a Hermitian metric over $X$ and $G$ to be the structure group of $E$ with center $C(G)$. \\
${}$ \\
We denote $\mathcal{A}$ to be the space of all $L_{k}^{2}$ (Sobolev norm) unitary connections on $E$ and $\mathcal{G}$ to be the $L_{k+1}^{2}$ unitary gauge transformation group, where $k$ is a large enough positive integer. Denote $\Omega^{0,i}(X,EndE)_{k}$ to be the completion of $\Omega^{0,i}(X,EndE)$ by $L_{k}^{2}$ norm.\\
${}$ \\
We denote the space of irreducible unitary connections by
\begin{equation}\mathcal{A}^{*}=\{A\in \mathcal{A} \textrm{ } | \textrm{ } \Gamma_{A}=C(G)\}, \nonumber \end{equation}
where $\Gamma_{A}=\{u\in\mathcal{G} \textrm{ } | \textrm{ } u(A)=A \}$ is the isotropic group at $A$.
$\mathcal{A}^{*}$ is a dense open subset of $\mathcal{A}$ \cite{dk}.
Let $\mathcal{G}^{0}=\mathcal{G}/C(\mathcal{G})$ be the reduced gauge group. We know the action $\mathcal{G}^{0}$ on $\mathcal{A}^{*} $ is free.
Define $\mathcal{B}_{1}=\mathcal{A}^{*}/\mathcal{G}^{0}$, which is a Banach manifold \cite{d}, \cite{fu}. \\
${}$ \\
We denote $\mathcal{M}_{c}(X,\mathcal{O}(1))$ or simply $\mathcal{M}_{c}$ to be the Gieseker moduli space of $\mathcal{O}(1)$-stable sheaves with given Chern character $c$.
We always assume $\mathcal{M}_{c}$ is compact, i.e. $\mathcal{M}_{c}=\overline{\mathcal{M}}_{c}$ ($\overline{\mathcal{M}}_{c}$ is the Gieseker moduli space of semi-stable sheaves)
which is satisfied under the coprime condition of degree and rank of the coherent sheaves \cite{hl}.
Let $\mathcal{M}_{c}^{o}$ be the analytic open subspace of $\mathcal{M}_{c}$ consisting of slope-stable holomorphic bundles which is possibly empty.

In this thesis, when we say $\mathcal{M}_{c}$ is smooth, we always mean it in the strong sense, namely the Kuranishi maps are zero. \\
${}$ \\
We define
\begin{equation}*_{4}: \Omega^{0,2}(X)\rightarrow \Omega^{0,2}(X),
\nonumber \end{equation}
\begin{equation}\alpha\wedge *_{4}\beta=(\alpha,\beta)_{g}\overline{\Omega}.
\nonumber \end{equation}
Coupled with bundle $(E,h)$, it is extended to
\begin{equation}*_{4}: \Omega^{0,2}(X,EndE)\rightarrow \Omega^{0,2}(X,EndE)
\nonumber \end{equation}
with $*_{4}^{2}=1$ \cite{dt}. Hence we can use this $*_{4}$ operator to define the (anti) self dual subspace of $\Omega^{0,2}(X,EndE)$ and furthermore $*_{4}$ splits the corresponding harmonic subspace into self dual and anti-self dual parts.

Then the $DT_{4}$ equation (\ref{complex ASD equation}) is defined to be
\begin{equation} \left\{ \begin{array}{l}
  F^{0,2}_{+}=0 \\ F\wedge\omega^{3}=0 ,    
\end{array}\right.
\nonumber \end{equation}
where the first equation is $F^{0,2}+*_{4}F^{0,2}=0$ and we assume $c_{1}(E)=0$ for simplicity in the moment map equation $ F\wedge\omega^{3}=0$. \\
${}$ \\
We denote $\mathcal{M}^{DT_{4}}(X,g,[\omega],c,h)$ or simply $\mathcal{M}^{DT_{4}}_{c}$ to be the space of gauge equivalence classes of solutions of the $DT_{4}$ equation
(\ref{complex ASD equation}) with respect to the given Chern character $c=ch(E)$. $r=2-\chi(E,E)$ denotes the corresponding real virtual dimension,
where $\chi(E,E)\triangleq\sum_{i}(-1)^{i}h^{i}(X,EndE)$. \\
${}$ \\
To define Donaldson type invariants using $\mathcal{M}^{DT_{4}}_{c}$, we need
\begin{center}(1) orientation, \quad (2) compactness,  \quad (3) transversality.  \end{center}
The orientability issue for $\mathcal{M}^{DT_{4}}_{c}$ is concerning the determinant line bundle $\mathcal{L}$ of the index bundle of the twisted Dirac operators.
We remark that if $\mathcal{M}_{c}^{o}\neq\emptyset$,
then $\mathcal{M}_{c}^{DT_{4}}=\mathcal{M}_{c}^{o}$ as sets (Theorem \ref{mo mDT4}). In this case,
\begin{equation}\mathcal{L}|_{E}=\big(\wedge^{top}Ext^{2}_{+}(E,E)\big)^{-1}\otimes \wedge^{top}Ext^{1}(E,E), \nonumber \end{equation}
where $Ext^{2}_{+}(E,E)$ is the self-dual subspace of $Ext^{2}(E,E)$. We always assume $\mathcal{L}$ is oriented (there are some partial results towards the orientability listed in the appendix).

Note that $w_{1}(\mathcal{L})=w_{1}(Ind)$, where $Ind$ is the index virtual bundle over $\mathcal{M}_{c}^{DT_{4}}$ for the operator
\begin{equation} \Omega^{1}(X,g_{E})_{k}\rightarrow \Omega^{0,2}_{+}(X,EndE)_{k-1}\oplus\Omega^{0}(X,g_{E})_{k-1}\oplus\Omega^{0}(X,g_{E})_{k-1},
\nonumber\end{equation}
\begin{equation}a=a^{1,0}+a^{0,1}\mapsto (\pi_{+}\overline{\partial}_{A}a^{0,1}, d_{A}^{*}a, {d_{A}^{c}}^{*}a).
\nonumber \end{equation}
where $A\in\mathcal{M}_{c}^{DT_{4}}$. We have
\begin{equation}Ind|_{E}=Ext^{1}(E,E)-Ext^{2}_{+}(E,E).  \nonumber \end{equation}
Define the complexified index bundle $Ind_{\mathbb{C}}\triangleq Ind\otimes_{\mathbb{R}}\mathbb{C}$. By Remark \ref{remark1},
\begin{equation}Ind_{\mathbb{C}}|_{E}\cong Ext^{1}(E,E)-Ext^{2}(E,E)+Ext^{3}(E,E),\nonumber \end{equation}
which has the Serre duality quadratic form $Q_{Serre}$.
Similar to a construction in $K$-theory \cite{atiyah}, there exists a trivial bundle with the trivial standard quadratic form $(\mathbb{C}^{N},q)$ such that
$(Ind_{\mathbb{C}},Q_{Serre})\oplus (\mathbb{C}^{N},q)$ becomes a quadratic bundle (a vector bundle with a non-degenerate quadratic form). If $c_{1}(Ind_{\mathbb{C}})=0$, then the structure group of the quadratic bundle can be reduced to $SO(n,\mathbb{C})$ which makes the structure group of the corresponding real bundle inside $SO(n,\mathbb{R})$ \cite{eg}.
However, the above reduction of the structure group to $SO(n,\mathbb{C})$ involves a choice on each component of $\mathcal{M}_{c}^{DT_{4}}$ \cite{eg}, which corresponds exactly to a choice of orientation of $\mathcal{L}$ on $\mathcal{M}_{c}^{DT_{4}}$.
\begin{definition}\label{nat cpx ori}
If $c_{1}(Ind_{\mathbb{C}})=0$, we call the above $SO(n,\mathbb{C})$-reduction a choice of orientation for $Ind_{\mathbb{C}}$. If the corresponding real bundle has a complex orientation, then we call $Ind_{\mathbb{C}}$ has a natural complex orientation, denoted by $o(\mathcal{O})$.
\end{definition}
We note that, $Ind_{\mathbb{C}}$ has the advantage over $Ind$ by being well defined also on $\mathcal{M}_{c}$
whereas $\mathcal{L}$ can be defined even when $c\notin\bigoplus_{k}\textrm{ }H^{k,k}(X)$.

To pick a coherent choice of orientation for all components of the moduli space, as in Donaldson theory \cite{dk}, we need to extend the index bundle and its determinant line bundle to some big connected space such that $\mathcal{M}_{c}^{DT_{4}}$ (or $\mathcal{M}_{c}$) embeds inside with induced index (or determinant line) bundle.
We know $\mathcal{M}_{c}^{DT_{4}}\hookrightarrow \mathcal{B}_{1}$, where the determinant line bundle $\mathcal{L}$ extends naturally.

For $\mathcal{M}_{c}$ with $Hol(X)=SU(4)$, by Seidel-Thomas twist \cite{js}, we can identify it with some component of $\mathcal{M}_{bdl}$, the moduli space of simple holomorphic bundles with some fixed Chern class, where the index bundle $Ind_{\mathbb{C}}$ also extends. By choosing a Hermitian metric, we can further imbed $\mathcal{M}_{bdl}$ into the space of gauge equivalence classes of irreducible unitary connections $\mathcal{B}^{*}$, where the determinant line bundle $\mathcal{L}$ of the operator mentioned above is defined. Note that, one choice of orientation of $\mathcal{L}$ gives an orientation for $Ind_{\mathbb{C}}$ on $\mathcal{M}_{bdl}$.

By \cite{dk}, $\mathcal{B}_{1}$ and $\mathcal{B}^{*}$ are connected, there are only two orientations for any orientable bundle. We assume from now on that the determinant line bundle $\mathcal{L}$ on $\mathcal{B}^{*}$ (or $\mathcal{B}_{1}$) is oriented.
\begin{definition}
We define an orientation data, denoted by $o(\mathcal{L})$ to be a choice of orientation for $\mathcal{L}$ on $\mathcal{M}_{c}^{DT_{4}}$ which is induced from an orientation of $\mathcal{L}$ on $\mathcal{B}_{1}$ or a choice of an orientation of $Ind_{\mathbb{C}}$ on $\mathcal{M}_{c}$ which is induced from an orientation of $\mathcal{L}$ on $\mathcal{B}^{*}$.
\end{definition}
\begin{remark} ${}$ \\
1. The orientation data may involve a choice of Seidel-Thomas twist for $\mathcal{M}_{c}$. \\
2. Making a choice of orientation on $\mathcal{L}$ from the ambient space $\mathcal{B}^{*}$ is for the purpose of deformation invariance of the theory. If we have natural orientation for $Ind_{\mathbb{C}}$ on $\mathcal{M}_{c}$, such as  $Ind_{\mathbb{C}}$ has a natural complex orientation (Definition \ref{nat cpx ori}, as we will see there are plenty such examples), then we will just use that natural orientation without referring to $\mathcal{B}^{*}$ and we assume the natural orientation can also be induced from an orientation of $\mathcal{L}$ on $\mathcal{B}^{*}$.
\end{remark}

To compactify the $DT_{4}$ moduli space $\mathcal{M}_{c}^{DT_{4}}$, we start with its local Kuranishi structure (Theorem \ref{Kuranishi str of cpx ASD thm}).
\begin{theorem}\label{mo mDT4}
If $\mathcal{M}_{c}^{o}\neq \emptyset$, the local Kuranishi model of $\mathcal{M}_{c}^{DT_{4}}$ near $d_{A}$ with $F^{0,2}_{A}=0$ can be described as
\begin{equation} \xymatrix@1{
\kappa_{+}=\pi_{+}(\kappa): H^{0,1}(X,EndE) \ar[r]^{\quad \quad \quad \kappa}
& H^{0,2}(X,EndE)\ar[r]^{\pi_{+}} & H^{0,2}_{+}(X,EndE) },  \nonumber \end{equation}
where $\kappa$ is a canonical Kuranishi map for $\mathcal{M}_{c}^{o}$ near $\overline{\partial}_{A}$ (the (0,1) part of $d_{A}$) determined by the $DT_{4}$ equation and $\pi_{+}$ is projection to the self-dual forms.

Furthermore, the closed imbedding between analytic spaces possibly with non-reduced structures
\begin{equation}\mathcal{M}_{c}^{o}\hookrightarrow \mathcal{M}_{c}^{DT_{4}}  \nonumber \end{equation}
is also a bijective map on closed points.
\end{theorem}
In general, we want to obtain a compactification of the $DT_{4}$ moduli space $\overline{\mathcal{M}}_{c}^{DT_{4}}$ by extending the above bijective map $\mathcal{M}_{c}^{o}\rightarrow \mathcal{M}_{c}^{DT_{4}}$ to $\mathcal{M}_{c}\rightarrow\overline{\mathcal{M}}_{c}^{DT_{4}}$ while
the local analytic structure of $\overline{\mathcal{M}}_{c}^{DT_{4}}$ is given by $\kappa_{+}^{-1}(0)$, where
\begin{equation}\kappa_{+}=\pi_{+}(\kappa): Ext^{1}(\mathcal{F},\mathcal{F})\rightarrow  Ext^{2}_{+}(\mathcal{F},\mathcal{F}) \nonumber \end{equation}
and $\kappa$ is a Kuranishi map for $\mathcal{M}_{c}$ at $\mathcal{F}$.

Note that $\mathcal{M}_{c}$ may not contain any locally free sheaf, but the above gluing approach to define an analytic space $\overline{\mathcal{M}}_{c}^{DT_{4}}$ could still make sense.
We then call $\overline{\mathcal{M}}_{c}^{DT_{4}}$ the generalized $DT_{4}$ moduli space (Definition \ref{generalized DT4}) if it comes from gluing local models of the above type and can be identified with $\mathcal{M}_{c}$ as sets.
The name generalized $DT_{4}$ moduli space comes from the fact that it may not parameterize any locally free sheaf in general while the $DT_{4}$ moduli space consists of connections on bundles only.

It is then obvious that if $\mathcal{M}_{c}=\mathcal{M}_{c}^{o}\neq\emptyset$, $\overline{\mathcal{M}}_{c}^{DT_{4}}$ exists and $\overline{\mathcal{M}}_{c}^{DT_{4}}=\mathcal{M}_{c}^{DT_{4}}$.
We have the following less obvious gluing results (Proposition \ref{gene DT4 if Mc smooth}, \ref{ob=v+v*}).
\begin{proposition}\label{condition of vir gene DT4}
If (i) $\mathcal{M}_{c}$ is smooth or (ii) for any closed point $\mathcal{F}\in \mathcal{M}_{c}$, there exists a complex vector space $V_{\mathcal{F}}$ and a linear isometry
\begin{equation}(Ext^{2}(\mathcal{F},\mathcal{F}),Q_{Serre})\cong (T^{*}V_{\mathcal{F}},Q_{std})  \nonumber \end{equation}
such that the image of a Kuranishi map $\kappa$ of $\mathcal{M}_{c}$ at $\mathcal{F}$ satisfies
\begin{equation}\emph{Image}(\kappa)\subseteq V_{\mathcal{F}},  \nonumber \end{equation}
where $Q_{Serre}$ is the Serre duality pairing and $Q_{std}$ is induced from the standard pairing between $V_{\mathcal{F}}$ and $V_{\mathcal{F}}^{*}$.
Then the generalized $DT_{4}$ moduli space exists and $\overline{\mathcal{M}}^{DT_{4}}_{c}=\mathcal{M}_{c}$ as real analytic spaces.
\end{proposition}
After building up the moduli space, we have to handle the transversality issue, i.e. making sense of its fundamental class despite the fact that it may contain many components of different dimensions.

Firstly, we show that when the $DT_{4}$ moduli space is compact, its virtual fundamental class exists.
\begin{theorem}(Theorem \ref{main theorem}) \\
Assume $\mathcal{M}_{c}=\mathcal{M}_{c}^{o}\neq\emptyset$ and assume there exists an orientation data $o(\mathcal{L})$.
Then $\mathcal{M}^{DT_{4}}_{c}$ is compact and its virtual fundamental class exists as a cycle
$[\mathcal{M}^{DT_{4}}_{c}]^{vir}\in H_{r}(\mathcal{B}_{1},\mathbb{Z})$.

Furthermore, if the above assumptions are satisfied by a continuous family of Calabi-Yau four-folds $X_{t}$ parameterized by $t\in [0,1]$, then the cycle in $H_{r}(\mathcal{B}_{1},\mathbb{Z})$ is independent of $t$.
\end{theorem}
We also define the virtual fundamental class (Definition \ref{virtual cycle when Mc smooth}, \ref{virtual cycle when ob=v+v*}) for the above two cases when the generalized $DT_{4}$ moduli space exists and $\overline{\mathcal{M}}^{DT_{4}}_{c}=\mathcal{M}_{c}$.

As in the case of Donaldson theory \cite{dk}, we can use the $\mu$-map to cut down the degree of the virtual fundamental class and define the corresponding $DT_{4}$ invariants (Definition \ref{DT4 inv of bundles}, \ref{DT4 inv of sheaves}).

Since we can only define $DT_{4}$ invariants in several cases under different assumptions, to make all cases consistent, we propose several axioms that $DT_{4}$ invariants should satisfy. Axioms $(3)$-$(6)$ are showed in the thesis and axioms $(1)$,$(2)$ are verified when we have definition of virtual fundamental classes of the (generalized) $DT_{4}$ moduli spaces.  \\
${}$ \\
\textbf{Axioms of $DT_{4}$ invariants}:
\begin{axiom}Given a triple $(X,\mathcal{O}(1),c)$ and an auxiliary choice of an orientation data $o(\mathcal{L})$, where $(X,\mathcal{O}(1))$ is a polarized Calabi-Yau four-fold,
$c\in H^{even}_{c}(X,\mathbb{Q})$ is a (compactly supported) cohomology class,
the $DT_{4}$ invariant (i.e. Donaldson-Thomas four-folds invariant) of this quadruple, denoted by $DT_{4}(X,\mathcal{O}(1),c,o(\mathcal{L}))$ is a map
\begin{equation}DT_{4}(X,\mathcal{O}(1),c,o(\mathcal{L})): GrSym^{*}\big(H_{*}(X,\mathbb{Z})\otimes \mathbb{Z}[x_{1},x_{2},...]\big)
\rightarrow \mathbb{Z}, \nonumber \end{equation}
$($$GrSym$ means graded symmetric with respect to the parity of the degree of $H_{*}(X)$ $)$ satisfying : \\
$\textbf{(1)}$ \textbf{Orientation reversed}
\begin{equation}DT_{4}(X,\mathcal{O}(1),c,o(\mathcal{L}))=-DT_{4}(X,\mathcal{O}(1),c,-o(\mathcal{L})),  \nonumber \end{equation}
where $-o(\mathcal{L})$ denotes the opposite orientation of $o(\mathcal{L})$. \\
$\textbf{(2)}$ \textbf{Deformation invariance}
\begin{equation}DT_{4}(X_{0},\mathcal{O}(1)|_{X_{0}},c,o(\mathcal{L}_{0}))=DT_{4}(X_{1},\mathcal{O}(1)|_{X_{1}},c,o(\mathcal{L}_{1})),  \nonumber \end{equation}
where $(X_{t},\mathcal{O}(1))$ is a continuous family of complex structures and $o(\mathcal{L}_{t})$ is an orientation data on the family determinant line bundle with $t\in [0,1]$. \\
$\textbf{(3)}$ \textbf{Vanishing for compact hyper-K\"ahler manifolds  }
\begin{equation}DT_{4}(X,\mathcal{O}(1),c,o(\mathcal{L}))=0,  \nonumber \end{equation}
when $X$ is compact hyper-K\"ahler (\ref{nu+}).  \\
$\textbf{(4)}$ \textbf{$DT_{4}/DT_{3}$ correpsondence}
\begin{equation}DT_{4}(X,\pi^{*}\mathcal{O}_{Y}(1),c,o(\mathcal{O}))=DT_{3}(Y,\mathcal{O}_{Y}(1),c^{'}), \nonumber \end{equation}
where $\pi: X=K_{Y}\rightarrow Y$ is projection and $(Y,\mathcal{O}_{Y}(1))$ is a polarized compact Fano threefold with $H^{0}(Y,K_{Y}^{-1})\neq0$. $c=(0,c|_{H_{c}^{2}(X)}\neq 0,c|_{H_{c}^{4}(X)},c|_{H_{c}^{6}(X)},c|_{H_{c}^{8}(X)})$ and $\mathcal{M}_{c}$ consists of slope-stable sheaves.

In this setup, sheaves in $\mathcal{M}_{c}$ is of type $\iota_{*}(\mathcal{F})$ where $\iota: Y\rightarrow K_{Y}$ is the zero section and $c^{'}=ch(\mathcal{F})\in H^{even}(Y)$ is uniquely determined by $c$. $o(\mathcal{O})$ is the natural complex orientation for $Ind_{\mathbb{C}}$ over $\mathcal{M}_{c}$.

$DT_{3}(Y,\mathcal{O}_{Y}(1),c^{'})$ is the $DT_{3}$ invariant of $(Y,\mathcal{O}_{Y}(1))$ with certain insertion fields (Theorem \ref{compact supp DT4}). \\
$\textbf{(5)}$ \textbf{Normalization 1}
\begin{equation}DT_{4}(X,\mathcal{O}(1),c,o(\mathcal{L}))=DT_{4}^{\mu_{1}}(X,\mathcal{O}(1),c,o(\mathcal{L})), \nonumber \end{equation}
if $X$ is compact and $\mathcal{M}_{c}=\mathcal{M}_{c}^{o}\neq\emptyset$.

$DT_{4}^{\mu_{1}}(X,\mathcal{O}(1),c,o(\mathcal{L}))$ is defined using virtual fundamental class of $\mathcal{M}_{c}^{DT_{4}}$ and the corresponding $\mu$-map (\ref{mu map for bundles}). \\
$\textbf{(6)}$ \textbf{Normalization 2}
\begin{equation}DT_{4}(X,\mathcal{O}(1),c,o(\mathcal{L}))=DT_{4}^{\mu_{2}}(X,\mathcal{O}(1),c,o(\mathcal{L})), \nonumber \end{equation}
if $\mathcal{M}_{c}\neq\emptyset$ is smooth or satisfies the condition in definition \ref{virtual cycle when ob=v+v*}.

$DT_{4}^{\mu_{2}}(X,\mathcal{O}(1),c,o(\mathcal{L}))$ is defined using virtual fundamental class of $\overline{\mathcal{M}}_{c}^{DT_{4}}$ and the corresponding $\mu$-map (\ref{u2 map}). \\
\end{axiom}
In normalization axioms, the construction depends on the existence of the virtual fundamental class and the $\mu$-map descendent
fields as mentioned above. Throughout the thesis, we will often only mention the $DT_{4}$ virtual cycle (the virtual fundamental class of the generalized $DT_{4}$ moduli space) instead of using
the corresponding $DT_{4}$ invariants for convenience purposes. \\

Similar to the case of Calabi-Yau three-folds, we also study the $DT/GW$ correspondence for compact Calabi-Yau four-folds in some specific cases.
We have theorem \ref{DT=GW}, \ref{DT=GW2}.
\begin{theorem}
Let $X$ be a compact Calabi-Yau four-fold. If $\mathcal{M}_{c}$ with the given Chern character $c=(1,0,0,-PD(\beta),-1)$
is smooth and consists of ideal sheaves of smooth connected genus zero imbedded curves only.
Assume the $GW$ moduli space $\overline{\mathcal{M}}_{0,0}(X,\beta)\cong \mathcal{M}_{c}$ and use the natural complex orientation $o(\mathcal{O})$, then $\overline{\mathcal{M}}_{c}^{DT_{4}}$ exists and $\overline{\mathcal{M}}_{c}^{DT_{4}}\cong \overline{\mathcal{M}}_{0,0}(X,\beta)$. Furthermore,

$(1)$ if $Hol(X)=SU(4)$, then $[\overline{\mathcal{M}}_{c}^{DT_{4}}]^{vir}=[\overline{\mathcal{M}}_{0,0}(X,\beta)]^{vir}$.

$(2)$ if $Hol(X)=Sp(2)$, i.e compact irreducible hyper-K\"ahler, \\  then $[\overline{\mathcal{M}}_{c}^{DT_{4}}]^{vir}=0$.

Furthermore, $[\overline{\mathcal{M}}_{c}^{DT_{4}}]^{vir}_{hyper-red}=[\overline{\mathcal{M}}_{0,0}(X,\beta)]^{vir}_{red}$ (see theorem \ref{DT=GW2})
\end{theorem}
When $X=T^{*}S$, we only consider sheaves with scheme theoretical support inside $S$ (it is of type $\iota_{*}\mathcal{F}$, where $\iota: S\hookrightarrow X$). To ensure that they form components of moduli space of sheaves on $X$. We make the following assumption
\begin{equation}\label{introduction vanishing}Ext^{0}_{S}(\mathcal{F},\mathcal{F}\otimes \Omega_{S}^{1})=0, \quad
Ext^{2}_{S}(\mathcal{F},\mathcal{F})=0,  \end{equation}
which is satisfied when (i) $S=\mathbb{P}^{2}$, $\mathcal{F}$ is torsion-free slope stable or (ii) $S$ is del-Pezzo, $\mathcal{F}$ is an ideal sheaf of points.

We then denote $\mathcal{M}_{c}^{S_{cpn}}=\{\iota_{*}\mathcal{F} \textrm{ } | \mathcal{F}\in \mathcal{M}_{c}(S) \}$ to be the components of moduli space of sheaves on $X$ which can be identified with $\mathcal{M}_{c}(S)$ (moduli of stable sheaves on $S$ with Chern character $c\in H^{even}(S)$). Under the above assumptions, we know $\mathcal{M}_{c}^{S_{cpn}}$ is smooth and it can be showed that its virtual dimension is negative which leads to the vanishing of the virtual cycle. However, after taking away the trivial part of the obstruction bundle and consider the reduced virtual cycle (Definition \ref{red vir cycle}), we have
\begin{theorem}
Let $X=T^{*}S$, where $S$ is a compact algebraic surface with $q(S)=0$. Under assumption (\ref{introduction vanishing}), we have $[\mathcal{M}_{c}^{S_{cpn}}]^{vir}=0$. Furthermore,

$(1)$ $[\mathcal{M}_{c}^{S_{cpn}}]^{vir}_{red}=0$, when $c|_{H^{0}(S)}\geq2$.

$(2)$ $[\mathcal{M}_{c}^{S_{cpn}}]^{vir}_{red}=1$, when $c=(1,c|_{H^{2}(S)},0)$.

$(3)$ $[\mathcal{M}_{c}^{S_{cpn}}]^{vir}_{red}=e(Hilb^{n}(S))$, when $c=(1,0,-n)$, $n\geq 1$. \\
Moreover, they fit into the following generating function
\begin{equation}\sum_{n\geq0}[\mathcal{M}_{(1,0,-n)}^{S_{cpn}}]^{vir}_{red}q^{n}=\prod_{k\geq1}\frac{1}{(1-q^{k})^{e(S)}}_{.} \nonumber \end{equation}
\end{theorem}

Actually, Li-Qin \cite{lq} had provided examples when $\mathcal{M}_{c}=\mathcal{M}_{c}^{o}\neq\emptyset$. By studying their examples, we have theorem \ref{liqin eg}.
\begin{theorem}
Let $X$ be a generic smooth hyperplane section in $W=\mathbb{P}^{1}\times\mathbb{P}^{4}$ of bi-degree $(2,5)$.
Let
\begin{equation}cl=[1+(-1,1)|_{X}]\cdot[1+(\epsilon_{1}+1,\epsilon_{2}-1)|_{X}],
\nonumber
\end{equation}
\begin{equation}k=(1+\epsilon_{1})\left(\begin{array}{l}6-\epsilon_{2} \\ \quad 4\end{array}\right), \quad \epsilon_{1},\epsilon_{2}=0,1.
\nonumber\end{equation}
Denote $\overline{\mathcal{M}}_{c}(L_{r})$ to be the moduli space of Gieseker $L_{r}$-semistable rank-2
torsion-free sheaves with Chern character $c$ (which can be easily read from the total Chern class $cl$), where $L_{r}=\mathcal{O}_{W}(1,r)|_{X}$. \\
$(1)$ If
\begin{equation}\frac{15(2-\epsilon_{2})}{6+5\epsilon_{1}+2\epsilon_{2}}<r<\frac{15(2-\epsilon_{2})}{\epsilon_{1}(1+2\epsilon_{2})},
\nonumber\end{equation}
then $\overline{\mathcal{M}}_{c}^{DT_{4}}$ exists and $\overline{\mathcal{M}}_{c}^{DT_{4}}=\overline{\mathcal{M}}_{c}(L_{r})=\mathbb{P}^{k}$, $[\overline{\mathcal{M}}_{c}^{DT_{4}}]^{vir}=[\mathbb{P}^{k}]$. \\
${}$ \\
$(2)$ If
\begin{equation} 0<r<\frac{15(2-\epsilon_{2})}{6+5\epsilon_{1}+2\epsilon_{2}},\nonumber\end{equation}
then $\overline{\mathcal{M}}_{c}^{DT_{4}}=\emptyset $ and $[\overline{\mathcal{M}}_{c}^{DT_{4}}]^{vir}=0$.
\end{theorem}
Lastly, for the case of ideal sheaves of one point, one have theorem \ref{moduli of one point}.
\begin{theorem}
If $Hol(X)=SU(4)$ and $c=(1,0,0,0,-1)$, then $\overline{\mathcal{M}}_{c}^{DT_{4}}$ exists and $\overline{\mathcal{M}}_{c}^{DT_{4}}=X$, $[\overline{\mathcal{M}}_{c}^{DT_{4}}]^{vir}=\pm PD\big(c_{3}(X)\big)$.
\end{theorem}
We also study the equivariant $DT_{4}$ theory for ideal sheaves of curves and points. We define the corresponding equivariant $DT_{4}$ invariants
for toric Calabi-Yau four-folds by the virtual localization formula \cite{gp}. Because the torus fixed points of the Gieseker moduli space are discrete, we
do not need $\overline{\mathcal{M}}_{c}^{DT_{4}}$ to define invariants. Hence we get the definition without constraints on moduli spaces. \\
${}$ \\
\textbf{Content of the thesis }: In section 2, we will study the $*_{4}$ operator on the space of bundle valued differential forms which
is key to the definition of $DT_{4}$ equations and the construction of moduli spaces. In section 3, we study the local structure of the $DT_{4}$ moduli spaces and
construct the corresponding finite dimensional model. In section 4, we compactify the $DT_{4}$ moduli spaces under certain assumptions. Under the gluing assumption \ref{assumption on gluing}, we define the
generalized $DT_{4}$ moduli spaces. We show for some cases the gluing assumption \ref{assumption on gluing} is easy to check. In section 5, we construct the virtual fundamental class of (generalized) $DT_{4}$ moduli spaces. In section 6, we study the $DT_{4}$ theory for compactly supported sheaves on local Calabi-Yau four-folds. In section 7,
we define the equivariant $DT_{4}$ invariants for toric $CY_{4}$.
The lase section consists of several calculations of $DT_{4}$ invariants when $\mathcal{M}_{c}$ is smooth.

\newpage

\section{The $*_{4}$ operator}

\subsection{The $*_{4}$ operator for bundles}
In this section, we introduce the $*_{4}$ operator on the space of bundle valued differential forms which
is key to the construction of $DT_{4}$ equations and moduli spaces.

As proposed by Donaldson and Thomas in \cite{dt},
We have a commutative diagram
\begin{equation}   \xymatrix{
                      \Omega^{0,k}(X) \ar[rr]^{*} \ar[dr]_{*_{4}}
                        &  &    \Omega^{4,4-k}(X)  \ar[dl]^{\lrcorner\Omega }    \\
                                  & \Omega^{0,4-k}(X).     } \nonumber\end{equation}
The commutativity of the above diagram defines the  $*_{4}$ operator, i.e.
\begin{equation} \alpha\wedge*_{4}\alpha=|\alpha|^{2}\overline{\Omega}.\nonumber \end{equation}
The $*$ here is the usual $\mathbb{C}$-anti-linear Hodge star and $\alpha$ is a $(0,k)$ form on the the given Calabi-Yau four-fold $X$. The $*_{4}$ also
satisfies the identity $*_{4}^{2}=1$.

More genenerally, the above $*_{4}$ operator can be extended to act on\\
$\Omega^{0,k}(X,EndE)$ as follows: take a unitary frame $ \{e_{i}\} $ with respect to the Hermitian metric $h$ on $E$ while $ \{e^{i}\} $
denotes the dual frame. Define
\begin{equation}*_{4}(\alpha\otimes e_{i} \otimes e^{j})=*_{4}(\alpha)\otimes e_{j} \otimes e^{i}.
\nonumber \end{equation}
This is a well defined operator which satisfies the identity $*_{4}^{2}=1$ too.
\begin{remark}If we take $\theta\Omega$ instead of $\Omega$, where $\theta$ is a complex number s.t $\theta\overline{\theta}=1$,
then we get a different $*_{4}$ which also satisfies $\theta\Omega\wedge\overline{\theta}\overline{\Omega}=dvol$.
\end{remark}

Now we want to show the $*_{4}$ operator descends to the harmonic subspace.
\begin{lemma}\label{lem descendence of star}
$\overline{\partial}_{E}^{*}=\overline{\partial}_{E}^{*_{4}}$ acting on $\Omega^{0,\bullet}(X,EndE)$, where $\overline{\partial}_{E}$
is a holomorphic $(0,1)$ connection and  $\overline{\partial}_{E}^{*_{4}}$ is defined to be $*_{4}\overline{\partial}_{E}*_{4}$ up to a sign
which is the same as that appears in $\overline{\partial}_{E}^{*}$.
\end{lemma}
\begin{proof}Take a unitary frame $ \{e_{i}\} $ with respect to Hermitian metric $h$ on $E$.
\begin{eqnarray*}*\overline{\partial}_{E}*(\alpha\otimes e_{i} \otimes e^{j})&=&*\overline{\partial}_{E}(*\alpha\otimes e_{j} \otimes e^{i}) \\
&=&*\big((\overline{\partial}_{E}*\alpha)\otimes e_{j} \otimes e^{i}+(-1)^{|\alpha |} *\alpha\wedge \overline{\partial}_{E}(e_{j} \otimes e^{i})\big) \\
&=&*\overline{\partial}_{E}*(\alpha)\otimes e_{i} \otimes e^{j}+(-1)^{|\alpha |} *(*\alpha\wedge\sum\beta_{k,l}\otimes e_{k} \otimes e^{l})\\
&=&*\overline{\partial}_{E}*(\alpha)\otimes e_{i} \otimes e^{j}+(-1)^{|\alpha |} \sum*(*\alpha\wedge\beta_{k,l})\otimes e_{l} \otimes e^{k}.
\end{eqnarray*}
Similarly, we have
\begin{equation}*_{4}\overline{\partial}_{E}*_{4}(\alpha\otimes e_{i} \otimes e^{j})=*_{4}\overline{\partial}_{E}*_{4}(\alpha)\otimes e_{i}
\otimes e^{j}+(-1)^{|\alpha |}\sum*_{4}(*_{4}\alpha\wedge\beta_{k,l})\otimes e_{l} \otimes e^{k} .\nonumber \end{equation}
Then by the fact that  $*=\Omega\wedge*_{4}$, $\Omega$ is holomorphic with respect to the fixed holomorphic structure of the underlying
Calabi-Yau four-fold and $|\Omega|=1$, we can get the answer directly.
\end{proof}
\begin{corollary}\label{cor star4 cut coho into ASD}
The $*_{4}$ operator splits the space $\Omega^{0,2}(X,EndE)$ into
\begin{equation}\Omega^{0,2}(X,EndE)=\Omega^{0,2}_{+}(X,EndE)\oplus\Omega^{0,2}_{-}(X,EndE) \nonumber\end{equation}
based on the eigenvalues of the operator. Furthermore, it descends to the harmonic subspace and
\begin{equation}H^{0,2}(X,EndE)=H^{0,2}_{+}(X,EndE)\oplus H^{0,2}_{-}(X,EndE).  \nonumber\end{equation}
\end{corollary}
\begin{proof}
By the fact that $*_{4}^{2}=1$ and Lemma \ref{lem descendence of star}.
\end{proof}
\begin{remark}\label{remark1} ${}$\\
1. The above $*_{4}$ operator is naturally extended to the corresponding Banach spaces $\Omega^{0,\bullet}(X,EndE)_{k}$. \\
2. It is obvious that $\sqrt{-1}H^{0,2}_{+}(X,EndE)=H^{0,2}_{-}(X,EndE)$.
Thus these two eigenspaces have the same dimension.
\end{remark}

\subsection{The $*_{4}$ operator for general coherent sheaves}
Now we extend the definition of $*_{4}$ to $Ext^{2}(\mathcal{F},\mathcal{F})$, where $\mathcal{F}$ is any coherent sheaf on $X$.
Because $X$ is projective, we can resolve $\mathcal{F}$ by a complex of holomorphic vector bundles
$E^{*}\rightarrow\mathcal{F}\rightarrow 0$.
Then we form the double complex
\begin{equation}(K^{p,q}=\Omega^{0,q}(X,\mathcal{H}om^{p}(E^{*},E^{*})),\overline{\partial},\delta), \nonumber\end{equation}
where $\overline{\partial}$ is defined in terms of the holomorphic structures on $E^{*}$ and $\delta$ is induced from the differential in the above resolved complex \cite{hl}. There exists two natural filtrations on the total complex $(C^{*},D)$
inducing two spectral sequences abutting to $Ext^{*}(\mathcal{F},\mathcal{F})$, where
\begin{equation} C^{n}\triangleq\bigoplus_{k} \bigoplus_{p+q=n}\Omega^{0,q}(X,\mathcal{H}om(E^{k},E^{k+p}))
\nonumber \end{equation}
and $D=\overline{\partial}+(-1)^{q}\delta$ is the differential of the total complex.
\begin{lemma}\label{double cpx is elliptic}
$(C^{*},D)$ is an elliptic complex.
\end{lemma}
\begin{proof}
Firstly, the operator $D$ is a $C^{\infty}$ differential operator. This is done by recalling
$\delta(\varphi)=\delta_{E}\circ\varphi-(-1)^{deg\varphi}\varphi\circ \delta_{E}$ and $\delta_{E}:E^{k}\rightarrow E^{k+1}$
is a morphism between bundles induced from the differential in the resolved complex. Then $\delta$ is a linear operator on $\Omega^{0,q}(X,\mathcal{H}om (E^{k}, E^{k+p}))$.
While $\overline{\partial}$ is obviously a differential operator, thus $D$ is a differential operator.

Secondly, we fix Hermitian metrics on the underlying topological vector bundles of the resolved complex, denoted by $(E^{k},h_{k})$.
Combining with K\"ahler-Einstein metric and the holomorphic top form $\Omega$,
we can define a operator $*_{4}$ on $\Omega^{0,*}(X,\mathcal{H}om^{*}(E^{*},E^{*}))$ by linearity and a similar construction as before
on each grading piece $\Omega^{0,q}(X,\mathcal{H}om(E^{k},E^{k+p}))$. Then we define $D^{*_{4}}=\pm*_{4}D*_{4}$
(there is a similar argument as Lemma \ref{lem descendence of star}
showing that $D^{*_{4}}=D^{*}$). Define $\Delta_{D}=D^{*}D+DD^{*}$, then $\Delta_{D}=\Delta_{\overline{\partial}}$+ 1st order term of
$\overline{\partial}$ + a linear operator. Thus $D$ is an elliptic operator \cite{wells}.
\end{proof}
\begin{corollary} \label{cutting for sheaves}
The above $*_{4}$ operator on $C^{*}$ descends to $D$-harmonic subspace of the total complex $(C^{*},D)$. Furthermore, it splits $H^{2}(C^{*},D)\cong
Ext^{2}(\mathcal{F},\mathcal{F})$ into
\begin{equation}Ext^{2}(\mathcal{F},\mathcal{F})=Ext^{2}_{+}(\mathcal{F},\mathcal{F})\oplus Ext^{2}_{-}(\mathcal{F},\mathcal{F})
\nonumber \end{equation}
based on the eigenvalues of the operator.
\end{corollary}
\begin{proof}
By a similar argument as Corollary \ref{cor star4 cut coho into ASD} .
\end{proof}
\begin{remark} The subspace $Ext^{2}_{+}(\mathcal{F},\mathcal{F})$ depends on the choice of metrics on bundles in the resolved complex.
To make the above subspaces coherent when $\mathcal{F}$ varies in $\mathcal{M}_{c}$,
we should fix a metric on each complex vector bundle and assume the topological type of these bundles and
the length of the resolved complex are fixed over $\mathcal{M}_{c}$. Since we will not use $Ext^{2}_{+}(\mathcal{F},\mathcal{F})$ in the most general situation in this thesis, we do not discuss it at the moment.
\end{remark}

\newpage
\section{Local Kuranishi structure of $DT_{4}$ moduli spaces}
To study gauge theory on $X$,
we need to consider the space of all connections on $E$ up to gauge equivalences $\mathcal{B}=\mathcal{A}/\mathcal{G}$
and some subspace of it. We are especially interested in the moduli space (actually stack) of holomorphic structures on the underlying
topological vector bundle $E$.
To avoid the large automorphisms, we restrict ourselves to the moduli space of stable holomorphic bundles.
The stability here is with respect to the ample line bundle corresponding to the K\"ahler class $[\omega]$.

By the renowned theorem of Donaldson-Uhlenbeck-Yau \cite{UY}, the moduli space of slope stable bundles with underlying topological structure $E$ is
the moduli space of irreducible holomorphic Hermitian-Yang-Mills connections on $E$.

However, the holomorphic Hermitian-Yang-Mills equations
are unfortunately overdetermined in higher dimensional gauge theory. \\
${}$ \\
\textbf{The $DT_{4}$ moduli space}.
We start with our basic data
\begin{equation}
\begin{array}{lll}
      &  \textrm{ } (E,h)
      \\  &  \quad \textrm{ } \downarrow \\   &   (X,g,\omega).
\end{array}\nonumber\end{equation}
Following the idea of Donaldson and Thomas \cite{dt}, we should use the calibrated form $\Omega$ to cut down the number of equations
in the holomorphic HYM elliptic system. We are thus reduced to the Donaldson-Thomas's complex ASD equations \cite{dt}:
\begin{equation} \left\{ \begin{array}{l} \label{complex ASD equation}          
  F^{0,2}_{+}=0 \\ F\wedge\omega^{3}=0 ,    
\end{array}\right.\end{equation}
where the first equation is $F^{0,2}+*_{4}F^{0,2}=0$.
\begin{remark} ${}$ \\
1. For notation simplicity, we assume $c_{1}(E)=0$ in the moment map equation $F\wedge\omega^{3}=0$. \\
2. The unitary gauge transformation group preserves the above equations.  \\
3. We call the above equations the $DT_{4}$ equations.\\
4. Later we can see from the viewpoint of deformation-obstruction theory, the above cutting of $(0,2)$ curvature term ensures that we have perfect-obstruction theory,
which leads to the construction of virtual fundamental class of the moduli space.
\end{remark}
Then the definition of $DT_{4}$ moduli spaces follows from the above $DT_{4}$ equations.
\begin{definition}
We define the \textbf{$DT_{4}$ moduli space} $\mathcal{M}^{DT_{4}}_{c}$ to be the space of gauge equivalence classes of solutions of
equations (\ref{complex ASD equation}) , i.e. $\mathcal{M}^{DT_{4}}_{c}\hookrightarrow \mathcal{A}^{*}/\mathcal{G}^{0}$
as the zero loci of section $s=(\wedge F,F^{0,2}_{+})$ of the Banach bundle $E$ over $\mathcal{B}_{1}$,
where $\mathcal{B}_{1}=\mathcal{A^{*}}/\mathcal{G}^{0}$ and $E=\mathcal{A^{*}}\times_{\mathcal{G}^{0}}(\Omega^{0}(X,g_{E})_{k-1}\oplus\Omega^{0,2}_{+}(X,EndE)_{k-1})$.
\end{definition}
The $DT_{4}$ moduli space is a real analytic space possibly with non-reduced structure.
\begin{remark} The Banach manifold $\mathcal{A}^{*}/\mathcal{G}^{0}$ involves a choice of a large integer $k $ in the $L_{k}^{2}$ Sobolev norm completion.
By essentially the same argument as Proposition 4.2.16 \cite{dk}, we know the $DT_{4}$ moduli space is independent of the choice of $k$.
Because of this, we will omit $k $ in the notation of $DT_{4}$ moduli spaces.
\end{remark}
${}$ \\
\textbf{ Relations between $\mathcal{M}^{DT_{4}}_{c}$ and $\mathcal{M}_{c}^{o}$ }. By the definition of $\mathcal{M}^{DT_{4}}_{c}$, we obviously
have a map between two sets
\begin{equation}\mathcal{M}_{c}^{o}\rightarrow \mathcal{M}^{DT_{4}}_{c}. \nonumber \end{equation}
By Lemma \ref{C2 condition}, if $ch_{2}(E)\in H^{2,2}(X,\mathbb{C})$, $F^{0,2}_{+}=0\Rightarrow F^{0,2}=0$. Then the map is bijective.   \\
${}$ \\
\textbf{Local structure of $DT_{4}$ moduli spaces}. Since we have identified $\mathcal{M}^{DT_{4}}_{c}$ and $\mathcal{M}_{c}^{o}$ as sets,
now we want to set up the relations between the local analytic structures of these two spaces. In fact, we will show that there exists a closed imbedding $\mathcal{M}_{c}^{o}\hookrightarrow \mathcal{M}^{DT_{4}}_{c}$ between these two analytic spaces possibly with non-reduced structures. \\

We start with a unitary connection $d_{A}\in \mathcal{M}^{DT_{4}}_{c}$, denote its $(0,1)$ part by $\overline{\partial}_{A}$.
We have $F^{0,2}_{A}=0$ under the condition $ch_{2}(E)\in H^{2,2}(X,\mathbb{C})$. We assume without loss of generality that $d_{A}$
is a smooth connection by Proposition 4.2.16 \cite{dk}. \\

By the Hodge decomposition theorem with respect to $h$ on $E$.
\begin{equation}\Omega^{0,2}(X,EndE)_{k-1}=H^{0,2}(X,EndE)\oplus \overline{\partial}_{A}\Omega^{0,1}(X,EndE)_{k}
\oplus\overline{\partial}_{A}^{*}\Omega^{0,3}(X,EndE)_{k}  \nonumber \end{equation}
\begin{equation} I=\mathbb{H}^{0,2}+P_{\overline{\partial}_{A}}+ P_{\overline{\partial}_{A}^{*}} \nonumber \end{equation}
Meanwhile
\begin{equation}*_{4}:\overline{\partial}_{A}\Omega^{0,1}(X,EndE)_{k}\cong\overline{\partial}_{A}^{*}\Omega^{0,3}(X,EndE)_{k}
\nonumber\end{equation}
and
\begin{equation}*_{4}: H^{0,2}(X,EndE)\cong H^{0,2}(X,EndE)  \nonumber\end{equation}
induce
\begin{equation}H^{0,2}(X,EndE)=H^{0,2}_{+}(X,EndE)\oplus H^{0,2}_{-}(X,EndE).  \nonumber\end{equation}
Hence, we know
\begin{equation}F^{0,2}_{+}(d_{A}+a)=0 \Leftrightarrow a''\textrm{}\emph{satisfies} \quad (\ref{equation 3}), \textrm{}\textrm{} (\ref{equation 4}), \nonumber\end{equation}
where (\ref{equation 3}), (\ref{equation 4}) are defined to be
\begin{equation}\label{equation 3}\overline{\partial}_{A}a''+P_{\overline{\partial}_{A}}(a''\wedge a'')
+*_{4}P_{\overline{\partial}_{A}^{*}}(a''\wedge a'')=0 ,   \end{equation}
\begin{equation}\label{equation 4}\pi_{+}\circ\mathbb{H}^{0,2}(a''\wedge a'')=0. \end{equation}
Here $a\in\Omega^{1}(X,g_{E})_{k}$ and $a''\in\Omega^{0,1}(X,EndE)_{k}$ is its (0,1) part. \\

Near $d_{A}$, the $DT_{4}$ moduli space $\mathcal{M}^{DT_{4}}_{c}$ can be described as
\begin{equation} \{a\in\Omega^{1}(X,g_{E})_{k} \big{|}\textrm{ } \| a\|_{k} < \epsilon, \textrm{ } d_{A}^{*}a=0 ,
\textrm{ } d_{A}+a  \textrm{ }  s.t  \textrm{ } (\ref{complex ASD equation}) \}, \nonumber\end{equation}
where $d_{A}^{*}a=0$ is the linear unitary gauge fixing condition and $\epsilon$ is a small positive number. \\

We introduce a new space $\mathcal{M}_{A}^{+}$ which will help us establish the relations of local structures
between $\mathcal{M}_{c}^{DT_{4}}$ and $\mathcal{M}_{c}^{o}$.
\begin{lemma}\label{identify SA MA}
The map takes the unitary connection to $(0,1)$ connection
\begin{equation}d_{A}+a\longmapsto \overline{\partial}_{A}+a'' \nonumber \end{equation}
induces a local isomorphism near the origin
\begin{equation}\{a\in\Omega^{1}(X,g_{E})_{k} \big{|}\textrm{ } \|a\|_{k} < \epsilon, \textrm{ } d_{A}^{*}a=0 , \textrm{ } d_{A}+a  \textrm{ }
s.t \textrm{ }
(\ref{complex ASD equation}) \}\cong \mathcal{M}_{A}^{+},  \nonumber \end{equation}
where
\begin{equation}  \mathcal{M}_{A}^{+}\triangleq\{a''\in\Omega^{0,1}(X,EndE)_{k} \textrm{ }\big{|}\textrm{ }\| a''\|_{k} < \epsilon'',\textrm{ }
a'' \textrm{ } s.t \quad (\ref{equation 3}), \textrm{} (\ref{equation 4}), \textrm{} (\ref{equation 5})  \}
\nonumber  \end{equation}
and (\ref{equation 5}) is defined to be
\begin{equation}\label{equation 5}\overline{\partial}_{A}^{*}a''-\frac{i}{2}\wedge(a'\wedge a''+a''\wedge a')=0 ,  \end{equation}
$\epsilon''$ is a small positive number determined by $\epsilon$ and the above isomorphism.
\end{lemma}
\begin{proof}
Locally, we consider an ambient space of the $DT_{4}$ moduli space
\begin{equation} \{a\in\Omega^{1}(X,g_{E})_{k} \big{|}\textrm{ } \|a\|_{k} < \epsilon, \textrm{ } d_{A}^{*}a=0 ,
\textrm{ } \wedge(d_{A}a+a\wedge a)=0, \textrm{ }  a''  \textrm{ }  s.t  \textrm{ } (\ref{equation 3}) \}. \nonumber\end{equation}

By the isomorphism sending unitary connections to $(0,1)$ connections,
we identify the above space with an open subset of $Q_{A}$ by implicit function theorem ($\epsilon \ll 1$), where
\begin{equation}   Q_{A}=\{a''\in\Omega^{0,1}(X,EndE)_{k} \textrm{ }\big{|}\textrm{ }\|a''\|_{k} < \epsilon'',\textrm{ }
a'' \textrm{ } s.t \quad (\ref{equation 3}), \textrm{} (\ref{equation 5})  \}. \nonumber  \end{equation}
Because $F^{0,2}_{+}=0$ is not affected by the map
$d_{A}+a\longmapsto \overline{\partial}_{A}+a'' $,
we get isomorphic analytic subspaces after adding it
to both of the two ambient spaces.
\end{proof}
Now we establish the isomorphism between $Q_{A}$ and $H^{0,1}(X,EndE)$.
\begin{lemma}\label{QA iso to H1}
The harmonic projection map
\begin{equation}\mathbb{H}^{0,1}: Q_{A}\rightarrow H^{0,1}(X,EndE)   \nonumber \end{equation}
is a local analytic isomorphism if $\epsilon''$ is small.
\end{lemma}
\begin{proof}
Define
\begin{equation}q: \Omega^{0,1}(EndE)_{k}\rightarrow H^{0,1}(EndE)\oplus {\overline{\partial}_{A}^{*}\Omega^{0,1}(EndE)}_{k}\oplus
{\overline{\partial}_{A}^{*}\Omega^{0,2}(EndE)}_{k-1}  \nonumber \end{equation}
such that
\begin{equation}q(a'')=\bigg(\mathbb{H}(a''),\overline{\partial}_{A}^{*}a''-\frac{i}{2}\wedge(a'\wedge a''+a''\wedge a'),
\overline{\partial}_{A}^{*}\big(\overline{\partial}_{A}a''+P_{\overline{\partial}_{A}}(a''\wedge a'')
+*_{4}P_{\overline{\partial}_{A}^{*}}(a''\wedge a'')\big)\bigg).  \nonumber \end{equation}
This is well defined because
\begin{eqnarray*}\wedge(\varphi)&=&\wedge(\mathbb{H}(\varphi))+\wedge(\overline{\partial}_{A}\overline{\partial}_{A}^{*}G\varphi)+
\wedge(\overline{\partial}_{A}^{*}\overline{\partial}_{A}G\varphi)  \\
&=& \wedge(\mathbb{H}(\varphi))+ \overline{\partial}_{A}\wedge(\overline{\partial}_{A}^{*}G\varphi)+
\overline{\partial}_{A}^{*}(\wedge\overline{\partial}_{A}G\varphi) \\
&=& \wedge(\mathbb{H}(\varphi))+0+\overline{\partial}_{A}^{*}(\wedge\overline{\partial}_{A}G\varphi),
\nonumber\end{eqnarray*}
where $\varphi=a'\wedge a''+a''\wedge a'\in\Omega^{1,1}(X,End_{0}E)$ and we have used the fact that $\wedge$ vanishes on a one form.  \\
Meanwhile:
\begin{equation} \overline{\partial}_{A}\wedge\mathbb{H}(\varphi)=\wedge\overline{\partial}_{A}\mathbb{H}(\varphi)\pm
\overline{\partial}_{A}^{*}\mathbb{H}(\varphi)=0.
\nonumber\end{equation}
And similarly:
\begin{equation} \overline{\partial}_{A}^{*}\wedge\mathbb{H}(\varphi)=0.\nonumber\end{equation}
Thus $\wedge\mathbb{H}(\varphi)\in H^{0}(End_{0}E)=0$ by the simpleness of
$(E,\overline{\partial}_{A})$.
Hence image of the map $q$ is in the target space.

Now we take the differentiation of $q$ at $0$,
\begin{equation}dq_{0}(v)=(\mathbb{H}(v),\overline{\partial}_{A}^{*}v,\overline{\partial}_{A}^{*}\overline{\partial}_{A}v ), \nonumber\end{equation}
which is a diffeomorphism whose inverse is given by
\begin{equation}dq_{0}^{-1}(u_{0},u_{1},u_{2})=u_{0}+G\overline{\partial}_{A}u_{1}+Gu_{2}. \nonumber \end{equation}
By the implicit function theorem, $q$ is a local analytic isomorphism around the origin and $Q_{A}=q^{-1}(H^{0,1}(EndE)\times0\times0)$.
\end{proof}


By Lemma \ref{identify SA MA}, we are reduced to study the space $\mathcal{M}_{A}^{+}$ (we may need to shrink it by requiring $\epsilon''$ to be smaller). In fact, it is enough to study the following space:
\begin{equation}\mathcal{M}_{A}=\{a''\in\Omega^{0,1}(X,EndE)_{k} \textrm{ }\big{|}\textrm{ }\| a''\|_{k} < \epsilon'',\textrm{ }
\mathbb{H}^{0,2}(a''\wedge a'')=0 ,\textrm{} a'' \textrm{ }
s.t \quad (\ref{equation 3}), \textrm{} (\ref{equation 5}) \}.   \nonumber \end{equation}
By the above lemma, $\mathcal{M}_{A}\hookrightarrow Q_{A}$ is an analytic subspace of the finite dimensional smooth manifold $Q_{A}$. \\

We will show that $\mathcal{M}_{A}$ is isomorphic to an analytic neighbourhood of $\overline{\partial}_{A}$ in $\mathcal{M}_{c}^{o}$ with analytic topology. \\

To achieve this, we first introduce some notations. Denote
\begin{equation}P:\Omega^{0,1}(X,EndE)_{k}\rightarrow \Omega^{0,2}(X,EndE)_{k-1} \nonumber \end{equation}
to be the local analytic map defined by
\begin{equation}P(a'')=\overline{\partial}_{A}a''+a''\wedge a'' \nonumber \end{equation}
and
define an analytic map $\lambda: Q_{A}\rightarrow Q_{A}\times \Omega^{0,2}(X,EndE)_{k-1}$ by
\begin{equation}\lambda(a'')=(a'',P(a'')). \nonumber \end{equation}
Note that
\begin{equation}\label{QA def}  Q_{A}\cap P^{-1}(0)=\{a'' \textrm{ }\big{|}\textrm{ }\| a''\|_{k} < \epsilon'',\textrm{ }
F^{0,2}(\overline{\partial}_{A}+a'')=0 ,\textrm{}  a'' \textrm{ }
s.t \textrm{} (\ref{equation 5}) \},    \end{equation}
which gives a neighbourhood of $\overline{\partial}_{A}$ in $\mathcal{M}_{c}^{o}$ and
$Q_{A}\cap P^{-1}(0)\hookrightarrow \mathcal{M}_{A} $ as a closed analytic subspace.
Now we want to show they are actually the same analytic space
possibly with non-reduced structures, i.e. $Q_{A}\cap P^{-1}(0)=\mathcal{M}_{A}$. This is enough to set up
the relation of analytic structures between $\mathcal{M}_{c}^{DT_{4}}$ and $\mathcal{M}_{c}^{o}$. \\
The image of the above map $\lambda$ satisfies
\begin{lemma}\label{image of lamda}
\begin{equation}Im(\lambda)\subseteq F, \nonumber \end{equation}
where
\begin{equation}  F=\left\{ \begin{array}{lll}
   (a'',\theta)\in Q_{A}\times\Omega^{0,2}_{k-1}\textrm{ }\big{|} & \overline{\partial}_{A}^{*}\theta=
   \overline{\partial}_{A}^{*}(*_{4}\overline{\partial}_{A}^{*}G(a''\wedge \theta)), &
   \overline{\partial}_{A}^{*}(\overline{\partial}_{A}\theta+a''\wedge\theta)=0  
\nonumber\end{array}\right\}.\end{equation}
\end{lemma}
\begin{proof}
Let
\begin{eqnarray*}\theta &=&\overline{\partial}_{A}a''+a''\wedge a'' \\
&=&  \overline{\partial}_{A}a''+P_{\overline{\partial}_{A}}(a''\wedge a'')+\mathbb{H}(a''\wedge a'')
+P_{\overline{\partial}_{A}^{*}}(a''\wedge a'').
\end{eqnarray*}
By definition
\begin{equation}a''\in Q_{A}\Rightarrow \overline{\partial}_{A}\alpha+P_{\overline{\partial}_{A}}(\alpha\wedge\alpha)
+*_{4}P_{\overline{\partial}_{A}^{*}}(\alpha\wedge\alpha)=0.
\nonumber \end{equation}
Hence
\begin{eqnarray}\label{bianchi 0}\theta-\mathbb{H}(a''\wedge a'')&=&P_{\overline{\partial}_{A}^{*}}(a''\wedge a'')-*_{4}P_{\overline{\partial}_{A}^{*}}(a''\wedge a'') \\
&=& \overline{\partial}_{A}^{*}\overline{\partial}_{A}G(a''\wedge a'')-*_{4}\overline{\partial}_{A}^{*}\overline{\partial}_{A}G(a''\wedge a'').
\end{eqnarray}
Taking $\overline{\partial}_{A}$ to both sides of $\theta=\overline{\partial}_{A}a''+a''\wedge a''$, we get
\begin{equation}\overline{\partial}_{A}\theta=\overline{\partial}_{A}(a''\wedge a'').\nonumber \end{equation}
Combined with the Bianchi identity $\overline{\partial}_{A}\theta+a''\wedge\theta=0$, we have
\begin{equation}\label{bianchi 1}\overline{\partial}_{A}(a''\wedge a'')=-a''\wedge\theta. \end{equation}
Using (\ref{bianchi 0}) and (\ref{bianchi 1}), we have
\begin{equation}\theta-\mathbb{H}^{0,2}(a''\wedge a'')=
\overline{\partial}_{A}^{*}\overline{\partial}_{A}G(a''\wedge a'')+*_{4}\overline{\partial}_{A}^{*}G(a''\wedge \theta). \nonumber \end{equation}
After taking $\overline{\partial}_{A}^{*}$, we finally get
\begin{equation}\overline{\partial}_{A}^{*}\theta=\overline{\partial}_{A}^{*}\big(*_{4}\overline{\partial}_{A}^{*}G(a''\wedge\theta)\big).
\nonumber \end{equation}
\end{proof}
Now we will show that $F$ defined above is actually a finite dimensional smooth manifold.
\begin{lemma}\label{F iso to H1 H2}
The harmonic projection map
\begin{equation}(\mathbb{H}^{0,1}\times \mathbb{H}^{0,2}): F\rightarrow H^{0,1}(X,EndE)\times H^{0,2}(X,EndE) \nonumber \end{equation}
is a local analytic isomorphism.
\end{lemma}
\begin{proof}
Now we define a map $f$
\begin{equation}f:Q_{A}\times\Omega^{0,2}(X,EndE)_{k-1}\rightarrow Q_{A}\times\Omega^{0,2}(X,EndE)_{k-3} \nonumber \end{equation}
as follows
\begin{equation}f(a'',\theta)=\bigg(a'',\mathbb{H}^{0,2}(\theta)+\overline{\partial}_{A}\overline{\partial}_{A}^{*}\big(\theta-*_{4}
\overline{\partial}_{A}^{*}G(a''\wedge\theta)\big)+\overline{\partial}_{A}^{*}(\overline{\partial}_{A}\theta+a''\wedge\theta)\bigg).
\nonumber \end{equation}
It is easy to check that $f$ is a local analytic isomorphism by using implicit function theorem
\begin{equation}df_{0,0}(v_{1},v_{2})=(v_{1},\mathbb{H}^{0,2}v_{2}+\overline{\partial}_{A}\overline{\partial}_{A}^{*}v_{2}
+\overline{\partial}_{A}^{*}\overline{\partial}_{A}v_{2}) \nonumber\end{equation}
whose inverse is given by
\begin{equation}df_{0,0}^{-1}(u_{1},u_{2})=(u_{1},\mathbb{H}^{0,2}u_{2}+Gu_{2}). \nonumber \end{equation}
Hence $F=f^{-1}(Q_{A}\times H^{0,2}(X,EndE))$ and the projection map
\begin{equation}(\mathbb{H}^{0,1}\times \mathbb{H}^{0,2}) : F\rightarrow H^{0,1}(X,EndE)\times H^{0,2}(X,EndE)\nonumber\end{equation}
gives a local chart of $F$.
\end{proof}
\begin{lemma}\label{vanishing of prW2}
If $\mathbb{H}^{0,2}(\theta)=0$ , $(a'',\theta)\in F$ and $\|a''\|_{k}\ll1$, then $\theta=0$. Here
\begin{equation}\mathbb{H}^{0,2}: \Omega^{0,2}(X,EndE)_{k-1}\rightarrow H^{0,2}(X,EndE)\nonumber\end{equation}
is the harmonic projection map.
\end{lemma}
\begin{proof}
By the Hodge decomposition and $(a'',\theta)\in F$,
\begin{eqnarray*}\theta&=&\mathbb{H}^{0,2}(\theta)+\overline{\partial}_{A}^{*}\overline{\partial}_{A}G\theta+
\overline{\partial}_{A}\overline{\partial}_{A}^{*}G\theta \\&=&-G\overline{\partial}_{A}^{*}(a''\wedge\theta)+
G\overline{\partial}_{A}\overline{\partial}_{A}^{*}(*_{4}\overline{\partial}_{A}^{*}G(a''\wedge\theta)),\nonumber\end{eqnarray*}
then
\begin{equation}
\|\theta\|_{k-1}\leq C_{1}\|a''\|_{k}\|\theta\|_{k-1}+C_{2}\| a''\|_{k}\|\theta\|_{k-1}
=C\| a''\|_{k}\|\theta\|_{k-1}.\nonumber\end{equation}
$C$ is a constant independent of $a'',\theta$. Hence we can get $\theta=0$ if $\|a''\|_{k}\ll1 $ .
\end{proof}
\begin{corollary}
The following three analytic spaces are set theoretically identical
\begin{equation}\mathcal{M}_{A}=\{a''\in\Omega^{0,1}(X,EndE)_{k} \textrm{ }\big{|}\textrm{ }\|a''\|_{k} < \epsilon'',\textrm{ } \mathbb{H}^{0,2}(a''\wedge a'')=0 ,\textrm{} a'' \textrm{ } s.t \quad (\ref{equation 3}), \textrm{} (\ref{equation 5}) \}, \nonumber \end{equation}
\begin{equation}  Q_{A}\cap P^{-1}(0)=\{a'' \textrm{ }\big{|}\textrm{ }\|a''\|_{k} < \epsilon'',\textrm{ }
F^{0,2}(\overline{\partial}_{A}+a'')=0 ,\textrm{} a'' \textrm{ } s.t \textrm{} (\ref{equation 5}) \}, \nonumber \end{equation}
\begin{equation}Q_{A}\cap P^{-1}(0)=\left\{ \begin{array}{lll}& \|a''\|_{k} < \epsilon'', \textrm{ } a'' \textrm{ } s.t \textrm{} (\ref{equation 5}) \\ a'' \textrm{ }\Bigg{|} & \overline{\partial}_{A}a''+P_{\overline{\partial}_{A}}(a''\wedge a'')=0   \\ & \mathbb{H}^{0,2}(a''\wedge a'')=0 \nonumber\end{array}\right\}.\end{equation}
We use the same notation for the second and the third spaces because they are isomorphic as analytic spaces by the standard Kuranishi theory.
\end{corollary}
\begin{proof}
We only need to show that $Q_{A}\cap P^{-1}(0)$ contains
the first set $\mathcal{M}_{A}$ as its subset i.e. we need to show
\begin{equation}\forall\textrm{ }\overline{\partial}_{A}+a''\in \mathcal{M}_{A} \Rightarrow F^{0,2}(\overline{\partial}_{A}+a'')=0. \nonumber \end{equation}
Since $\mathcal{M}_{A}$ is a subset of $Q_{A}$, we can apply Lemma \ref{image of lamda} to its image under the map $\lambda$.
Combined with Lemma \ref{vanishing of prW2}, we finish the proof.
\end{proof}
Furthermore, we can identify the above three spaces as analytic spaces possibly with non-reduced structures.
Let us first recall a lemma from \cite{m}
\begin{lemma}\label{miyajima lemma}\cite{m}. Let $E$, $G$ be Banach spaces with direct sum decomposition $E=F_{1}+F_{2}$. If a local analytic map
\begin{equation}h: E\rightarrow G
\nonumber\end{equation}
vanishes identically on $F_{2}$, then there exists a local analytic map
\begin{equation}f: E \rightarrow L(F_{2},G),
\nonumber \end{equation}
such that $h(t,s)=<f(t,s),s>$.
\end{lemma}
\begin{proposition}\label{QA intesect P=0 equals NA}
We have the following identification
\begin{equation}\label{Kuranishi theorem} Q_{A}\cap P^{-1}(0)=\mathcal{M}_{A} \nonumber \end{equation}
as analytic spaces possibly with non-reduced structures.
\end{proposition}
\begin{proof}
Consider the local analytic map $\lambda:Q_{A}\rightarrow F$ by
\begin{equation}\lambda(\alpha)=(\alpha,P(\alpha)). \nonumber \end{equation}
By Lemma \ref{image of lamda}, the map is well defined. By Lemma \ref{QA iso to H1} and \ref{F iso to H1 H2},
we have the following commutative diagram
\begin{equation}
\xymatrix{\ar @{} [dr] |{} Q_{A} \ar@/^3pc/[rr]^{P(\alpha)=\overline{\partial}_{A}\alpha+\alpha\wedge \alpha}
\ar[d]^{\mathbb{H}^{0,1}}_{\wr\mid} \ar[r]^{\lambda(\alpha)=(\alpha,P(\alpha))} & F \ar[d]^{\mathbb{H}^{0,1}\times \mathbb{H}^{0,2}}_{\wr\mid} \ar[r]^{\pi_{2}(\alpha,\theta)=\theta}
& \Omega^{0,2}(X,EndE)_{k-1} \\
H^{0,1}(X,EndE) \ar[r]^{\lambda^{'}\quad \quad \quad \quad}
\ar@/_5pc/[urr]_{P^{'}(t)=\pi_{2}^{'}\circ \lambda^{'}(t)=\pi_{2}^{'}(t,\mathbb{H}^{0,2}\circ P(\alpha(t)))}
& H^{0,1}(X,EndE)\times H^{0,2}(X,EndE) \ar[ur]_{\quad \quad\pi_{2}^{'}}  }.
\nonumber \end{equation}

With respect to the local charts $(Q_{A},\mathbb{H}^{0,1})$ and $(F,(\mathbb{H}^{0,1}\times \mathbb{H}^{0,2}))$, $\lambda$ is expressed
by $\lambda^{'}=\big(t,\mathbb{H}^{0,2}\circ P(\alpha(t))\big)$, $\pi_{2}$ is expressed by $\pi_{2}^{'}$ and $P$ is expressed by $P^{'}$
where $t\in H^{0,1}(X,EndE)$ and $\alpha(t)=(\mathbb{H}^{0,1})^{-1}(t)$.

By Lemma \ref{vanishing of prW2}, we know that $\mathbb{H}^{0,2}\circ P(\alpha(t))=0\Rightarrow P(\alpha(t))=0$ which obviously implies $P^{'}(t)=0$.

Thus we can apply Lemma \ref{miyajima lemma} to our case, we then get
\begin{eqnarray*}P^{'}(t)&=&\pi_{2}^{'}\big(t,\mathbb{H}^{0,2}\circ P(\alpha(t))\big) \\
&=&<\eta\big(t,\mathbb{H}^{0,2}\circ P(\alpha(t))\big),\textrm{ }\mathbb{H}^{0,2}\circ P(\alpha(t))>
\end{eqnarray*}
for some local analytic map $\eta: H^{0,1}\times H^{0,2} \rightarrow L(H^{0,2},\Omega^{0,2}_{k-1})$ where $L(H^{0,2},\Omega^{0,2}_{k-1})$ is the
space of analytic maps between the Banach spaces $H^{0,2}$ and $\Omega^{0,2}_{k-1}$.

Thus we can see that the two ideals
$\mathcal{I}_{P}\subseteq \mathcal{I}_{\mathbb{H}^{0,2}\circ P}$ when we view $(Q_{A},\mathcal{O}_{Q_{A}})$ as a ringed analytic space.
We then get $\mathcal{I}_{P}=\mathcal{I}_{\mathbb{H}^{0,2}\circ P}$ because $\mathcal{I}_{P}\supseteq \mathcal{I}_{\mathbb{H}^{0,2}\circ P}$ is obviously true.

By writing out the equations we finish the proof.
\end{proof}
Thus we have proved that $\mathcal{M}_{A}$ is isomorphic to an analytic neighbourhood of $\overline{\partial}_{A}$ in $\mathcal{M}_{c}^{o}$ with analytic topology. \\
${}$ \\
\textbf{Conclusion}:
\begin{theorem}\label{Kuranishi str of cpx ASD thm} ${}$ \\
The local Kuranishi model of $\mathcal{M}_{c}^{DT_{4}}$ near $d_{A}$ with $F^{0,2}_{A}=0$ can be described as
\begin{equation}\tilde{\tilde{\kappa}}_{+}=\pi_{+}(\tilde{\tilde{\kappa}}):  H^{0,1}(X,EndE)\cap B_{\epsilon}\rightarrow H^{0,2}_{+}(X,EndE), \nonumber \end{equation}
where $\tilde{\tilde{\kappa}}$ is a Kuranishi map for $\mathcal{M}_{c}^{o}$ near $\overline{\partial}_{A}$ determined by $DT_{4}$ equations
\big(see appendix (\ref{kappa double tilta})\big). $B_{\epsilon}$ is a small open ball containing the origin of the deformation space.
The map $\pi_{+}$ is the projection map to the self-dual subspace of the obstruction space.
\end{theorem}
\begin{proof}
By definition
\begin{equation}\mathcal{M}_{A}={\tilde{\tilde{\kappa}}}^{-1}(0), \nonumber \end{equation}
where $\tilde{\tilde{\kappa}}: H^{0,1}(X,EndE)\cap B_{\epsilon}\rightarrow H^{0,2}(X,EndE) $ is define to be
\begin{equation}\tilde{\tilde{\kappa}}(\alpha)=\mathbb{H}^{0,2}\big(q^{-1}(\alpha)\wedge q^{-1}(\alpha)\big)
\nonumber\end{equation}
and $q$ is defined in appendix, see (\ref{g double tilta}). By Proposition \ref{QA intesect P=0 equals NA}, $\tilde{\tilde{\kappa}}$ is a Kuranishi map for $\mathcal{M}_{c}^{o}$.
Composing with $\pi_{+}$, we get $\mathcal{M}_{A}^{+}=\big(\pi_{+}\tilde{\tilde{\kappa}}\big)^{-1}(0) $.
\end{proof}
\begin{remark} ${}$ \\
1. Under the assumption $\mathcal{M}_{c}^{o}\neq\emptyset$, we have a bijective map
\begin{equation}\mathcal{M}_{c}^{o}\rightarrow \mathcal{M}_{c}^{DT_{4}}.  \nonumber \end{equation}
By Proposition \ref{QA intesect P=0 equals NA} and the above Kuranishi theorem on $\mathcal{M}_{c}^{DT_{4}}$, we know the bijective map is
actually a closed imbedding as analytic space possibly with non-reduced structures. \\
2. For notation simplicity, we will always restrict to a small neighbourhood of the origin in $Ext^{1}(\mathcal{F},\mathcal{F})$
when we talk about Kuranishi maps
\begin{equation}\kappa: Ext^{1}(\mathcal{F},\mathcal{F})\rightarrow Ext^{2}(\mathcal{F},\mathcal{F}) \nonumber \end{equation}
for any coherent sheaf $\mathcal{F}$ and omit the small ball $B_{\epsilon}$ in the notation from now on.
\end{remark}

\newpage
\section{Compactification of $DT_{4}$ moduli spaces}
Now we come to the issue of compactification of the $DT_{4}$ moduli space (\ref{complex ASD equation}).
As Uhlenbeck, generally Tian \cite{t} have shown,
we need to consider connections with singularities supported on codimension 4 subspaces when we compactify the moduli space of holomorphic HYM connections. This becomes very difficult when the real dimension of the underlying manifold is bigger than 4.
Even if one could compactify it, as Tian showed in his paper,
one was still lack of understanding of the local Kuranishi structures of the compactified moduli space.

Our attempted approach here is the algebro-geometric compactification using moduli space of semi-stable sheaves.

\subsection{Stable bundles compactification of $DT_{4}$ moduli spaces}
In this subsection, under the assumption that $\mathcal{M}_{c}\neq\emptyset$ consists of slope-stable bundles only, we prove that $\mathcal{M}_{c}^{DT_{4}}$ is compact.

Given a connection on $E$ with curvature $F$. By Chern-Weil theory, we have
\begin{equation} Tr(F^{2})=-8\pi^{2}ch_{2}(E). \nonumber \end{equation}
Then
\begin{equation}
-8\pi^{2}\int ch_{2}(E)\wedge\Omega=\int Tr(F^{0,2}\wedge F^{0,2})\wedge\Omega \nonumber\end{equation}
\begin{equation}=\int Tr(F^{0,2}_{+}\wedge F^{0,2}_{+})\wedge\Omega+\int Tr(F^{0,2}_{-}\wedge F^{0,2}_{-})\wedge\Omega\nonumber\end{equation}
\begin{equation}+\int Tr(F^{0,2}_{+}\wedge F^{0,2}_{-})\wedge\Omega+\int Tr(F^{0,2}_{-}\wedge F^{0,2}_{+})\wedge\Omega\nonumber\end{equation}
\begin{equation}\label{plus norm equal minus norm}
=\int\mid F^{0,2}_{+}\mid^{2}\wedge\Omega\wedge\overline{\Omega} -\int\mid
F^{0,2}_{-}\mid^{2}\wedge\Omega\wedge\overline{\Omega}+\int\sqrt{-1}\chi\wedge\Omega\wedge\overline{\Omega},\end{equation}
where $\chi$ is some real valued function.
\begin{lemma}(Lewis) \cite{lewis}\label{C2 condition}
If $ch_{2}(E)\in H^{2,2}(X,\mathbb{C})$ or has no component of type $(4,0)$,
then $F^{0,2}_{+}=0$ implies $F^{0,2}=0$.
\end{lemma}
\begin{proof}
Note that $\Omega$ is $(4,0)$ form and $\chi$ is a real valued function.
\end{proof}
\begin{corollary} \label{cptness by stable bdl}
If $\mathcal{M}_{c}^{o}=\mathcal{M}_{c}=\overline{\mathcal{M}}_{c}\neq\emptyset$,
then $\mathcal{M}_{c}^{DT_{4}}$ is compact.
\end{corollary}
\begin{proof}
By the assumptions and Lemma \ref{C2 condition}.
\end{proof}
From the viewpoint of local Kuranishi models (i.e. Theorem \ref{Kuranishi str of cpx ASD thm}), Lemma \ref{C2 condition} says
\begin{equation}\pi_{+}\tilde{\tilde{\kappa}}=0 \Rightarrow \tilde{\tilde{\kappa}}=0 .\nonumber\end{equation}

\begin{proposition} Given a local holomorphic map
\begin{equation}\kappa: H^{0,1}(X,EndE)\rightarrow H^{0,2}(X,EndE) \nonumber\end{equation}
such that $\kappa_{+}=0 \Rightarrow \kappa=0$ and $\kappa(0)=0$, where
\begin{equation}\kappa_{+}=\pi_{+}\circ\kappa : H^{0,1}(X,EndE)\rightarrow H^{0,2}_{+}(X,EndE). \nonumber\end{equation}
Then image of $\kappa$ can not be a neighbourhood of the origin.
\end{proposition}
\begin{proof}
By assumptions, $\kappa(U(0))\cap H^{0,2}_{+}(X,EndE)=\{0\}$, where $U(0)$ is a small neighbourhood of the origin in $H^{0,1}(X,EndE)$.
\end{proof}
\begin{remark}
In fact, we can find examples of local analytic maps such that $rank(\kappa)=ext^{2}(E,E)-1$
(but so far we do not know any example of $\mathcal{M}_{c}$ whose local analytic structure is of this type).

Because $\mathcal{M}_{c}^{o}=\mathcal{M}_{c}^{DT_{4}}$ as sets, the dimension of
$\mathcal{M}_{c}^{DT_{4}}$ may be less than its virtual dimension.
Hence it is possible that some local parts of $\mathcal{M}_{c}^{DT_{4}}$ do not contribute to the $DT_{4}$
invariants which will be defined later.
\end{remark}

\subsection{Attempted general compactification of $DT_{4}$ moduli spaces}
In this subsection, we propose an attempted approach to the general compactification of $\mathcal{M}_{c}^{DT_{4}}$. Under the gluing assumptions, we define the generalized
$DT_{4}$ moduli space $\overline{\mathcal{M}}_{c}^{DT_{4}}$ as the gluing of local models.

In some cases, we can get rid of the gluing assumption and prove the existence of $\overline{\mathcal{M}}_{c}^{DT_{4}}$ directly.\\

Recall, if we assume
$\mathcal{M}_{c}^{o}\neq\emptyset$, we have a bijective map
\begin{equation}\mathcal{M}_{c}^{o}\rightarrow \mathcal{M}_{c}^{DT_{4}},  \nonumber \end{equation}
which is a closed imbedding as analytic space possibly with non-reduced structures.
The idea of general compactification is to extend the above map to a bijective map
\begin{equation}\mathcal{M}_{c}\rightarrow \overline{\mathcal{M}}_{c}^{DT_{4}},  \nonumber \end{equation}
where $\overline{\mathcal{M}}_{c}^{DT_{4}}$ is realized as the gluing of local models. \\

Furthermore, the local model of $\overline{\mathcal{M}}_{c}^{DT_{4}}$ near a stable sheaf $\mathcal{F}$ should be defined
as
\begin{equation}\kappa_{+}=\pi_{+}\circ\kappa:Ext^{1}(\mathcal{F},\mathcal{F})\rightarrow Ext^{2}_{+}(\mathcal{F},\mathcal{F}),
\nonumber \end{equation}
where $\kappa$ is a Kuranishi map of $\mathcal{M}_{c}$ at $\mathcal{F}$ and
$\pi_{+}$ is the projection map as in Corollary \ref{cutting for sheaves}.

However, the Kuranishi map $\kappa$ is unique only up to local re-parametrizations of $Ext^{1}(\mathcal{F},\mathcal{F})$.
Meanwhile, the $*_{4}$ is a real operator and if we use different re-parametrization, the resulting models may be different in general, i.e.
\begin{equation}(\pi_{+}\circ\kappa_{1})^{-1}(0)\ncong (\pi_{+}\circ\kappa_{2})^{-1}(0)  \nonumber \end{equation}
for different $\kappa_{i}$, $i=1, 2$.

For the purpose of gluing, we need to pick a coherent choice of local Kuranishi models for $\mathcal{M}_{c}$. In the case when
$\mathcal{M}_{c}=\mathcal{M}_{c}^{o}$, the moment map equation in $DT_{4}$ equations (\ref{complex ASD equation}) gives such a choice.
In general, we need a similar moment map equation for $\mathcal{M}_{c}$. This is achieved by a quiver representation of $\mathcal{M}_{c}$
due to \cite{bchr}. We then propose a candidate local model near each $\mathcal{F}\in \mathcal{M}_{c}$ based on their work. Since we do not know how to glue the local models at the moment, we leave the construction to the appendix. The notations and results of \cite{bchr} will also be recalled in the appendix. \\
${}$ \\
We will always make the following assumptions when we talk about
the $DT_{4}$ invariants in the case when $\mathcal{M}_{c}^{o}\neq \mathcal{M}_{c}$.
\begin{assumption}\label{assumption on gluing} ${}$\\
We assume there exists a real analytic space $\overline{\mathcal{M}}^{DT_{4}}_{c}$ and a bijective map
\begin{equation}\mathcal{M}_{c}\rightarrow\overline{\mathcal{M}}^{DT_{4}}_{c}  \nonumber \end{equation}
such that near each closed point of $\mathcal{M}_{c}$ say $\mathcal{F}$, the
local structure of $\overline{\mathcal{M}}^{DT_{4}}_{c}$ is of form $\kappa_{+}^{-1}(0)$, where
\begin{equation}\kappa_{+}=\pi_{+}\circ\kappa:Ext^{1}(\mathcal{F},\mathcal{F})\rightarrow Ext^{2}_{+}(\mathcal{F},\mathcal{F}).
\nonumber \end{equation}
$\kappa$ is a Kuranishi map at $\mathcal{F}$
and $Ext^{2}_{+}(\mathcal{F},\mathcal{F})$ is a half dimensional real subspace of $Ext^{2}(\mathcal{F},\mathcal{F})$ on which the Serre duality quadratic form is positive definite.
\end{assumption}
\begin{definition}\label{generalized DT4}
Under Assumption \ref{assumption on gluing}, we get a compact real analytic space $\overline{\mathcal{M}}^{DT_{4}}_{c}$. We call it the
\textbf{generalized $DT_{4}$ moduli space}.
\end{definition}
%
In the case when $\mathcal{M}_{c}$ is smooth, we have the following obvious gluing result.
\begin{proposition}\label{gene DT4 if Mc smooth}
If the Gieseker moduli space $\mathcal{M}_{c}$ is smooth, the generalized $DT_{4}$ moduli space exists and
$\overline{\mathcal{M}}^{DT_{4}}_{c}=\mathcal{M}_{c}$ as real analytic spaces.
\end{proposition}
\begin{proof}
By the assumption, all Kuranishi maps are zero. The conclusion is obvious from the definitions.
\end{proof}
There is another interesting case when we can get $\overline{\mathcal{M}}^{DT_{4}}_{c}$ without the gluing assumption.
\begin{proposition}\label{ob=v+v*}
Assume for any closed point $\mathcal{F}\in \mathcal{M}_{c}$, there is a splitting of obstruction space
\begin{equation}Ext^{2}(\mathcal{F},\mathcal{F})=V_{\mathcal{F}}\oplus V_{\mathcal{F}}^{*}  \nonumber \end{equation}
such that $V_{\mathcal{F}}$ is its maximal isotropic subspace with respect to the Serre duality pairing and the image of a Kuranishi map $\kappa$ at $\mathcal{F}$ satisfies
\begin{equation}\emph{Image}(\kappa)\subseteq V_{\mathcal{F}}.  \nonumber \end{equation}
Then the generalized $DT_{4}$ moduli space exists and $\overline{\mathcal{M}}^{DT_{4}}_{c}=\mathcal{M}_{c}$ as real analytic spaces.
\end{proposition}
\begin{proof}
Pick a Hermitian metric $h$ on $V_{\mathcal{F}}$ which induces a Hermitian metric on $V_{\mathcal{F}}^{*}$. We abuse the notation $h$ for the direct sum Hermitian metric on $Ext^{2}(\mathcal{F},\mathcal{F})=V_{\mathcal{F}}\oplus V_{\mathcal{F}}^{*}$. Define $*_{4}: Ext^{2}(\mathcal{F},\mathcal{F})\rightarrow Ext^{2}(\mathcal{F},\mathcal{F})$ such that $Q_{Serre}(\alpha,*_{4}\beta)=h(\alpha,\beta)$, where $Q_{Serre}$ denotes the Serre duality pairing. Then for $\kappa(\alpha)\in V_{\mathcal{F}}$, we have $*_{4}\big(\kappa(\alpha)\big)\in V_{\mathcal{F}}^{*}$ which implies $\kappa_{+}=0\Rightarrow \kappa=0$ by the assumption $\emph{Image}(\kappa)\subseteq V_{\mathcal{F}}$.
\end{proof}
\begin{remark} We will see the above conditions are satisfied for compactly supported sheaves on certain local $CY_{4}$ manifolds.
\end{remark}

\newpage
\section{Virtual cycle construction}
For moduli problems in algebraic geometry, we can use GIT construction to form the moduli space.
If one wants to define invariants of a moduli problem, we need to define the fundamental class of the moduli space.
But in general the moduli space is very singular and may have a lot of different components with different dimensions,
how can one get a correct cycle (deformation invariant) to represent the fundamental class of the moduli space.
The answer came from the original idea of Fulton-MacPherson's localized top Chern class \cite{fulton},
then it is generalized to Fredholm Banach bundles over Banach manifolds by Brussee \cite{brussee} and developed in moduli problems by Li-Tian \cite{lt1}, Behrend-Fantechi \cite{bf} in full generality.

The idea is to imbed the moduli space into a smooth ambient space (maybe infinite dimension) as a zero loci of some section of some natural bundle.
Then we generically perturb the section to define the Euler class of that bundle.
The correct cycle one may want to take is the Poincar\'{e} dual of that Euler class.

However in general, the moduli space can only be imbedded in finite dimensional smooth space locally.
Then one may allow the ambient space to be a infinite dimension one (usually a canonical one with nice local structures),
this is the idea of \cite{lt2}. If one does not want to encounter infinite dimensional stuff,
one way of achieving this in algebro-geometric setting is to study the deformation-obstruction theory of the moduli problem.
If the obstruction theory is perfect in sense of Li-Tian \cite{lt1}.
Then one can construct a global cone of equal dimension over the moduli sitting inside the locally free obstruction bundle.
Intersecting the cone with the zero section of the obstruction bundle gives the virtual fundamental class of the moduli space.
The equivalence of the above two approaches above was proved in \cite{lt3}.

Meanwhile, K. Behrend and B. Fantachi \cite{bf} studied the cotangent complex of the moduli space directly.
They constructed a zero dimensional cone stack called intrinsic normal cone inside the intrinsic normal sheaf of the cotangent complex.
Under perfectness assumption too, they defined the virtual fundamental class of the moduli space which turned out to be equivalent
to Li-Tian's construction \cite{kim}.


\subsection{Virtual cycle construction of $DT_{4}$ moduli spaces }

Now let us define the virtual fundamental class of $\mathcal{M}_{c}^{DT_{4}}$.

We have
\begin{equation}\label{Fredholm bundle}
\begin{array}{lll}
      & E =& \mathcal{A^{*}}\times_{\mathcal{G}^{0}}(\Omega^{0}(X,g_{E})_{k-1}\oplus\Omega^{0,2}_{+}(X,EndE)_{k-1})
      \\  & \quad  & \qquad \downarrow \\ \mathcal{M}^{DT_{4}}_{c} & \hookrightarrow & \mathcal{B}_{1}=\mathcal{A^{*}}/\mathcal{G}^{0},
\end{array}\end{equation}
where $\mathcal{M}^{DT_{4}}_{c}\hookrightarrow \mathcal{B}_{1}$ is the $DT_{4}$ moduli space
defined as the zero loci of section $s=(\wedge F,F^{0,2}_{+})$ of the above Banach bundle.
We assume $\mathcal{M}_{c}=\mathcal{M}_{c}^{o}\neq\emptyset$ which leads to the compactness of $\mathcal{M}_{c}^{DT_{4}}$.\\
${}$ \\
\textbf{Remark}: The orientation of the above Banach bundle may be determined by the choice of orientation on the virtual vector space $-H^{0}(X)+H^{0,1}(X)-H^{0,2}_{+}(X)$ as in the case of Donaldson theory. We will discuss this in detail in a next paper and let us assume the determinant line bundle is orientable and choose an orientation data $o(\mathcal{L})$. Otherwise, all the invariants defined in the thesis will be mod 2. Actually, as we can see for many cases considered in this thesis, the orientation can be obviously obtained and some partial results on the orientability is showed in the appendix. \\

Now we check the Fredholm property of the above Banach bundle.
Take an open cover $\{U_{i}\}$ of $s^{-1}(0)$ in $\mathcal{A}^{*}/\mathcal{G}^{0}$, where
\begin{equation}U_{i}=\{d_{A_{i}}+a \textrm{ } \big{|} \textrm{ }  \|a\|_{k}< \epsilon ,  \textrm{ }  d_{A_{i}}^{*}a=0 \}. \nonumber\end{equation}
Note that
\begin{equation}E\mid_{U_{i}}=U_{i}\times(\Omega^{0}(X,g_{E})_{k-1}\oplus\Omega^{0,2}_{+}(X,EndE)_{k-1}). \nonumber\end{equation}
On the intersection of two charts, we have the following commutative diagram:
\begin{equation}\xymatrix{
  E\mid_{U_{i}} \ar[d]_{\pi_{i}} \ar[r]^{\Phi_{ij}} & E\mid_{U_{j}} \ar[d]^{\pi_{j}} \\
  U_{i} \ar[r]^{\phi_{ij}} & U_{j}   }
\nonumber\end{equation}
the map $\phi_{ij}$ is the gauge transformation on $U_{ij}$, and $\Phi_{ij}$ is adjoint action in the fiber direction.

The section $s$ near $d_{A}$ with $\overline{\partial}_{A}^{2}=0$ is given by
\begin{equation}  \mathcal{G}\curvearrowright \Omega^{1}(X,g_{E})_{k}\rightarrow \Omega^{0}(X,g_{E})_{k-1}\oplus\Omega^{0,2}_{+}(X,EndE)_{k-1},
\nonumber\end{equation}
\begin{equation}a=a^{0,1}+a^{1,0}\mapsto (\wedge F(d_{A}+a^{0,1}+a^{1,0}), F^{0,2}_{+}(\overline{\partial}_{A}+a^{0,1})).
\nonumber \end{equation}
Because we fix the Hermitian metric $h$ on $E$, we identify unitary connections with $(0,1)$ connections
$\Omega^{1}(X,g_{E})_{k}\cong\Omega^{0,1}(X,EndE)_{k} $.
When we fix the unitary gauge, we get a subspace $kerd_{A}^{*}\subseteq\Omega^{1}(X,g_{E})_{k}$.

By the K\"ahler identity, $[\wedge,d_{A}]=i(\overline{\partial}_{A}^{*}-{\partial}_{A}^{*})$. We have
\begin{equation}ker(ds)\mid_{A}\cong H^{0,1}(X,EndE),\nonumber\end{equation}
\begin{eqnarray*}coker(ds)\mid_{A}&=&\frac{\Omega^{0}(X,g_{E})_{k-1}\oplus
\Omega^{0,2}_{+}(X,EndE)_{k-1}}{(i\overline{\partial}_{A}^{*}a^{0,1}-i\partial^{*}a^{1,0})\oplus\overline{\partial}_{A}^{+}a^{0,1}} \\
&=&\frac{\Omega^{0}(X,g_{E})_{k-1}\oplus \Omega^{0,2}_{+}(X,EndE)_{k-1}}{(2i\overline{\partial}_{A}^{*}a^{0,1})\oplus\overline{\partial}_{A}^{+}a^{0,1}} \\
&=& H^{0}(X,g_{E})\oplus H^{0,2}_{+}(X,EndE).\nonumber\end{eqnarray*}
The second equality is by gauge fixing condition $d_{A}^{*}(a)=\overline{\partial}_{A}^{*}a^{0,1}+\partial_{A}^{*}a^{1,0}=0 $. \\
Thus we have proved the Fredholm property of $s$ on $s^{-1}(0)$.

Then by Proposition 14 of \cite{brussee}, the Euler class of the above Fredholm Banach bundle exists, i.e. $[\mathcal{M}^{DT_{4}}_{c}]^{vir}\in H_{r}(\mathcal{B}_{1},\mathbb{Z})$ and we call it the virtual fundamental class of the $DT_{4}$ moduli space. \\
${}$ \\
\textbf{Deformation invariance}.
Now we show the virtual cycle defined above is independent of the choice of \\
$(1)$ the holomorphic top form $\Omega$, \\
$(2)$ the Hermitian metric $h$ on $E$,  \\
$(3)$ the parameter $t$ of any continuous deformation of complex structures of Calabi-Yau four-fold $X_{t}$ under the assumption that
$\mathcal{M}_{c}^{o}=\mathcal{M}_{c}\neq\emptyset$ for all $X_{t}$.

We will show in detail the independence of $(1)$ here, $(2)$ and $(3)$ can be proved similarly.
\begin{remark}Because the factor $\Omega^{0}(X,g_{E})_{k-1}$ is independent of some of the above choices, we will sometimes omit it in the
expression of the Banach bundle for notation simplicity.
\end{remark}
${}$ \\
\textbf{Independence of $\Omega$}:
We fix the Hermitian metric $h$, the complex structure of $X$ and an orientation data $o(\mathcal{L})$. Choose two holomorphic top form $\Omega^{4,0}$,
$\theta_{1}\Omega^{4,0}$ ($\mid\theta_{1}\mid=1$), which corresponds to $*$, $*_{1}$ respectively
(we use the notation $*$ instead of $*_{4}$ for simplicity in this section). Let $\theta_{1}=a+b\sqrt{-1}$,
then $*_{1}=(a+b\sqrt{-1})* $ with $a^{2}+b^{2}=1$. We first assume $b\neq0$. Then there exists a bundle isomorphism
\begin{equation}
\xymatrix{
  \Omega^{0,2}_{+}(X,EndE)_{k-1} \ar[rr]^{f_{1}} \ar[dr]_{\pi}
                &  &    \Omega^{0,2}_{+_{1}}(X,EndE)_{k-1} \ar[dl]^{\pi_{1}}    \\
                & B\triangleq\mathcal{B}_{1}                 }.   \nonumber\end{equation}
The map $f_{1}$ is defined to be fiberwise multiplication by $\frac{1}{2}(a+1+b\sqrt{-1})$ (if $b=0$ $a=-1$, $f$ is
defined as multiplication by $\sqrt{-1}$). It is easy to show that $f_{1}$ is well-defined.

Denote $s_{1}^{'}=f_{1}^{-1}\circ s_{1}$, where $s_{1}=F^{0,2}_{+_{1}}$ is the canonical section of the
Banach bundle when the holomorphic top form is $\theta_{1}\Omega^{4,0}$.
By similar argument as before, we know $s_{1}^{'}: B\rightarrow\Omega^{0,2}_{+}(X,EndE)_{k-1}$ is also a Fredholm Banach bundle when adding back the factor $\Omega^{0}(X,g_{E})_{k-1}$ with the moment map section.

By the functorial property of the Euler class, we are reduced to prove
\begin{equation}e([s_{1}^{'}:B\rightarrow\Omega^{0,2}_{+}(X,EndE)_{k-1}])=e([s :B\rightarrow \Omega^{0,2}_{+}(X,EndE)_{k-1}]), \nonumber \end{equation}
where $s=F^{0,2}_{+}$ is the canonical section of the Banach bundle when the holomorphic top form is $\Omega^{4,0}$.

Now we consider a family of sections of the following Banach bundle
\begin{equation}s_{t}^{'}: B\rightarrow \Omega^{0,2}_{+}(X,EndE)_{k-1}, \nonumber \end{equation}
where
\begin{equation}s_{t}^{'}=f_{t}^{-1}\circ s_{t}=\big(\frac{1}{2}(\sqrt{1-t^{2}b^{2}}+1+t\cdot b\sqrt{-1})\big)^{-1}\cdot s_{t}
\nonumber\end{equation}
and $s_{t}=F^{0,2}_{+_{t}}$ is the canonical section of the Banach bundle when the holomorphic top form is
$\theta_{t}\cdot\Omega^{4,0}$, where $\theta_{t}=(\sqrt{1-t^{2}b^{2}}+t\cdot b\sqrt{-1})$.

We have the following commutative relation
\begin{equation}f_{t}\circ *\circ F^{0,2}=*_{t}\circ f_{t}\circ F^{0,2}, \nonumber \end{equation}
where $*_{t}=\theta_{t}\circ*$. Then,
\begin{equation}s_{t}^{'}=f_{t}^{-1}\circ s_{t}=f_{t}^{-1}\circ \pi_{+_{t}}F^{0,2}=\pi_{+}\circ(f_{t}^{-1}F^{0,2}),\nonumber \end{equation}
which connects $s_{0}^{'}=s$ and $s_{1}^{'}$.

Define
\begin{equation}S: B\times[0,1]\rightarrow \mathcal{A^{*}}\times_{\mathcal{G}^{0}}(\Omega^{0}(X,g_{E})_{k-1}\oplus\Omega^{0,2}_{+}(X,EndE)_{k-1}),
\nonumber \end{equation}
\begin{equation}S(A,t)=(\wedge F(A),s_{t}^{'}), \nonumber \end{equation}
which is an oriented Fredholm Banach bundle of index $v.d_{\mathbb{R}}(\mathcal{M}^{DT_{4}}_{c})+1$ such that $S|_{B\times 0}=s$
and $S|_{B\times 1}=s_{1}^{'}$.
\begin{remark} By topological reason, the above family miss the case when $*_{1}=-*$
which can be remedied by moving $*$ a little bit and then we can cover the whole $S^{1}$ family.
\end{remark}
Thus, we know that the Euler class $e([s: B\rightarrow E])$ is independent of the choice of the holomorphic top form by \cite{brussee}. \\
${}$ \\
\textbf{Independence of $h$}:
As for the case when we change the Hermitian metric $h$, because the space of all Hermitian metrics on the
given topological bundle is connected and contractable, the above argument go through
(actually we do not need the explicit expression of the isomorphism $f_{t}$ as stated above). The only difference is that we also need to identify the base $B$ for different choices of $h$ which is standard.\\
${}$ \\
\textbf{Independence of complex structures}:
For the last case, we fix a Hermitian metric, a continuous deformation of complex structures $J_{t}$ of $X$ and an orientation data $o(\mathcal{L})$ (it does not depend on the parameter $t$). We consider the following Fredholm Banach bundle
\begin{equation}s: B\times [0,1]\rightarrow \mathcal{A^{*}}\times_{\mathcal{G}^{0}}(\Omega^{0}(X,g_{E})_{k-1}\oplus\Omega^{0,2}_{+}(X,EndE)_{k-1}),
\nonumber \end{equation}
\begin{equation}s_{t}=(\wedge F, f_{t}^{-1}\circ F^{0,2}_{+_{t}}),\nonumber \end{equation}
where $*_{t}$ is the $*_{4}$ operator with respect to the holomorphic structure $J_{t}$ and
\begin{equation}f_{t}:\Omega^{0,2}_{+}(X,EndE)_{k-1}\rightarrow \Omega^{0,2}_{+_{t}}(X,EndE)_{k-1}\nonumber \end{equation}
is a Banach bundle isomorphism which commutes with the adjoint action of $\mathcal{G}$.
$f_{t}$ exists because the complex structure only affects the differential forms part of the underling manifold, not the topological bundle,
while the unitary gauge transformations act on the bundle $E$ only.

We have
\begin{equation}f_{t}(* F^{0,2})=*_{t}f_{t}(F^{0,2})\nonumber \end{equation}
by extending $f_{t}$ to the ASD (anti-self dual) subspace $\Omega^{0,2}_{-}(X,EndE)_{k-1}$ using the definition
$f_{t}(\sqrt{-1}\alpha)\triangleq\sqrt{-1}f_{t}(\alpha)$,
where $\alpha\in \Omega^{0,2}_{+}(X,EndE)_{k-1}$.

It is easy
to check that $s_{t}$ is a Fredholm Banach bundle of index $v.d_{\mathbb{R}}(\mathcal{M}^{DT_{4}}_{c})+1$ which connects $s_{0}$ and $s_{1}$.
Then by \cite{brussee}, we prove the deformation invariance of the virtual cycle of the $DT_{4}$ moduli space.

To sum up , we have the following result.
\begin{theorem}\label{main theorem}
Assume $\mathcal{M}_{c}=\mathcal{M}_{c}^{o}\neq\emptyset$ and assume there exists an orientation data $o(\mathcal{L})$.
Then $\mathcal{M}^{DT_{4}}_{c}$ is compact and its virtual fundamental class exists as a cycle
$[\mathcal{M}^{DT_{4}}_{c}]^{vir}\in H_{r}(\mathcal{B}_{1},\mathbb{Z})$.

Furthermore, if the above assumptions are satisfied by a continuous family of Calabi-Yau four-folds $X_{t}$ parameterized by $t\in [0,1]$, then the cycle in $H_{r}(\mathcal{B}_{1},\mathbb{Z})$ is independent of $t$.
\end{theorem}
\begin{remark} ${}$ \\
1. The Banach manifold $\mathcal{B}_{1}=\mathcal{A}^{*}/\mathcal{G}^{0}$ involves a choice of a large integer $k$ in $L_{k}^{2}$ norm completion. As stated before, the $DT_{4}$ moduli space is independent of the choice of $k$. Meanwhile, the homotopy-invariant properties of
$\mathcal{B}_{1}$ are insensitive to $k$ \cite{dk} and it is easy to show the virtual fundamental class does not depend on the choice of $k$. \\
2. Since $SU(4)\subset Spin(7)$, Calabi-Yau four-fold $X$ is also a $Spin(7)$ manifold. We have
\begin{equation} \Omega^{2}(X)=\Omega^{2}_{7}(X)\oplus\Omega^{2}_{21}(X),
\nonumber\end{equation}
\begin{equation} \Omega^{2}(X)\otimes_{\mathbb{R}}\mathbb{C}=\Omega^{1,1}_{0}(X)\oplus\Omega^{0,0}(X)<\omega>\oplus\Omega^{0,2}(X)\oplus\Omega^{2,0}(X).
\nonumber\end{equation}
Coupled with bundles, the deformation complex of $Spin(7)$ instantons \cite{lewis}
\begin{equation}\Omega^{0}(X,g_{E})\rightarrow\Omega^{1}(X,g_{E})\rightarrow\Omega^{2}_{7}(X,g_{E}) \nonumber \end{equation}
is the same as
\begin{equation}\Omega^{0}(X,g_{E})\rightarrow\Omega^{0,1}(X,EndE)\rightarrow\Omega^{0,2}_{+}(X,EndE)\oplus\Omega^{0}(X,g_{E}). \nonumber \end{equation}
Correspondingly, the $Spin(7)$ instanton equation
\begin{equation}\pi_{7}(F)=0
\nonumber
\end{equation}
is equivalent to $DT_{4}$ equation (\ref{complex ASD equation}), where
\begin{equation}\pi_{7}: \Omega^{2}(X,g_{E})\rightarrow \Omega^{2}_{7}(X,g_{E}) \nonumber \end{equation}
is the projection map.
Thus the $Spin(7)$ instantons counting is just the $DT_{4}$ invariant (defined later) when the base manifold is a Calabi-Yau four-fold.\\
3. Unlike the case of Calabi-Yau three-folds with $SU(3)$ holonomy \cite{js},
class in $H^{2,2}(X)$ may not remain in the same space when we deform the complex structure of the underlying Calabi-Yau four-fold $X$.
When we deform it to the case when non-algebraic stuff appears, i.e. $c\notin\bigoplus_{k}\textrm{ }H^{k,k}(X)$,
our definition will not work anymore and one has to use other nice analytic compactification to define the invariants.
However, if one believes in the deformation invariance of the $Spin(7)$ instantons counting on Calabi-Yau four-folds which is not defined in general at the moment,
it will be interesting to consider the invariants defined here.
\end{remark}
${}$ \\
\textbf{The $\mu_{1}$-map}.
Because the virtual dimension of $\mathcal{M}_{c}^{DT_{4}}$ is in general not zero, we need the $\mu$-map to cut down the dimension and define
the invariant.

Recall \cite{dk}, if $G=SU(2)$, there exists a universal $SO(3)$ bundle
\begin{equation}
\begin{array}{lll}
       \quad P^{ad} \\  \quad  \downarrow   \\ \mathcal{B}_{1}\times X.        
\end{array} \nonumber \end{equation}

Then we define the $\mu_{1}$-map using the slant product pairing,
\begin{equation}\mu_{1}: H_{*}(X)\otimes \mathbb{Z}[x_{1},x_{2},...]\rightarrow H^{*}(\mathcal{B}_{1}),  \nonumber\end{equation}
\begin{equation}\mu_{1}(\gamma,P)=P(0,-\frac{1}{4}p_{1}(P^{ad}),0,...)/\gamma. \nonumber \end{equation}
The $\mu_{1}$-map for other structure group $G$ can be similarly defined, but the formula will be more complicated.

We then define $DT_{4}$ invariants in terms of the virtual fundamental class and the $\mu_{1}$-map.
\begin{definition}\label{DT4 inv of bundles}Under the assumption in Theorem \ref{main theorem},
the $DT_{4}$ invariant of $(X,\mathcal{O}(1))$ with respect to Chern character $c$ and an orientation data $o(\mathcal{L})$
is defined to be a map
\begin{equation}\label{mu map for bundles}DT_{4}^{\mu_{1}}(X,\mathcal{O}(1),c,o(\mathcal{L})):
GrSym^{*}\big(H_{*}(X,\mathbb{Z})\otimes \mathbb{Z}[x_{1},x_{2},...]\big) \rightarrow \mathbb{Z} \end{equation}
such that
\begin{equation}DT_{4}^{\mu_{1}}(X,\mathcal{O}(1),c,o(\mathcal{L}))((\gamma_{1},P_{1}),(\gamma_{2},P_{2}),...) \nonumber \end{equation}
\begin{equation}=<\mu_{1}(\gamma_{1},P_{1})\cup \mu_{1}(\gamma_{2},P_{2})\cup... ,[\mathcal{M}^{DT_{4}}_{c}]^{vir}>, \nonumber \end{equation}
where $<,>$ denotes the natural pairing between homology and cohomology groups.
\end{definition}
\begin{remark} ${}$ \\
1. The above $DT_{4}$ invariants can be viewed as mathematically making sense of
partition functions of certain eight dimension quantum field theory \cite{bks} because of the standard super-symmetry localization \cite{witten}. \\
2. If two Calabi-Yau four-folds under Mukai flops \cite{mukai} are deformation equivalent to
each other \cite{h}, then the $DT_{4}$ invariants remain the same under Mukai flops.

\end{remark}

\subsection{Virtual cycle construction of generalized $DT_{4}$ moduli spaces }
In this subsection, we will concentrate on the virtual cycle construction of the generalized $DT_{4}$ moduli space $\overline{\mathcal{M}}_{c}^{DT_{4}}$ when it is defined without the gluing assumption \ref{assumption on gluing}. \\

The first case is when $\mathcal{M}_{c}$ is smooth. In this case, the obstruction sheaf $ob$ such that $ob|_{\mathcal{F}}=Ext^{2}(\mathcal{F},\mathcal{F})$ is a bundle with quadratic form $Q_{Serre}$, where $Q_{Serre}$ is the Serre duality pairing. By Lemma 5 \cite{eg}, there exists a real bundle $ob_{+}$ with positive definite quadratic form such that $ob\cong ob_{+}\otimes_{\mathbb{R}}\mathbb{C}$ as vector bundles with quadratic form and $w_{1}(ob_{+})=0 \Leftrightarrow c_{1}(ob)=0$. We call $ob_{+}$ the self-dual obstruction bundle and choose an orientation data $o(\mathcal{L})$ for $\mathcal{M}_{c}$ which gives an orientation on $ob_{+}$.
\begin{definition}\label{virtual cycle when Mc smooth}
When $\mathcal{M}_{c}$ is smooth, by Proposition \ref{gene DT4 if Mc smooth}, $\overline{\mathcal{M}}_{c}^{DT_{4}}$ exists and $\overline{\mathcal{M}}_{c}^{DT_{4}}=\mathcal{M}_{c}$. Assume there exists an orientation data $o(\mathcal{L})$ for $\mathcal{M}_{c}$.
Then the virtual fundamental class of $\overline{\mathcal{M}}_{c}^{DT_{4}}$ is defined to be
the Poincar\'{e} dual of the Euler class of the self-dual obstruction bundle over $\mathcal{M}_{c}$, i.e.
\begin{equation}[\overline{\mathcal{M}}^{DT_{4}}_{c}]^{vir}=PD(e(ob_{+}))\in H_{r}(\mathcal{M}_{c},\mathbb{Z}),\nonumber \end{equation}
where $r=2-\chi(\mathcal{F},\mathcal{F})$ (determined by $c$) is the real virtual dimension of $\overline{\mathcal{M}}_{c}^{DT_{4}}$.
\end{definition}
When $\mathcal{M}_{c}$ is smooth, we have the following lemma which is useful for later computations of $DT_{4}$ invariants
\begin{lemma}\cite{eg} \label{ASD equivalent to max isotropic}
Let $E\rightarrow U$ be a complex vector bundle with non-degenerate quadratic form. $V$ is a maximal isotropic subbundle of $E$. \\
(1) If $rk(E)=2n$, then the structure group of $E$ reduces to $SO(2n,\mathbb{C})$ and the half Euler class of $E$ (i.e. the Euler class of the corresponding real quadratic bundle) is $\pm c_{n}(V)$
where the sign
depends on the choice of the maximal isotropic family of $V$. \\
(2) If $rk(E)=2n+1$ and the the structure group of $E$ reduces to $SO(2n+1,\mathbb{C})$, then the class is zero.
\end{lemma}
${}$ \\
The next case where we have $\overline{\mathcal{M}}_{c}^{DT_{4}}$ without the gluing assumption is the following.
\begin{definition}\label{virtual cycle when ob=v+v*}
If there exists a perfect obstruction theory \cite{bf}
\begin{equation}\phi: \quad \mathcal{V}^{\bullet}\rightarrow \mathbb{L}^{\bullet}_{\mathcal{M}_{c}},  \nonumber \end{equation}
such that
\begin{equation}H^{0}(\mathcal{V}^{\bullet})|_{\{\mathcal{F}\}}=Ext^{1}(\mathcal{F},\mathcal{F}), \nonumber \end{equation}
\begin{equation}H^{-1}(\mathcal{V}^{\bullet})|_{\{\mathcal{F}\}}\oplus
H^{-1}(\mathcal{V}^{\bullet})|_{\{\mathcal{F}\}}^{*}=Ext^{2}(\mathcal{F},\mathcal{F}).   \nonumber \end{equation}
And $H^{-1}(\mathcal{V}^{\bullet})|_{\{\mathcal{F}\}}$ is a maximal isotropic subspace of $Ext^{2}(\mathcal{F},\mathcal{F})$ with respect
to the Serre duality pairing. Then by Proposition \ref{ob=v+v*}, $\overline{\mathcal{M}}_{c}^{DT_{4}}$ exists, $\overline{\mathcal{M}}_{c}^{DT_{4}}=\mathcal{M}_{c}$ and the index bundle $Ind_{\mathbb{C}}$ has a natural complex orientation $o(\mathcal{O})$.

The virtual fundamental class of $\overline{\mathcal{M}}_{c}^{DT_{4}}$ with respect to the natural complex orientation $o(\mathcal{O})$ is defined to be the virtual fundamental class of the above perfect
obstruction theory, i.e.
\begin{equation}[\overline{\mathcal{M}}^{DT_{4}}_{c}]^{vir}\triangleq[\mathcal{M}_{c},\mathcal{V}^{\bullet}]^{vir}\in A_{\frac{r}{2}}(\mathcal{M}_{c}),   \nonumber \end{equation}
where $r=2-\chi(\mathcal{F},\mathcal{F})$ is the real virtual dimension of $\overline{\mathcal{M}}_{c}^{DT_{4}}$.
\end{definition}
${}$ \\
\textbf{The $\mu_{2}$-map}:
Now we define a $\mu_{2}$-map for the above two cases. Denote the universal sheaf of $\mathcal{M}_{c}$ by $\mathfrak{F}$
\begin{equation}
\begin{array}{lll}
      \quad \quad \mathfrak{F} \\  \quad \quad \downarrow   \\ \mathcal{M}_{c}\times X.        
\end{array} \nonumber \end{equation}
Define the $\mu_{2}$-map using the slant product pairing
\begin{equation}\mu_{2}: H_{*}(X)\otimes \mathbb{Z}[x_{1},x_{2},...]\rightarrow H^{*}(\mathcal{M}_{c}), \nonumber \end{equation}
\begin{equation}\label{mu map for sheaves} \mu_{2}(\gamma,P)=P(c_{1}(\mathfrak{F}),c_{2}(\mathfrak{F}),...)/\gamma.   \end{equation}

\begin{definition}\label{DT4 inv of sheaves}In Definitions \ref{virtual cycle when Mc smooth}, \ref{virtual cycle when ob=v+v*},
the $DT_{4}$ invariants of $(X,\mathcal{O}(1))$ with respect to Chern character $c$ and an orientation data $o(\mathcal{L})$
is defined to be a map
\begin{equation}\label{u2 map}DT_{4}^{\mu_{2}}(X,\mathcal{O}(1),c,o(\mathcal{L})): GrSym^{*}\big(H_{*}(X,\mathbb{Z}) \otimes \mathbb{Z}[x_{1},x_{2},...]\big)
\rightarrow \mathbb{Z} \end{equation}
such that
\begin{equation}DT_{4}^{\mu_{2}}(X,\mathcal{O}(1),c,o(\mathcal{L}))((\gamma_{1},P_{1}),(\gamma_{2},P_{2}),...) \nonumber \end{equation}
\begin{equation}=<\mu_{2}(\gamma_{1},P_{1})\cup \mu_{2}(\gamma_{2},P_{2})\cup... ,[\overline{\mathcal{M}}^{DT_{4}}_{c}]^{vir}>,
\nonumber \end{equation}
where $<,>$ denotes the natural pairing between homology and cohomology groups.
\end{definition}

Actually the definition of $DT_{4}$ invariants here and the definition before (\ref{mu map for bundles}) coincide.
\begin{proposition}
Assume $\mathcal{M}_{c}=\mathcal{M}_{c}^{o}\neq\emptyset$. If (i) the condition in Definition \ref{virtual cycle when Mc smooth}
is satisfied or (ii) the condition in Definition \ref{virtual cycle when ob=v+v*} is satisfied with the further assumption that the Hermitian metric on each $Ext^{2}(E,E)\cong H^{-1}(\mathcal{V}^{\bullet})|_{\{E\}}\oplus H^{-1}(\mathcal{V}^{\bullet})|_{\{E\}}^{*}$ which is induced from the Hermitian metric on $E$ is the direct sum Hermitian metric induced from a metric on $H^{-1}(\mathcal{V}^{\bullet})|_{\{E\}}$
, then
\begin{equation}DT_{4}^{\mu_{1}}(X,\mathcal{O}(1),c,o(\mathcal{L}))=DT_{4}^{\mu_{2}}(X,\mathcal{O}(1),c,o(\mathcal{L})). \nonumber \end{equation}
\end{proposition}
\begin{proof}
When $\mathcal{M}_{c}$ is smooth, we have $\overline{\mathcal{M}}_{c}^{DT_{4}}=\mathcal{M}_{c}$ and the virtual fundamental class using the definition in the generalized
$DT_{4}$ moduli space is the Poincar\'{e} dual of the Euler class of the self-dual obstruction bundle. Since $\mathcal{M}_{c}$ consists of bundles only,
we can use the Euler class of the Fredholm Banach bundles as the definition of the virtual fundamental class which is nothing but the pushforward
(into $\mathcal{B}_{1}$) of the Poincar\'{e} dual of the Euler class of the self-dual obstruction bundle. Meanwhile the two universal family coincide when restricted to
$\mathcal{M}_{c}$ \cite{dk}, hence we finish the proof.

For the case when the condition in definition \ref{virtual cycle when ob=v+v*} is satisfied, denote $V_{E}=H^{-1}(\mathcal{V}^{\bullet})|_{\{E\}}$.
We assume the Hermitian metric on $Ext^{2}(E,E)\cong V_{E}\oplus V_{E}^{*}$ induced from the metric on the bundle $E$ is the direct sum Hermitian metric induced from a metric on $V_{E}$. By proposition \ref{ob=v+v*}, $Ext^{2}_{+}(E,E)\cong \{a+a^{*} | a\in V_{E}\}\cong V_{E}$ and the Kuranishi map $\kappa_{+}$ factors through $V_{E}$,
\begin{equation}
\xymatrix{
Ext^{1}(E,E) \ar[r]^{\quad \kappa} \ar[dr]_{\kappa_{+}}
& V_{E} \ar[d]^{a\mapsto a+a^{*}}_{\wr\mid}   \\
& Ext^{2}_{+}(E,E).    }
\nonumber \end{equation}
We then use the natural orientation of $V_{E}$ to give an orientation on $Ext^{2}_{+}(E,E)$. The proof of the equivalence of the virtual cycle constructions is standard \cite{lt3}.
\end{proof}

\newpage
\section{$DT_{4}$ invariants for compactly supported sheaves on local $CY_{4}$}
In this section, we study $DT_{4}$ invariants on local (non-compact) Calabi-Yau four-folds.

We first study the moduli space of compactly supported sheaves on Calabi-Yau four-folds of type $K_{Y}$, where $Y$ is a
compact Fano threefold. We define the $DT_{4}$ invariants and show that there is a $DT_{4}/DT_{3}$ correspondence.

Then we define the $DT_{4}$ virtual cycle for the moduli space of sheaves of type $\iota_{*}(\mathcal{F})$, where
\begin{equation}\iota: S\rightarrow X=T^{*}S \nonumber \end{equation}
is the zero section and $\mathcal{F}$ is a stable sheaf on a compact algebraic surface $S$. Then we relate the $DT_{4}$ virtual cycle on $T^{*}S$
to some known invariant on $S$.

\subsection{The case of $X=K_{Y}$}

We first describe the stability of compactly supported sheaves on $X$.
Denote $\iota: Y\rightarrow K_{Y}$ to be the zero section map and $\pi: K_{Y}\rightarrow Y $ to be the projection map.
Pick an ample line bundle $\mathcal{O}_{Y}(1)$ on $Y$ and define the Hilbert polynomial of a compactly supported coherent sheaf $\mathcal{F}$ on $X=K_{Y}$
to be $\chi(\mathcal{F}\otimes \pi^{*}\mathcal{O}_{Y}(k))$ for $k\gg0$. Then we can talk about Gieseker $\pi^{*}\mathcal{O}_{Y}(1)$-stability on compactly supported sheaves over $X$ \cite{hl}.
\begin{lemma}\label{cp supp}
Given a compact Fano threefold $Y$ with $H^{0}(Y,K_{Y}^{-1})\neq0$, then any $\pi^{*}\mathcal{O}_{Y}(1)$-stable sheaf with three-dimensional compact support on local Calabi-Yau four-fold $X=K_{Y}$ is of type $\iota_{*}(\mathcal{F})$,
where $\mathcal{F}$ is $\mathcal{O}_{Y}(1)$-stable on $Y$.
\end{lemma}
\begin{proof}
The proof here is similar to the proof of Lemma 7.1 of \cite{hua}. We only need to show that any compactly supported stable sheaf is
scheme theoretically supported on $Y$. Denote $Z$ to be the scheme theoretical support of a compactly supported stable sheaf $\mathcal{E}$.
By the trace map \cite{hl}, we have
\begin{equation}H^{0}(Z,\mathcal{O}_{Z})\hookrightarrow Ext^{0}_{Z}(\mathcal{E},\mathcal{E}). \nonumber \end{equation}
It suffices to show the dimension of
$H^{0}(Z,\mathcal{O}_{Z})$ is bigger than one to get contradiction because stable sheaf $\mathcal{E}$ is always simple.

By the assumption, dimension of the support of $\mathcal{E}$ is three.
Then $Z$ is an order $n\geq 1$ thickening of $Y$ in the normal direction inside $X$, i.e.
\begin{equation}Z=\mathbf{Spec}\big(\bigoplus_{i=0}^{n}K_{Y}^{-i}\big).\nonumber \end{equation}
There is a spectral sequence such that $E_{\infty}^{0,0}=H^{0}(Z,\mathcal{O}_{Z})$ with $E_{2}^{0,0}=H^{0}(Y,\oplus_{i=0}^{n}K_{Y}^{-i})$.
Thus
\begin{equation}H^{0}(Z,\mathcal{O}_{Z})=\oplus_{i=0}^{n}H^{0}(Y,K_{Y}^{-i}) .\nonumber \end{equation}
Because $H^{0}(Y,K_{Y}^{-1})\neq 0$, then $dim H^{0}(Z,\mathcal{O}_{Z})\geq 2$, which leads to a contradiction.
Thus $\mathcal{E}$ is of type $\iota_{*}\mathcal{F}$, where $\mathcal{F}$ is a sheaf on $Y$. Now we show $\mathcal{F}$ is stable with respect to $\mathcal{O}_{Y}(1)$.

By the projection formula, for any $k$, we have
\begin{equation}\iota_{*}(\mathcal{F}\otimes_{\mathcal{O}_{Y}}\mathcal{O}_{Y}(k))
=\iota_{*}(\mathcal{F}\otimes_{\mathcal{O}_{Y}}\iota^{*}\pi^{*}\mathcal{O}_{Y}(k))
=\iota_{*}\mathcal{F}\otimes_{\mathcal{O}_{X}}\pi^{*}\mathcal{O}_{Y}(k). \nonumber \end{equation}
Thus
\begin{equation}H^{*}(Y,\mathcal{F}\otimes_{\mathcal{O}_{Y}}\mathcal{O}_{Y}(k))
=H^{*}(X,\iota_{*}\mathcal{F}\otimes_{\mathcal{O}_{X}}\pi^{*}\mathcal{O}_{Y}(k)). \nonumber \end{equation}
Then we know the stability condition for $\iota_{*}\mathcal{F}$ on $X=K_{Y}$ is equivalent to the stability condition for $\mathcal{F}$ on $Y$.
\end{proof}
Now we want to study the relations between the obstruction theory of sheaf of type $\iota_{*}(\mathcal{F})$ on $X=K_{Y}$
and the obstruction theory of sheaf $\mathcal{F}$ on $X$.
\begin{definition}\cite{hua}
Let $L=\oplus_{i=0}^{d}L^{i}$ be a finite dimensional $L_{\infty}$ algebra over $\mathbb{C}$,
with its $\mathbb{C}$ products $\mu_{k}$. Let $\overline{L}$
be the graded vector space $L\oplus L[-d-1]$, i.e $\overline{L}^{i}=L^{i}\oplus(L^{d+1-i})^{*}$.
Define the cyclic pairing and $L_{\infty}$ products $\overline{\mu}_{k}: \wedge^{k}\overline{L}\rightarrow \overline{L}[2-k]$
according to rules: \\
$(1)$ define the bilinear form $\kappa$ on $\overline{L}$ by the natural pairing between $L$ and $L^{*}$, \\
$(2)$ if the inputs of $\overline{\mu}_{k}$ all belong to $L$, then $\overline{\mu}_{k}=\mu_{k}$,  \\
$(3)$ if more than one input belong to $L^{*}$, then define $\overline{\mu}_{k}=0$, \\
$(4)$ if there is exactly one input $a_{i}^{*}\in L^{*}$, then define $\overline{\mu}_{k}$ by
\begin{equation}\kappa(\overline{\mu}_{k}(a_{1},...,a_{i},...,a_{k}),b)=(-1)^{\epsilon}
\kappa(\mu_{k}(a_{i+1},...,a_{k},b,a_{1},...,a_{i-1}),a_{i}^{*})  \nonumber \end{equation}
for arbitary $b\in L$ and $\epsilon$ depends on $a_{i}$, $b$ only.  \\
We then call the $L_{\infty}$ algebra $(\overline{L},\overline{\mu}_{k},\kappa)$ the $d+1$ dimensional cyclic completion of $L$.
\end{definition}
\begin{lemma}\cite{segal} \label{segal theorem}
Let $Y$ be a smooth proper scheme of $dim_{\mathbb{C}}=d-1$, $\iota: Y\rightarrow K_{Y}$ be the zero section map. Then for any $S\in D^{b}(Y)$ the
$A_{\infty}$ algebra $Ext^{*}_{K_{Y}}(\iota_{*}S,\iota_{*}S)$ is the d-dim cyclic completion of $Ext^{*}_{Y}(S,S)$.
\end{lemma}
Using the above lemma, we have,
\begin{lemma}Let $\mathcal{F}$ be a torsion-free slope-stable sheaf on a compact Fano threefold $Y$.
Denote $\iota: Y\rightarrow K_{Y}=X$ to be the zero section map. Then we have canonical isomorphisms
\begin{equation}Ext^{1}_{X}(\iota_{*}\mathcal{F},\iota_{*}\mathcal{F})\cong Ext^{1}_{Y}(\mathcal{F},\mathcal{F}), \nonumber \end{equation}
\begin{equation}\label{cyclic cpt on coh}Ext^{2}_{X}(\iota_{*}\mathcal{F},\iota_{*}\mathcal{F})\cong
Ext^{2}_{Y}(\mathcal{F},\mathcal{F})\oplus Ext^{2}_{Y}(\mathcal{F},\mathcal{F})^{*}. \end{equation}
And a local Kuranishi map for moduli space of sheaves of type $\iota_{*}\mathcal{F}$ on $X$
\begin{equation}\kappa: Ext^{1}_{X}(\iota_{*}\mathcal{F},\iota_{*}\mathcal{F})\rightarrow
Ext^{2}_{X}(\iota_{*}\mathcal{F},\iota_{*}\mathcal{F}) \nonumber \end{equation}
can be identified with a local Kuranishi map for moduli space of sheaves of type $\mathcal{F}$ on $Y$
\begin{equation}Ext^{1}_{Y}(\mathcal{F},\mathcal{F})\rightarrow Ext^{2}_{Y}(\mathcal{F},\mathcal{F}). \nonumber \end{equation}
Furthermore, under the above identification, $Ext^{2}_{Y}(\mathcal{F},\mathcal{F})$ is a maximal isotropic subspace of $Ext^{2}_{X}(\iota_{*}\mathcal{F},\iota_{*}\mathcal{F})$ with
respect to the Serre duality pairing.
\end{lemma}
\begin{proof}
Locally around a sheaf $\iota_{*}\mathcal{F}$, we have the Kuranishi map
\begin{equation}\kappa: Ext^{1}_{X}(\iota_{*}\mathcal{F},\iota_{*}\mathcal{F})\rightarrow
Ext^{2}_{X}(\iota_{*}\mathcal{F},\iota_{*}\mathcal{F}), \nonumber \end{equation}
which can be described by the induced $L_{\infty}$ map on the minimal model.
By Lemma \ref{segal theorem}, we have
\begin{equation}Ext^{1}_{X}(\iota_{*}\mathcal{F},\iota_{*}\mathcal{F})=Ext^{1}_{Y}(\mathcal{F},\mathcal{F})\oplus
Ext^{3}_{Y}(\mathcal{F},\mathcal{F})^{*}=Ext^{1}_{Y}(\mathcal{F},\mathcal{F}), \nonumber \end{equation}
\begin{equation}\label{cyclic cpt on coh}Ext^{2}_{X}(\iota_{*}\mathcal{F},\iota_{*}\mathcal{F})=
Ext^{2}_{Y}(\mathcal{F},\mathcal{F})\oplus Ext^{2}_{Y}(\mathcal{F},\mathcal{F})^{*}. \end{equation}
And we can identify $\kappa$ with the Kuranishi map for the moduli space of stable sheaves on $Y$ near $\mathcal{F}$.

By (\ref{cyclic cpt on coh}), we know the quadratic pairing
\begin{equation}Ext^{2}_{X}(\iota_{*}\mathcal{F},\iota_{*}\mathcal{F})\otimes Ext^{2}_{X}(\iota_{*}\mathcal{F},\iota_{*}\mathcal{F})\rightarrow
Ext^{4}_{X}(\iota_{*}\mathcal{F},\iota_{*}\mathcal{F}) \nonumber \end{equation}
is just the quadratic pairing
\begin{equation}(Ext^{2}_{Y}(\mathcal{F},\mathcal{F})\oplus Ext^{2}_{Y}(\mathcal{F},\mathcal{F})^{*})\otimes (Ext^{2}_{Y}(\mathcal{F},\mathcal{F})\oplus Ext^{2}_{Y}(\mathcal{F},\mathcal{F})^{*})\rightarrow
Ext^{3}_{Y}(\mathcal{F},\mathcal{F}\otimes K_{Y}).   \nonumber \end{equation}
When we restrict to $Ext^{2}_{Y}(\mathcal{F},\mathcal{F})$, the pairing will produce an element in
\begin{equation}Ext^{4}_{Y}(\mathcal{F},\mathcal{F})=0. \nonumber \end{equation}
Hence, we can identify the Serre duality pairing on the space $Ext^{2}_{X}(\iota_{*}\mathcal{F},\iota_{*}\mathcal{F})$ with
the natural pairing between $Ext^{2}_{Y}(\mathcal{F},\mathcal{F})$ and $Ext^{2}_{Y}(\mathcal{F},\mathcal{F})^{*}$.
These two subspaces are both maximal isotropic subspaces.
\end{proof}
By the above lemma and Lemma \ref{cp supp}, we know that for a polarized compact Fano threefold $(Y,\mathcal{O}_{Y}(1))$ with $H^{0}(Y,K_{Y}^{-1})\neq0$, the moduli space of $\pi^{*}\mathcal{O}_{Y}(1)$ slope-stable compactly supported sheaves on $K_{Y}$ with compactly supported Chern character \cite{js} $c=(0,c|_{H_{c}^{2}(X)}\neq 0,c|_{H_{c}^{4}(X)},c|_{H_{c}^{6}(X)},c|_{H_{c}^{8}(X)})$ can be identified with moduli space of torsion-free $\mathcal{O}_{Y}(1)$ stable sheaves on $Y$ with certain Chern character $c^{'}\in H^{even}(Y)$ which is uniquely determined by $c$. Meanwhile condition in definition \ref{virtual cycle when ob=v+v*} is satisfied \cite{th}.

By definition \ref{virtual cycle when ob=v+v*}, the generalized $DT_{4}$ moduli space exists and we can identify its virtual fundamental class with the virtual fundamental class of the moduli space of stable sheaves on $Y$. Furthermore, if we use the same $\mu$-map (\ref{mu map for sheaves}) to define invariants, we can also identify them.
\begin{theorem}\label{compact supp DT4}($DT_{4}/DT_{3}$) \\
Let $\pi: X=K_{Y}\rightarrow Y$ be the projection map and $(Y,\mathcal{O}_{Y}(1))$ be a polarized compact Fano threefold with $H^{0}(Y,K_{Y}^{-1})\neq0$. If $c=(0,c|_{H_{c}^{2}(X)}\neq 0,c|_{H_{c}^{4}(X)},c|_{H_{c}^{6}(X)},c|_{H_{c}^{8}(X)})$ and $\mathcal{M}_{c}(X,\pi^{*}\mathcal{O}_{Y}(1))$ consists of slope-stable sheaves, then sheaves in $\mathcal{M}_{c}(X,\pi^{*}\mathcal{O}_{Y}(1))$ are of type $\iota_{*}(\mathcal{F})$, where $\iota: Y\rightarrow K_{Y}$ is the zero section and $c^{'}=ch(\mathcal{F})\in H^{even}(Y)$ is uniquely determined by $c$. Furthermore, the generalized $DT_{4}$ moduli space exists
\begin{equation}\overline{\mathcal{M}}_{c}^{DT_{4}}(X,\pi^{*}\mathcal{O}_{Y}(1))=\mathcal{M}_{c}(X,\pi^{*}\mathcal{O}_{Y}(1))\cong \mathcal{M}_{c'}(Y,\mathcal{O}_{Y}(1)) \nonumber \end{equation}
and its virtual fundamental class (Definition \ref{virtual cycle when ob=v+v*}) satisfies
\begin{equation}[\overline{\mathcal{M}}_{c}^{DT_{4}}(X,\pi^{*}\mathcal{O}_{Y}(1))]^{vir}=[\mathcal{M}_{c'}(Y,\mathcal{O}_{Y}(1))]^{vir}, \nonumber \end{equation}
where $[\mathcal{M}_{c'}(Y,\mathcal{O}_{Y}(1))]^{vir}$ is the $DT_{3}$ virtual cycle defined in \cite{th}.

Since $H_{*}(X)\cong H_{*}(Y)$ and $H^{*}(\overline{\mathcal{M}}_{c}^{DT_{4}}(X,\pi^{*}\mathcal{O}_{Y}(1)))\cong H^{*}(\mathcal{M}_{c'}(Y,\mathcal{O}_{Y}(1)))$, we can use the same $\mu_{2}$-map (\ref{mu map for sheaves}) to define invariants, then
\begin{equation}DT_{4}^{\mu_{2}}(X,\pi^{*}\mathcal{O}_{Y}(1),c,o(\mathcal{O}))=DT_{3}(Y,\mathcal{O}_{Y}(1),c^{'}), \nonumber \end{equation}
where $o(\mathcal{O})$ is the natural complex orientation and $DT_{3}(Y,\mathcal{O}_{Y}(1),c^{'})$ is defined by pairing $[\mathcal{M}_{c'}(Y,\mathcal{O}_{Y}(1))]^{vir}$ with the $\mu_{2}$-map (\ref{u2 map}).
\end{theorem}
\begin{remark} ${}$ \\
1. If we have a compact complex smooth four-fold $X$ (no need to be Calabi-Yau) containing a
Fano threefold $Y$ such that $\mathcal{N}_{Y/X}=K_{Y}$ and $\mathcal{N}_{Y/X}^{*}$ is ample, e.g $X=P(K_{Y}\oplus\mathcal{O})$.
Then we can define $DT_{4}$ invariants for stable sheaves supported in $Y$ because the renowned theorem of Grauert implies
that $X$ contains $K_{Y}$ as its open subset. \\
2. Under the assumption that $Y$ admits a full strong exceptional collection \cite{bri}, \cite{hua}, we have a quiver representation of $\mathcal{M}_{c}$ which may be helpful for the future study of $DT_{4}$ invariants.
\end{remark}

\subsection{The case of $X=T^{*}S$}

Now we consider $X=T^{*}S$.  It is a hyper-K\"ahler four-fold when $S=\mathbb{P}^{2}$.
We only consider counting torsion sheaves scheme theoretically supported on $S$.

Let $\mathcal{F}$ be a torsion-free Gieseker stable sheaf on a compact algebraic surface $(S,\mathcal{O}_{S}(1))$.
Denote $\iota: S\rightarrow T^{*}S$ to be the inclusion map, then $\iota_{*}(\mathcal{F})$ is a torsion sheaf over $X$
and denote $\pi: T^{*}S\rightarrow S$ to be the projection map. \\

Now we want to relate the obstruction theory of sheaf $\iota_{*}(\mathcal{F})$ on $X$ to the obstruction theory of $\mathcal{F}$ on $S$.

By the projection formula \cite{hart},
\begin{equation}\iota_{*}(\mathcal{F})=\iota_{*}(\iota^{*}\pi^{*}\mathcal{F}\otimes_{\mathcal{O}_{S}}\mathcal{O}_{S})
=\pi^{*}\mathcal{F}\otimes_{\mathcal{O}_{X}}\iota_{*}\mathcal{O}_{S},
\nonumber \end{equation}
where $\mathcal{F}$ is a complex of locally free sheaves on $S$.
Then we have a local to global spectral sequence
\begin{eqnarray*}Ext^{*}_{X}(\iota_{*}\mathcal{F},\iota_{*}\mathcal{F})&\Leftarrow&
H^{*}(X,\mathcal{E}xt^{*}_{X}(\iota_{*}\mathcal{F},\iota_{*}\mathcal{F})) \\
&=& H^{*}(X,\mathcal{E}xt^{*}_{X}(\pi^{*}\mathcal{F}\otimes_{\mathcal{O}_{X}}\iota_{*}\mathcal{O}_{S},
\pi^{*}\mathcal{F}\otimes_{\mathcal{O}_{X}}\iota_{*}\mathcal{O}_{S})) \\
&=& H^{*}(X,\mathcal{E}xt^{*}_{X}(\iota_{*}\mathcal{O}_{S},\iota_{*}\mathcal{O}_{S})\otimes_{\mathcal{O}_{X}}End (\pi^{*}\mathcal{F})) \\
&=& H^{*}(X,\mathcal{E}xt^{*}_{X}(\mathcal{O}_{S},\mathcal{O}_{S})\otimes_{\mathcal{O}_{X}} End (\pi^{*}\mathcal{F})) \\
&=& H^{*}(X,\iota_{*}(\wedge^{*}\mathcal{N}_{S/X})\otimes_{\mathcal{O}_{X}} \pi^{*}End \mathcal{F}) \\
&=& H^{*}(X,\iota_{*}(\wedge^{*}\mathcal{N}_{S/X}\otimes_{\mathcal{O}_{S}} \iota^{*}\pi^{*}End \mathcal{F})) \\
&=& H^{*}(S,\wedge^{*}\Omega^{1}_{S}\otimes End \mathcal{F}) \\
&=& Ext^{*}_{S}(\mathcal{F},\wedge^{*}\Omega^{1}_{S}\otimes \mathcal{F}).
\end{eqnarray*}
We give a criterion when the above spectral sequence degenerates at $E_{2}$ terms.
\begin{lemma}\label{TS lemma1}Let $\mathcal{F}$ be a torsion free sheaf on $S$. \\
$(1)$ If $Ext_{S}^{2}(\mathcal{F},\mathcal{F})=0$, the above spectral sequence degenerates at $E_{2}$.  \\
$(2)$ If the degree of $K_{S}$ is negative with respect to the chosen polarization $\mathcal{O}_{S}(1)$ and $\mathcal{F}$ is slope-stable, then we have $Ext_{S}^{2}(\mathcal{F},\mathcal{F})=0$.
\end{lemma}
\begin{proof}   ${}$ \\
$(1)$ Denote $E_{2}^{p,q}=Ext_{S}^{q}(\mathcal{F},\wedge^{p}\Omega^{1}_{S}\otimes \mathcal{F})$. We have
\begin{equation}E_{2}^{p-2,q+1}\rightarrow E_{2}^{p,q}\rightarrow E_{2}^{p+2,q-1}, \nonumber \end{equation}
whose cohomology is $E_{3}^{p,q}$. Then
\begin{equation}0\rightarrow E_{2}^{1,q}\rightarrow 0 \nonumber \end{equation}
yields $E_{3}^{1,q}\cong E_{2}^{1,q}$. Meanwhile we have
\begin{equation}E_{2}^{0,q+1}\rightarrow E_{2}^{2,q}\rightarrow 0 , 0\rightarrow E_{2}^{0,q}\rightarrow E_{2}^{2,q-1}.
\nonumber \end{equation}
Under the assumption that $Ext_{S}^{2}(\mathcal{F},\mathcal{F})=0$. We have
\begin{equation}E_{2}^{2,0}=E_{2}^{0,2}=0. \nonumber \end{equation}
Thus the above spectral sequence degenerates at $E_{2}$. \\
${}$ \\
$(2)$ By Serre duality, we have
\begin{equation} Ext_{S}^{2}(\mathcal{F},\mathcal{F})=Hom_{\mathcal{O}_{S}}(\mathcal{F},\mathcal{F}\otimes K_{S}).  \nonumber \end{equation}
By assumption,
\begin{equation}\mu(\mathcal{F})=\frac{deg(\mathcal{F})}{rk(\mathcal{F})}>\frac{deg(\mathcal{F}\otimes K_{S})}
{rk(\mathcal{F}\otimes K_{S})}= \mu(\mathcal{F}\otimes K_{S}).\nonumber \end{equation}
If the above homomorphism is not zero, choose such a nonzero morphism
\begin{equation}f:\mathcal{F}\rightarrow \mathcal{F}\otimes K_{S},  \nonumber \end{equation}
then
\begin{equation}0\neq\mathcal{F}/ker(f)\hookrightarrow \mathcal{F}\otimes K_{S}.  \nonumber \end{equation}
By the stability of $\mathcal{F}$, we have
\begin{equation}\mu(\mathcal{F}/ker(f))\geq\mu(\mathcal{F}).  \nonumber \end{equation}
Thus
\begin{equation}\mu(\mathcal{F}\otimes K_{S})<\mu(\mathcal{F})\leq\mu(\mathcal{F}/ker(f)), \nonumber \end{equation}
which contradicts with the semi-stability of $\mathcal{F}\otimes K_{S}$.
\end{proof}
Now, to ensure that sheaves scheme theoretically supported on $S$ can not move outside.
This can be done by finding conditions such that $Ext^{0}_{S}(\mathcal{F},\mathcal{F}\otimes \Omega_{S}^{1})=0$. If $S=\mathbb{P}^{2}$ and $\mathcal{F}$ is torsion-free slope stable, the condition is satisfied since $\Omega_{S}^{1}$ is stable. When $\mathcal{F}=I$ is an ideal sheaf of points, we have
\begin{lemma}\label{TS lemma2}
Let $\mathcal{F}=I$ be an ideal sheaf of points on $S$. If $h^{0,1}(S)=0$, then
$Ext^{0}_{S}(\mathcal{F},\mathcal{F}\otimes \Omega_{S}^{1})=0$.
\end{lemma}
\begin{proof}
By the short exact sequence
\begin{equation}0\rightarrow I\rightarrow \mathcal{O}_{S}\rightarrow \mathcal{O}_{Z}\rightarrow 0 ,\nonumber \end{equation}
we have
\begin{eqnarray*}0\rightarrow Ext^{0}_{S}(I,I\otimes \Omega_{S}^{1})\rightarrow Ext^{0}_{S}(I,\Omega_{S}^{1})&\cong&
Ext^{2}(\Omega_{S}^{1},I\otimes K_{S}) \\
&=& H^{2}(S,I\otimes K_{S}\otimes T_{S}),
\end{eqnarray*}
while
\begin{equation}0=H^{1}(S,\mathcal{O}_{Z}\otimes K_{S}\otimes T_{S})\rightarrow H^{2}(S,I\otimes K_{S}\otimes T_{S})\rightarrow
H^{2}(S,K_{S}\otimes T_{S})\cong H^{0}(S,\Omega_{S}^{1}) .\nonumber \end{equation}
Hence
\begin{equation}h^{1,0}=0\Rightarrow Ext^{0}_{S}(I,I\otimes \Omega_{S}^{1})=0 .\nonumber \end{equation}
\end{proof}
\begin{proposition}
Under the following assumption
\begin{equation}\label{vanishing assumption}Ext^{0}_{S}(\mathcal{F},\mathcal{F}\otimes \Omega_{S}^{1})=0,
Ext^{2}_{S}(\mathcal{F},\mathcal{F})=0,  \end{equation}
which is satisfied when (i) $S$ is del-Pezzo, $\mathcal{F}$ is an ideal sheaf of points on $S$ or (ii) when $S=\mathbb{P}^{2}$, $\mathcal{F}$ is slope-stable torsion-free on $S$, we have canonical isomorphisms
\begin{equation}Ext^{1}_{X}(\iota_{*}\mathcal{F},\iota_{*}\mathcal{F})\cong Ext^{1}_{S}(\mathcal{F},\mathcal{F}),\nonumber \end{equation}
\begin{equation}Ext^{2}_{X}(\iota_{*}\mathcal{F},\iota_{*}\mathcal{F})\cong Ext^{1}_{S}(\mathcal{F},\mathcal{F}\otimes \Omega_{S}^{1}).
\nonumber \end{equation}
\end{proposition}
\begin{proof}
By Lemma \ref{TS lemma1}, Lemma \ref{TS lemma2} and the degenerate spectral sequence.
\end{proof}
Under assumption (\ref{vanishing assumption}), we denote
\begin{equation}\mathcal{M}_{c}^{S_{cpn}}\triangleq\{\iota_{*}\mathcal{F} \textrm{ } | \mathcal{F}\in \mathcal{M}_{c}(S,\mathcal{O}_{S}(1))\}\cong \mathcal{M}_{c}(S,\mathcal{O}_{S}(1)) \nonumber \end{equation}
to be the components of moduli space of sheaves on $X$ which can be identified with $\mathcal{M}_{c}(S,\mathcal{O}_{S}(1))$ (moduli of $\mathcal{O}_{S}(1)$ stable sheaves on $S$ with Chern character $c\in H^{even}(S)$).

We will use the philosophy of defining $DT_{4}$ virtual cycles to define the virtual fundamental class of $\mathcal{M}_{c}^{S_{cpn}}$.

Note that $\mathcal{M}_{c}(S,\mathcal{O}_{S}(1))$ is smooth by assumption (\ref{vanishing assumption}), we define the $DT_{4}$ virtual cycle of $\mathcal{M}_{c}^{S_{cpn}}$ to be the Poincar\'{e} dual of the Euler class of the self-dual obstruction bundle as in Definition \ref{virtual cycle when Mc smooth}.
\begin{proposition}
Under assumption (\ref{vanishing assumption}), $\mathcal{M}_{c}^{S_{cpn}}\cong \mathcal{M}_{c}(S,\mathcal{O}_{S}(1))$ and the $DT_{4}$ virtual cycle $[\mathcal{M}_{c}^{S_{cpn}}]^{vir}=0$.
\end{proposition}
\begin{proof}
By the Hirzebruch-Riemann-Roch theorem and the assumption (\ref{vanishing assumption}), we have
\begin{equation}dim_{\mathbb{C}} Ext^{1}_{S}(\mathcal{F},\mathcal{F}\otimes \Omega_{S}^{1})=2dim_{\mathbb{C}} Ext^{1}_{S}(\mathcal{F},\mathcal{F})+
r^{2}e(S)-2 ,\nonumber \end{equation}
where $r\geq1$ is the rank of $\mathcal{F}$, $e(S)$ is the Euler characteristic of the surface $S$.

Then, we can see the real virtual dimension
\begin{equation}v.d_{\mathbb{R}}(\mathcal{M}_{c}^{S_{cpn}})\triangleq 2ext^{1}(\iota_{*}\mathcal{F},\iota_{*}\mathcal{F})-ext^{2}(\iota_{*}\mathcal{F},\iota_{*}\mathcal{F})=2-r^{2}e(S)< 0. \nonumber \end{equation}
\end{proof}
\begin{remark}
Note that since we have assumed $Ext^{2}_{S}(\mathcal{F},\mathcal{F})=0$, we automatically have $h^{0,2}(S)=0$.
\end{remark}
${}$ \\
\textbf{The reduced counting}.
Because of the above vanishing result, we take away the obvious trivial factor by considering trace-free part of the obstruction space and
get to the definition of reduced $DT_{4}$ virtual cycles.
\begin{definition}\label{red vir cycle}
Let $X=T^{*}S$ where $S$ is a compact algebraic surface with $h^{0,1}(S)=0$. Under assumption (\ref{vanishing assumption}),
we define the reduced $DT_{4}$ virtual cycle of $\mathcal{M}_{c}^{S_{cpn}}$ to be
the Poincar\'{e} dual of the Euler class of the self-dual trace-free obstruction bundle (if it is orientable)
\begin{equation}[\mathcal{M}_{c}^{S_{cpn}}]^{vir}_{red}\triangleq PD\big(e(ob_{0,+})\big)\in
H_{r_{red}}(\mathcal{M}_{c}\big(S,\mathcal{O}_{S}(1)\big),\mathbb{Z}), \nonumber \end{equation}
where $ob_{0,+}$ is the self-dual trace-free obstruction bundle, $r_{red}=h^{1,1}(S)+2-r^{2}(2+h^{1,1}(S))$ and $\mathcal{M}_{c}(S,\mathcal{O}_{S}(1))$ is
the Gieseker moduli space of $\mathcal{O}_{S}(1)$ stable sheaves on $S$.
\end{definition}
\begin{remark}
In the above definition, the Euler class involves a choice of orientation on each connected component of $\mathcal{M}_{c}^{S_{cpn}}$. We will see for most interesting cases, a natural orientation exists.
\end{remark}
We first show the vanishing of the reduced virtual cycle when sheaves in $\mathcal{M}_{c}(S,\mathcal{O}_{S}(1))$ have $rank\geq2$.
\begin{proposition}
$[\mathcal{M}_{c}^{S_{cpn}}]^{vir}_{red}=0$, if $c|_{H^{0}(S)}\geq2$.
\end{proposition}
\begin{proof}
The reduced virtual dimension $r_{red}=h^{1,1}+2-r^{2}(2+h^{1,1})< 0$ if $r\geq2$.
\end{proof}
As we can see, under assumption (\ref{vanishing assumption}) and $r=1$,  we have $r_{red}=0$.
The corresponding reduced $DT_{4}$ virtual cycle is zero dimensional. \\

For ideal sheaves of curves on $S$ which are just line bundles on $S$
\begin{equation}Ext^{1}_{S}(\mathcal{F},\mathcal{F})=H^{1}(S,\mathcal{O})=0,  \nonumber \end{equation}
which shows that both the tangent space and reduced obstruction space are zero. Then the moduli space is just one point and the reduced $DT_{4}$ invariant is 1 in this case.
\begin{proposition}
$[\mathcal{M}_{c}^{S_{cpn}}]^{vir}_{red}=1$, where $c=(1,c|_{H^{2}(S)},0)$.
\end{proposition}
Finally, we come the case of ideal sheaves of points on $S$.
\begin{lemma}
Let $S$ be a compact algebraic surface with $h^{0,i}(S)=0$, $i=1,2$. Let $I$ be an ideal sheaf of points on $S$,
then we have a canonical isomorphism
\begin{equation}Ext^{1}_{S}(I,I\otimes \Omega_{S}^{1})_{0}\cong Ext^{1}_{S}(\mathcal{O}_{Z},\mathcal{O}_{Z}\otimes \Omega_{S}^{1}).
\nonumber\end{equation}
Furthermore, under this identification, $Ext^{1}_{S}(\mathcal{O}_{Z},\mathcal{O}_{Z})$ is
a maximal isotropic subspace with respect to the Serre duality pairing.
\end{lemma}
\begin{proof}
Denote an ideal sheaf of $n$-points on $S$ by $I$. Taking cohomology of the short exact sequence of sheaves
\begin{equation}0\rightarrow I\otimes \Omega_{S}^{1}\rightarrow\Omega_{S}^{1}\rightarrow \mathcal{O}_{Z}\otimes \Omega_{S}^{1}\rightarrow 0,
\nonumber \end{equation}
where $\mathcal{O}_{Z}$ is the structure sheaf of $n$-points. We have
\begin{equation}\label{cotangent of S equ 0} 0\rightarrow H^{0}(S,\mathcal{O}_{Z}\otimes \Omega_{S}^{1})\rightarrow H^{1}(S,I\otimes \Omega_{S}^{1})\rightarrow H^{1}(S,\Omega_{S}^{1})\rightarrow 0, \end{equation}
and
\begin{equation}\label{cotangent of S equ 1} H^{2}(S,I\otimes\Omega_{S}^{1})\cong H^{2}(S,\Omega_{S}^{1})\cong H^{0}(S,\Omega_{S}^{1})=0.\end{equation}
Applying $Hom_{\mathcal{O}_{S}}(\mathcal{O}_{Z},\cdot)$ to
\begin{equation}0\rightarrow I\otimes \Omega_{S}^{1}\rightarrow\Omega_{S}^{1}\rightarrow \mathcal{O}_{Z}\otimes \Omega_{S}^{1}\rightarrow 0,
\nonumber \end{equation}
we get
\begin{equation}\label{cotangent of S equ 2}Ext^{0}_{S}(\mathcal{O}_{Z},I\otimes \Omega_{S}^{1})=0,
Ext^{0}_{S}(\mathcal{O}_{Z},\mathcal{O}_{Z}\otimes \Omega_{S}^{1})\cong Ext^{1}_{S}(\mathcal{O}_{Z},I\otimes \Omega_{S}^{1}), \end{equation}
\begin{equation}Ext^{1}_{S}(\mathcal{O}_{Z},\mathcal{O}_{Z}\otimes \Omega_{S}^{1})\cong Ext^{2}_{S}(\mathcal{O}_{Z},I\otimes \Omega_{S}^{1}).
\nonumber \end{equation}
Applying $Hom_{\mathcal{O}_{S}}(\cdot,I\otimes \Omega_{S}^{1})$ to
\begin{equation}0\rightarrow I\rightarrow\mathcal{O}_{S}\rightarrow \mathcal{O}_{Z} \rightarrow 0 ,
\nonumber \end{equation}
we have
\begin{equation}\rightarrow Ext^{i}_{S}(\mathcal{O}_{Z},I\otimes \Omega_{S}^{1})\rightarrow
Ext^{i}_{S}(\mathcal{O}_{S},I\otimes \Omega_{S}^{1})\rightarrow  Ext^{i}_{S}(I,I\otimes \Omega_{S}^{1})\rightarrow.\nonumber \end{equation}
By the condition $h^{0,1}(S)=0$, we have $Ext^{0}_{S}(I,I\otimes \Omega_{S}^{1})=0$ by Lemma \ref{TS lemma2}.
Using (\ref{cotangent of S equ 1}), (\ref{cotangent of S equ 2}), we can get
\begin{equation}0\rightarrow Ext^{0}_{S}(\mathcal{O}_{Z},\mathcal{O}_{Z}\otimes \Omega_{S}^{1})\rightarrow
H^{1}(S,I\otimes\Omega_{S}^{1})\rightarrow \nonumber \end{equation}
\begin{equation}\rightarrow Ext^{1}_{S}(I,I\otimes \Omega_{S}^{1})\rightarrow Ext^{1}_{S}(\mathcal{O}_{Z},\mathcal{O}_{Z}\otimes \Omega_{S}^{1})\rightarrow 0 .\nonumber \end{equation}
By (\ref{cotangent of S equ 0}), we get
\begin{equation}0\rightarrow H^{1}(S,\Omega_{S}^{1})\rightarrow Ext^{1}_{S}(I,I\otimes \Omega_{S}^{1})\rightarrow Ext^{1}_{S}(\mathcal{O}_{Z},\mathcal{O}_{Z}\otimes \Omega_{S}^{1})\rightarrow0.\nonumber\end{equation}
where the first injective map is the inclusion of the trace factor.

Taking into consideration of the Serre duality pairing,
\begin{equation}Ext^{1}_{S}(I,I\otimes \Omega_{S}^{1})_{0}\otimes Ext^{1}_{S}(I,I\otimes \Omega_{S}^{1})_{0}\rightarrow
Ext^{2}_{S}(I,I\otimes \Omega_{S}^{2})\rightarrow H^{2,2}(S)\nonumber \end{equation}
can be identified with
\begin{equation}Ext^{1}_{S}(\mathcal{O}_{Z},\mathcal{O}_{Z}\otimes \Omega_{S}^{1})\otimes Ext^{1}_{S}(\mathcal{O}_{Z},\mathcal{O}_{Z}\otimes \Omega_{S}^{1})\rightarrow Ext^{2}_{S}(\mathcal{O}_{Z},\mathcal{O}_{Z}\otimes \Omega_{S}^{2})\rightarrow H^{2,2}(S), \nonumber \end{equation}
where the last map is taking trace. Furthermore it can be identified with
\begin{equation}(Ext^{1}_{S}(\mathcal{O}_{Z},\mathcal{O}_{Z})\oplus Ext^{1}_{S}(\mathcal{O}_{Z},\mathcal{O}_{Z}\otimes \Omega_{S}^{2}))\otimes
(Ext^{1}_{S}(\mathcal{O}_{Z},\mathcal{O}_{Z})\oplus Ext^{1}_{S}(\mathcal{O}_{Z},\mathcal{O}_{Z}\otimes \Omega_{S}^{2}))
\rightarrow \mathbb{C}. \nonumber \end{equation}
Since
\begin{equation}Ext^{1}_{S}(\mathcal{O}_{Z},\mathcal{O}_{Z})\otimes Ext^{1}_{S}(\mathcal{O}_{Z},\mathcal{O}_{Z})\rightarrow
Ext^{2}_{S}(\mathcal{O}_{Z},\mathcal{O}_{Z})\rightarrow H^{2}(S,\mathcal{O}_{S})=0.\nonumber \end{equation}
$Ext^{1}_{S}(\mathcal{O}_{Z},\mathcal{O}_{Z})$ is a maximal isotropic subspace with respect to the Serre duality pairing.
\end{proof}
Hence, after taking away the trivial factor $H^{1}(S,\Omega_{S}^{1})$, the maximal isotropic sub-bundle of the reduced obstruction bundle exists and it can be identified with the tangent bundle of Hilbert scheme of points on $S$. Note that this gives a natural orientation on the self-dual trace-free obstruction bundle.
By Lemma \ref{ASD equivalent to max isotropic}, the reduced $DT_{4}$ virtual cycle is the Euler characteristic of Hilbert scheme of $n$-points on $S$.
\begin{theorem}\label{DT4 of cotangent bundle of S}
Let $X=T^{*}S$, where $S$ is a compact algebraic surface with $q(S)=0$. Assume (\ref{vanishing assumption}) which is satisfied when $S$ is del-Pezzo. Choose the natural complex orientation $o(\mathcal{O})$ on the self-dual trace-free obstruction bundle, then
\begin{equation}[\mathcal{M}_{c}^{S_{cpn}}]^{vir}_{red}=e(Hilb^{n}(S)) \nonumber \end{equation}
for $c=(1,0,-n)$, $n\geq 1$.

Furthermore, they fit into the following generating function
\begin{equation}\sum_{n\geq0}[\mathcal{M}_{(1,0,-n)}^{S_{cpn}}]^{vir}_{red}q^{n}=\prod_{k\geq1}\frac{1}{(1-q^{k})^{e(S)}}_{.} \nonumber \end{equation}
\end{theorem}
\begin{proof}
By the above discussions and \cite{cheah}.
\end{proof}

\newpage
\section{$DT_{4}$ invariants on toric $CY_{4}$ via localization}

In this section, we restrict to the moduli space of ideal sheaves of curves $I_{n}(X,\beta)$,
where $X$ is a toric Calabi-Yau four-fold.  $X$ admits a $(\mathbb{C}^{*})^{4}$-action which
can be naturally lifted to the moduli space. If we restrict to the three dimensional sub-torus $T\subseteq (\mathbb{C}^{*})^{4}$
which preserves the holomorphic top form of $X$, the action will also preserve the Serre duality pairing. \\

By the philosophy of virtual localization \cite{gp}, we will define the corresponding equivariant $DT_{4}$ invariants. Roughly speaking, we should have
\begin{equation}\int_{[\overline{\mathcal{M}}_{n,\beta}^{DT_{4}}(X)]^{vir}}\prod_{i=1}^{r}\gamma_{i}\thickapprox\sum_{[\mathcal{I}]
\in I_{n}(X,\beta)^{T}}\int_{[S(\mathcal{I})]^{vir}}\prod_{i=1}^{r}\gamma_{i}|_{\mathcal{I}}\cdot
\frac{\sqrt{e_{T}(Ext^{2}(\mathcal{I},\mathcal{I}))}}{e_{T}(Ext^{1}(\mathcal{I},\mathcal{I}))},
\nonumber \end{equation}
where $\overline{\mathcal{M}}_{n,\beta}^{DT_{4}}(X)$ denotes the undefined generalized $DT_{4}$ moduli space whose reduced structure
should be the same as the reduced structure of $I_{n}(X,\beta)$ and $\gamma_{i}$ are certain insertion fields we only
define on the right hand side. \\

Using a similar argument of Lemma 6, 8 in \cite{moop}, one can show that for $\mathcal{I}\in I_{n}(X,\beta)^{T}$ which is a $T$-fixed point, the $T$-representation
\begin{equation}Ext^{1}(\mathcal{I},\mathcal{I}), \quad Ext^{2}(\mathcal{I},\mathcal{I})  \nonumber \end{equation}
contain no trivial sub-representations.

Hence when we are reduced to the local contribution, we can get ride of the non-reduced structures and consider
everything on $I_{n}(X,\beta)$ instead of the generalized $DT_{4}$ moduli space. \\

For $\mathcal{I}\in I_{n}(X,\beta)^{T}$, form the following complex vector bundle whose fiber is $V_{\mathcal{I}}\triangleq Ext^{2}(\mathcal{I},\mathcal{I})$,
\begin{equation}
\begin{array}{lll}
      & \quad ET\times_{T}V_{\mathcal{I}}
      \\  &  \quad \quad \quad \downarrow \\   &   ET\times_{T}\{\mathcal{I}\}=BT.
\end{array}\nonumber\end{equation}
The Serre duality pairing naturally induces a non-degenerate pairing $Q_{Serre}$ on $ET\times_{T}V_{\mathcal{I}}$ because $T$ preserves the holomorphic top form. Thus, $(ET\times_{T}V_{\mathcal{I}},Q_{Serre})$ becomes a quadratic bundle.

By the theory of characteristic classes of quadratic bundles \cite{eg} (see also the next section), there exists a half Euler class $e(ET\times_{T}V_{\mathcal{I}},Q_{Serre})$ if $c_{1}(ET\times_{T}V_{\mathcal{I}})=0$.

\begin{definition}
For $\mathcal{I}\in I_{n}(X,\beta)^{T}$ with $c_{1}(ET\times_{T}V_{\mathcal{I}})=0$, we define
\begin{equation}e_{T}(Ext^{2}_{iso}(\mathcal{I},\mathcal{I}))\triangleq e(ET\times_{T}V_{\mathcal{I}},Q_{Serre})\in H^{*}_{T}(pt) \nonumber \end{equation}
to be the half Euler class of the above quadratic bundle \cite{eg}, where $Q_{Serre}$ denotes the induced Serre duality pairing.
\end{definition}

\begin{remark}
If $dim_{\mathbb{C}}V_{\mathcal{I}}$ is odd, the half Euler class is zero.
If $dim_{\mathbb{C}}V_{\mathcal{I}}$ is even, the half Euler class is unique up to a sign which is determined by an $SO(N,\mathbb{C})$ reduction of the structure group of the quadratic bundle
\begin{equation}(ET\times_{T}V_{\mathcal{I}},Q_{Serre})\rightarrow  ET\times_{T}\{\mathcal{I}\}. \nonumber \end{equation}
\end{remark}

Now let us define the $DT_{4}$ invariants for ideal sheaves of curves on toric $CY_{4}$ under the assumption that
$c_{1}(ET\times_{T}Ext^{2}(\mathcal{I},\mathcal{I}))=0$ for any $\mathcal{I}\in I_{n}(X,\beta)^{T}$.
Because we do not have Seidel-Thomas twist for toric $CY_{4}$, we do not know how to give a compatible orientation for different components of $ET\times_{T}I_{n}(X,\beta)^{T}$. We just arbitrarily choose an orientation for each component at the moment.
\begin{definition}
Assume for any $\mathcal{I}\in I_{n}(X,\beta)^{T}$, $c_{1}(ET\times_{T}V_{\mathcal{I}})=0$.
The toric orientation data is defined to be a choice of $SO(N,\mathbb{C})$ reduction of the structure group of the quadratic bundle
\begin{equation}(ET\times_{T}V_{\mathcal{I}},Q_{Serre})\rightarrow  ET\times_{T}\{\mathcal{I}\}, \nonumber \end{equation}
for each $\mathcal{I}\in I_{n}(X,\beta)^{T}$.
\end{definition}
\begin{definition}
Given $[\mathcal{I}]\in I_{n}(X,\beta)^{T}$, $P\in \mathbb{Z}[x_{1},x_{2},...,x_{m}]$ and $\gamma\in H^{*}_{T}(X,\mathbb{Z})$, define
\begin{equation}
\pi_{*}([\mathcal{I}],P,\gamma)\triangleq\pi_{*}\Big(P\big(c_{i}^{T}(\mathfrak{I}|_{[\mathcal{I}]\times X})\big)\cup \gamma\Big)\in H^{*}_{T}(pt),  \nonumber \end{equation}
where $\pi_{*}: H^{*}_{T}(X)\rightarrow H^{*}_{T}(pt)$ is the equivariant push-forward and $\mathfrak{I}\rightarrow I_{n}(X,\beta)\times X$ is the universal ideal sheaf.
\end{definition}

\begin{definition}
Given a toric Calabi-Yau four-fold $X$, $\beta\in H_{2}(X,\mathbb{Z})$, $n\in \mathbb{N}_{+}$, polynomials
$P_{i}\in \mathbb{Z}[x_{1},x_{2},...,x_{m_{i}}]$, insertion fields $\gamma_{i}\in H^{*}_{T}(X,\mathbb{Z})$ for $i=1,2,...,r$ and a toric orientation data,
the equivariant $DT_{4}$ invariant for ideal sheaves of curves associated with the above data is defined to be
\begin{equation}Z_{DT_{4}}\big(X,n\big|(P_{1},\gamma_{1}),...,(P_{r},\gamma_{r})\big)_{\beta}=\sum_{[\mathcal{I}]\in I_{n}(X,\beta)^{T}}
\bigg(\prod_{i=1}^{r}\big(\pi_{*}([\mathcal{I}],P_{i},\gamma_{i})\big)\bigg)\cdot
\frac{e_{T}(Ext^{2}_{iso}(\mathcal{I},\mathcal{I}))}{e_{T}(Ext^{1}(\mathcal{I},\mathcal{I}))},    \nonumber \end{equation}
where we assume $c_{1}^{T}(Ext^{2}(\mathcal{I},\mathcal{I}))=0$, for any $\mathcal{I}\in I_{n}(X,\beta)^{T}$ and the sign in $e_{T}(Ext^{2}_{iso}(\mathcal{I},\mathcal{I}))$ is compatible with the chosen toric orientation data.
\end{definition}
\begin{remark} ${}$ \\
$1$. The insertion fields in the above definition is just one plausible choice. If we want to match it with the corresponding $GW$ invariants, we may need some adjustment as toric 3-folds cases \cite{moop}. \\
$2$. As we can see
\begin{equation}\Bigg(\frac{e_{T}(Ext^{2}_{iso}(\mathcal{I},\mathcal{I}))}{e_{T}(Ext^{1}(\mathcal{I},\mathcal{I}))}\Bigg)^{2}
=\pm\frac{e_{T}(Ext^{2}(\mathcal{I},\mathcal{I}))}{e_{T}(Ext^{3}(\mathcal{I},\mathcal{I}))e_{T}(Ext^{1}(\mathcal{I},\mathcal{I}))}.
\nonumber \end{equation}
Meanwhile the right hand side has a generalization to any dimensional smooth toric variety. We do not know whether
there is any counting meaning for the right hand side for arbitrary dimensional smooth toric variety since
we do not have the corresponding virtual theory for sheaves counting at the moment. \\
$3$. If the base manifold is an algebraic surface, the moduli space of torsion-free stable sheaves is smooth with vanishing obstruction spaces at
least for $K3$ and del-Pezzo surfaces. The Euler characteristic of the moduli is automatically a virtual invariant.

For Calabi-Yau threefolds, if $\mathcal{M}_{c}$ is smooth, the
$DT_{3}$ invariant is the Euler characteristic of the moduli space up to a sign.

In the above two cases, the Euler characteristic
of the moduli space in some sense represents the sheaves virtual counting which is the partition function of certain topological quantum field theory.

But for $CY_{4}$, the Euler characteristic can not reflect the corresponding $DT_{4}$ invariants in general.
We should consider the Euler class of the self-dual obstruction bundle and the quadratic form takes its role in the invariants.
Maybe this is one of the reasons that we did not get a closed formula for the generating function of the Euler characteristics
of Hilbert schemes of points on a complex four-fold \cite{cheah}. We are wondering whether there is a closed formula for the generating function of
$DT_{4}$ invariants for ideal sheaves of points.
\end{remark}

\newpage
\section{Computational examples}
We will give some computational examples of $DT_{4}$ invariants in this section.
Our examples will concentrate on the case when $\mathcal{M}_{c}$ is smooth.
By definition \ref{virtual cycle when Mc smooth}, the virtual fundamental class of $\overline{\mathcal{M}}_{c}^{DT_{4}}$ is just the
Poincar\'{e} dual of the Euler class of the self-dual obstruction bundle.
In this case, the virtual cycle has its origin in the theory of characteristic classes \cite{eg}.

Let us recall the characteristic class theory for complex vector bundles with quadratic forms, say $(V,q)$ on
a projective manifold $X$ \cite{switzer}. We assume the structure group of the bundle is $SO(n,\mathbb{C})$,
$n$ is the rank of $V$.
By homotopy equivalence:
\begin{equation} SO(n,\mathbb{C})\sim SO(n,\mathbb{R}),\nonumber\end{equation}
we have
\begin{equation} H^{*}(BSO(n,\mathbb{C});\mathbb{Q})\cong H^{*}(BSO(n,\mathbb{R});\mathbb{Q}).\nonumber\end{equation}
Meanwhile
\begin{equation}H^{*}(BSO(2r,\mathbb{C});\mathbb{Q})\cong \mathbb{Q}[p_{1},...p_{r-1},e],\nonumber\end{equation}
\begin{equation}H^{*}(BSO(2r+1,\mathbb{C});\mathbb{Q})\cong \mathbb{Q}[p_{1},...p_{r}],\nonumber\end{equation}
where $p_{i}=(-1)^{i}c_{2i}(\omega^{+}\otimes \mathbb{C})$, $\omega^{+}$ is the universal $SO(n)$ bundle and
$e=e(\omega^{+})$ is called the half Euler class of $(V,q)$.

When $\mathcal{M}_{c}$ is smooth and $(V,q)=(ob,Q_{Serre})$, where $ob$ is the obstruction bundle with $ob|_{\mathcal{F}}=Ext^{2}(\mathcal{F},\mathcal{F})$, $Q_{Serre}$ is the
Serre duality pairing, then $\overline{\mathcal{M}}_{c}^{DT_{4}}$ exists and $\overline{\mathcal{M}}_{c}^{DT_{4}}=\mathcal{M}_{c}$. Furthermore $[\overline{\mathcal{M}}_{c}^{DT_{4}}]^{vir}=PD(e)$ with the chosen orientation.

By Lemma \ref{ASD equivalent to max isotropic}, we will use the top Chern class of the maximal isotropic subbundle (if it exists) to compute the half Euler class $e$.

\subsection{$DT_{4}/GW$ correspondence in some special cases }
As is well known \cite{briDTPT} \cite{moop} \cite{pandpixton} \cite{toda}, we have equivalence between
Donaldson-Thomas ideal sheaves invariants and Gromov-Witten invariants on Calabi-Yau three-folds. One may expect,
there will be a similar gauge-string duality for Calabi-Yau four-folds. We will study such a correspondence in some very special cases.

Fix $c=1-PD(\beta)-nPD(1)$,
where $\beta\in H_{2}(X,\mathbb{Z})$ is a fixed curve class and $n=\chi (\mathcal{O}_{C})$, $\mathcal{O}_{C}$ is the structure sheaf of a curve $C$.

We first compute the real virtual dimension of $\overline{\mathcal{M}}_{c}^{DT_{4}}$ which is defined to be $2ext^{1}(I_{C},I_{C})-ext^{2}(I_{C},I_{C})$, $I_{C}\in \mathcal{M}_{c}$.
\begin{lemma}\label{v.d of ideal sheaves of curves}The real virtual dimension of the generalized $DT_{4}$ moduli space $\overline{\mathcal{M}}_{c}^{DT_{4}}$
with $c=ch(I_{C})=1-PD(\beta)-nPD(1)$ on a compact Calabi-Yau four-fold $X$ satisfies
\begin{equation}\textrm{ } v.d_{\mathbb{R}}(\overline{\mathcal{M}}_{c}^{DT_{4}})=2n,  \quad \quad \textrm{ } \textrm{ }\textmd{if} \quad Hol(X)=SU(4), \nonumber \end{equation}
\begin{equation}v.d_{\mathbb{R}}(\overline{\mathcal{M}}_{c}^{DT_{4}})=2n-1, \quad \textmd{if} \quad  Hol(X)=Sp(2). \nonumber \end{equation}
\end{lemma}
\begin{proof}
By the Hirzebruch-Riemann-Roch theorem,
\begin{equation}\chi(I_{C},I_{C})
=2\int_{X}ch_{4}(I_{C})+ 2h^{0}(X,\mathcal{O}_{X})+h^{2}(X,\mathcal{O}_{X}).
\nonumber \end{equation}
Thus, for $Hol(X)=Sp(2)$, we have
\begin{equation}ext^{1}(I_{C},I_{C})-\frac{1}{2}ext^{2}(I_{C},I_{C})=n-\frac{1}{2},  \nonumber \end{equation}
\begin{equation}v.d_{\mathbb{R}}(\overline{\mathcal{M}}_{c}^{DT_{4}})=2n-1 .\nonumber \end{equation}
Similarly, we have the result when $Hol(X)=SU(4)$.
\end{proof}
\begin{remark} ${}$ \\
1. The generalized $DT_{4}$ moduli space of ideal sheaves of subschemes is not always defined since
it depends on the gluing assumption \ref{assumption on gluing}.
However, we still define its virtual dimension as stated above.
At least in the case when $\mathcal{M}_{c}$ is smooth, $\overline{\mathcal{M}}_{c}^{DT_{4}}$ exists and $\overline{\mathcal{M}}_{c}^{DT_{4}}=\mathcal{M}_{c}$. \\
2. As we know, $v.d_{\mathbb{R}}(\overline{\mathcal{M}}_{c}^{DT_{4}})=2n=2(1-g)$ if $Hol(X)=SU(4)$. The $DT_{4}$ invariants for ideal sheaves of curves with $g\geq 2$ vanish which coincides with the situation of Gromov-Witten invariants.
\end{remark}

Now let us come to the $DT_{4}/GW$ correspondence when $\mathcal{M}_{c}$ is smooth and $Hol(X)=SU(4)$ or $Sp(2)$.

\subsubsection{The case of $Hol(X)=SU(4)$ }
We start with one dimensional closed subschemes of $X$ with $Hol(X)=SU(4)$.
\begin{lemma}\label{DT GW 1} Let $X$ be a compact Calabi-Yau four-fold with $Hol(X)=SU(4)$.
Let $C\hookrightarrow X$ be a closed subscheme with $dim_{\mathbb{C}}C\leq 1$ and $H^{1}(X,\mathcal{O}_{C})=0$,
where $\mathcal{O}_{C}$ is the structure sheaf of $C$.
Then we have canonical isomorphisms
\begin{equation}Ext^{i}(I_{C},I_{C})\cong Ext^{i}(\mathcal{O}_{C},\mathcal{O}_{C}), \quad i=1,2, \nonumber \end{equation}
where $I_{C}$ is the ideal sheaf of $C$ in $X$.
\end{lemma}
\begin{proof}
Taking cohomology of the following short exact sequence of sheaves
\begin{equation}0\rightarrow I_{C}\rightarrow \mathcal{O}_{X}\rightarrow \mathcal{O}_{C}\rightarrow 0,  \nonumber \end{equation}
we get
\begin{equation}\rightarrow H^{i}(X,\mathcal{O}_{X})\rightarrow H^{i}(X,\mathcal{O}_{C})
\rightarrow  H^{i+1}(X,I_{C})\rightarrow.   \nonumber \end{equation}
We have $H^{i}(X,\mathcal{O}_{X})=0$ for $i=1,2,3$ by $Hol(X)=SU(4)$.
$H^{i}(X,\mathcal{O}_{C})=0$ for $i=2,3,4$, because $dim_{\mathbb{C}}\mathcal{O}_{C}\leq1$. Thus
\begin{equation}\label{DT GW equation 0} H^{1}(X,\mathcal{O}_{C})\cong H^{2}(X,I_{C}), H^{3}(X,I_{C})=0,
H^{4}(X,I_{C})\cong H^{4}(X,\mathcal{O}_{X}). \end{equation}
Taking $Hom(\mathcal{O}_{C}, \cdot)$ to $0\rightarrow I_{C}\rightarrow \mathcal{O}_{X}\rightarrow \mathcal{O}_{C}\rightarrow 0$, we have
\begin{equation}\rightarrow Ext^{i}(\mathcal{O}_{C},\mathcal{O}_{X})\rightarrow
Ext^{i}(\mathcal{O}_{C},\mathcal{O}_{C})\rightarrow Ext^{i+1}(\mathcal{O}_{C},I_{C})\rightarrow.  \nonumber \end{equation}
By Serre duality, $Ext^{i}(\mathcal{O}_{C},\mathcal{O}_{X})\cong Ext^{4-i}(\mathcal{O}_{X},\mathcal{O}_{C})=0$ if $i=0,1,2$. Hence
\begin{equation}\label{DT GW equation 1} Ext^{0}(\mathcal{O}_{C},I_{C})=0, Ext^{k}(\mathcal{O}_{C},\mathcal{O}_{C})
\cong Ext^{k+1}(\mathcal{O}_{C},I_{C}), k=0,1. \end{equation}
Meanwhile $Ext^{3}(\mathcal{O}_{C},\mathcal{O}_{X})\cong H^{1}(X,\mathcal{O}_{C})=0$ by our assumption, thus
\begin{equation}\label{DT GW equation 2} Ext^{2}(\mathcal{O}_{C},\mathcal{O}_{C})\cong Ext^{3}(\mathcal{O}_{C},I_{C}). \end{equation}
Using
\begin{equation}H^{0}(X,\mathcal{O}_{C})\cong Ext^{4}(\mathcal{O}_{C},\mathcal{O}_{X})\twoheadrightarrow
Ext^{4}(\mathcal{O}_{C},\mathcal{O}_{C})\cong Ext^{0}(\mathcal{O}_{C},\mathcal{O}_{C}) \nonumber \end{equation}
and the above long exact sequence, we have
\begin{equation}\label{DT GW equation 3} Ext^{1}(\mathcal{O}_{C},\mathcal{O}_{C})\cong
Ext^{3}(\mathcal{O}_{C},\mathcal{O}_{C})\cong Ext^{4}(\mathcal{O}_{C},I_{C}). \end{equation}
Applying the functor $Hom(\cdot,I_{C})$ to $0\rightarrow I_{C}\rightarrow \mathcal{O}_{X}\rightarrow \mathcal{O}_{C}\rightarrow 0$, we get
\begin{equation}\rightarrow Ext^{i}(\mathcal{O}_{X},I_{C})\rightarrow Ext^{i}(I_{C},I_{C})\rightarrow
Ext^{i+1}(\mathcal{O}_{C},I_{C})\rightarrow .\nonumber \end{equation}
By (\ref{DT GW equation 0}), we know $H^{2}(X,I_{C})\cong H^{1}(X,\mathcal{O}_{C})=0$ and $H^{3}(X,I_{C})=0$. Hence
\begin{equation}\label{DT GW equation 4}Ext^{2}(I_{C},I_{C})\cong Ext^{3}(\mathcal{O}_{C},I_{C}).  \end{equation}
The remaining sequence of the long exact sequence is
\begin{equation}\label{DT GW equation 5} Ext^{3}(I_{C},I_{C})\cong Ext^{4}(\mathcal{O}_{C},I_{C}) , \end{equation}
because $H^{4}(X,\mathcal{O}_{X})\cong Ext^{4}(\mathcal{O}_{X},I_{C})\twoheadrightarrow Ext^{4}(I_{C},I_{C})$ and they have the same dimensions.

By (\ref{DT GW equation 2}),(\ref{DT GW equation 3}),(\ref{DT GW equation 4}),(\ref{DT GW equation 5}),
we have canonically identified the deformation and obstruction spaces of the structure sheaf and ideal sheaf of $C$ in $X$, i.e.
\begin{equation}Ext^{i}(I_{C},I_{C})\cong Ext^{i}(\mathcal{O}_{C},\mathcal{O}_{C}), \quad i=1,2. \nonumber \end{equation}
\end{proof}
If we further assume $C$ is a connected smooth imbedded curve inside $X$, we have
\begin{lemma}\label{DT GW 2}  If $C$ is a connected genus zero smooth imbedded curve inside $X$, then we have canonical isomorphisms
\begin{equation}Ext^{1}(I_{C},I_{C})\cong H^{0}(C,\mathcal{N}_{C/X}), \nonumber \end{equation}
\begin{equation}Ext^{2}(I_{C},I_{C})\cong H^{1}(C,\mathcal{N}_{C/X})\oplus H^{1}(C,\mathcal{N}_{C/X})^{*}, \nonumber \end{equation}
where $\mathcal{N}_{C/X}$ is the normal bundle of $C$.
Furthermore, under this identification, $H^{1}(C,\mathcal{N}_{C/X})$ is a maximal isotropic subspace of
$Ext^{2}(I_{C},I_{C})$ with respect to the Serre duality pairing.
\end{lemma}
\begin{proof}
We have the local to global spectral sequence which degenerates at $E_{2}$ terms
\begin{eqnarray*}Ext^{*}_{X}(\mathcal{O}_{C},\mathcal{O}_{C})&\cong&
H^{*}(X,\mathcal{E}xt^{*}_{X}(\mathcal{O}_{C},\mathcal{O}_{C})) \\
&\cong& H^{*}(X,\iota_{*}(\wedge^{*}\mathcal{N}_{C/X})) \\
&\cong& H^{*}(C,\wedge^{*}\mathcal{N}_{C/X}),
\end{eqnarray*}
where $\iota: C\hookrightarrow X$ is the imbedding map. Then we have
\begin{equation}Ext^{1}(\mathcal{O}_{C},\mathcal{O}_{C})\cong H^{0}(C,\mathcal{N}_{C/X})\oplus H^{1}(C,\mathcal{O}_{C}), \nonumber \end{equation}
\begin{equation}Ext^{2}(\mathcal{O}_{C},\mathcal{O}_{C})\cong H^{0}(C,\wedge^{2}\mathcal{N}_{C/X})\oplus H^{1}(C,\mathcal{N}_{C/X}).
\nonumber \end{equation}
By $H^{1}(X,\mathcal{O}_{C})=0$ and the perfect pairing
\begin{equation} \wedge^{2}\mathcal{N}_{C/X}\otimes\mathcal{N}_{C/X}\rightarrow \wedge^{3}\mathcal{N}_{C/X}\cong \Omega_{C}.\nonumber \end{equation}
We have canonical isomorphisms using Serre duality \cite{hart}
\begin{equation}Ext^{1}(I_{C},I_{C})\cong H^{0}(C,\mathcal{N}_{C/X}), \nonumber \end{equation}
\begin{equation}Ext^{2}(I_{C},I_{C})\cong H^{1}(C,\mathcal{N}_{C/X})^{*}\oplus H^{1}(C,\mathcal{N}_{C/X}). \nonumber \end{equation}
As for the pairing, we have
\begin{equation}Ext^{2}(I_{C},I_{C})\otimes Ext^{2}(I_{C},I_{C})\rightarrow Ext^{4}(I_{C},I_{C})\rightarrow H^{4}(X,\mathcal{O}_{X}),
\nonumber \end{equation}
where the last trace map is a canonical isomorphism. The above pairing can be identified with the pairing
\begin{equation}(H^{0}(C,\wedge^{2}\mathcal{N}_{C/X})\oplus H^{1}(C,\mathcal{N}_{C/X}))\otimes
(H^{0}(C,\wedge^{2}\mathcal{N}_{C/X})\oplus H^{1}(C,\mathcal{N}_{C/X})), \nonumber\end{equation}
\begin{equation}\rightarrow H^{1}(C,\wedge^{3}\mathcal{N}_{C/X})\cong\mathbb{C}.\nonumber\end{equation}
Meanwhile
\begin{equation} H^{1}(C,\mathcal{N}_{C/X})\otimes H^{1}(C,\mathcal{N}_{C/X})\rightarrow 0. \nonumber\end{equation}
Hence we know $H^{1}(C,\mathcal{N}_{C/X})$ is a maximal isotropic subspace under the above identification.
\end{proof}
We have canonically identified the deformation and obstruction spaces of $DT_{4}$ theory and $GW$ theory in the above case.
If we further assume $\mathcal{M}_{c}$ is smooth and the $GW$ moduli space can be identified with $\mathcal{M}_{c}$,
then we have the following $DT_{4}/GW$ correspondence by Definition \ref{virtual cycle when Mc smooth} and Lemma \ref{ASD equivalent to max isotropic}.
\begin{theorem}\label{DT=GW}
Let $X$ be a compact Calabi-Yau four-fold with $Hol(X)=SU(4)$. If $\mathcal{M}_{c}$ with given Chern character
$c=(1,0,0,-PD(\beta),-1) $
is smooth and consists of ideal sheaves of smooth connected genus zero imbedded curves only.
Assume the $GW$ moduli space $\overline{\mathcal{M}}_{0,0}(X,\beta)\cong \mathcal{M}_{c}$ and use the natural complex orientation $o(\mathcal{O})$. Then $\overline{\mathcal{M}}_{c}^{DT_{4}}$ exists, $\overline{\mathcal{M}}_{c}^{DT_{4}}\cong\overline{\mathcal{M}}_{0,0}(X,\beta)$ and  $[\overline{\mathcal{M}}_{c}^{DT_{4}}]^{vir}=[\overline{\mathcal{M}}_{0,0}(X,\beta)]^{vir}$.
\end{theorem}
\begin{remark}
If there exists an embedding $i: \mathbb{C}\mathbb{P}^{3}\hookrightarrow X$
and $\beta=i_{*}[H^{2}]$, where $[H^{2}]\in H_{2}(\mathbb{C}\mathbb{P}^{3})$ is the generator.
By the identification of the deformation spaces and the negativity of the normal bundle $\mathcal{N}_{\mathbb{P}^{3}/X}$,
the ideal sheaves of curves in class $\beta$ are always of the form $I_{C}$, where $C\hookrightarrow \mathbb{C}\mathbb{P}^{3}\hookrightarrow X$.

Meanwhile $\mathcal{M}_{c}$ with $c=ch(I_{C})=(1,0,0,-PD(\beta),-1)$,
$\beta=i_{*}[H^{2}]\in H_{2}(X)$ is smooth. Hence the conditions in Theorem \ref{DT=GW} are satisfied.
\end{remark}
Then we explicitly compute the $DT_{4}$ invariant in this case.
\begin{proposition}
Let $c=(1,0,0,-PD(\beta),-1)\in H^{even}(X,\mathbb{Z})$.
If there exists an embedding $i: \mathbb{P}^{3}\hookrightarrow X$ and $\beta=i_{*}[H^{2}]$,
then $\overline{\mathcal{M}}_{c}^{DT_{4}}$ exists and $\overline{\mathcal{M}}_{c}^{DT_{4}}\cong \overline{\mathcal{M}}_{0,0}(K_{\mathbb{P}^{3}},[H^{2}])$. Furthermore, $[\overline{\mathcal{M}}_{c}^{DT_{4}}]^{vir}=[\overline{\mathcal{M}}_{0,0}(K_{\mathbb{P}^{3}},[H^{2}])]^{vir}$.
\end{proposition}
\begin{proof}
By the positivity of $TC\cong T\mathbb{P}^{1}$, we have $H^{1}(C,\mathcal{N}_{C/X})\cong H^{1}(C,\iota^{*}TX)$, where
$\iota: C\hookrightarrow X$.
Meanwhile,
\begin{equation}0\rightarrow T\mathbb{P}^{3}\rightarrow TX\mid_{\mathbb{P}^{3}}\rightarrow \mathcal{N}_{\mathbb{P}^{3}/X}\rightarrow 0
\nonumber \end{equation}
induces $H^{1}(C,\iota^{*}TX)\cong H^{1}(C,\iota^{*}K_{\mathbb{P}^{3}})$
which is the obstruction space of Gromov-Witten theory of $K_{\mathbb{P}^{3}}$ by
$H^{1}(C,\iota^{*}TK_{\mathbb{P}^{3}})\cong H^{1}(C,\iota^{*}K_{\mathbb{P}^{3}})$.
\end{proof}

\subsubsection{The case of $Hol(X)=Sp(2)$ }
When $Hol(X)=Sp(2)$, we similarly have
\begin{lemma}Let $X$ be a compact irreducible hyper-K\"ahler four-fold.
Under the assumption that $C\hookrightarrow X$ is a closed subscheme with $dim_{\mathbb{C}}C\leq 1$ and $H^{1}(X,\mathcal{O}_{C})=0$, we have
\begin{equation}Ext^{1}(I_{C},I_{C})\cong Ext^{1}(\mathcal{O}_{C},\mathcal{O}_{C}), \nonumber \end{equation}
\begin{equation}0\rightarrow H^{2}(X,\mathcal{O}_{X})\rightarrow Ext^{2}(I_{C},I_{C})\rightarrow Ext^{2}(\mathcal{O}_{C},\mathcal{O}_{C})\rightarrow 0 .\nonumber \end{equation}
\end{lemma}
\begin{proof}By a similar argument as the case of $Hol(X)=SU(4)$.
\end{proof}
\begin{lemma}If $C$ is a connected genus zero smooth imbedded curve inside $X$, then we have
\begin{equation}Ext^{1}(I_{C},I_{C})\cong H^{0}(C,\mathcal{N}_{C/X}), \nonumber \end{equation}
\begin{equation}0\rightarrow H^{2}(X,\mathcal{O}_{X})\rightarrow Ext^{2}(I_{C},I_{C})\rightarrow H^{1}(C,\mathcal{N}_{C/X})\oplus H^{1}(C,\mathcal{N}_{C/X})^{*} \rightarrow 0. \nonumber \end{equation}
\end{lemma}
Note the $GW$ obstruction space $H^{1}(C,\mathcal{N}_{C/X})\cong H^{1}(C,\iota^{*}TX)$, if $C\cong \mathbb{P}^{1}$.
Meanwhile, the short exact sequence
\begin{equation}0\rightarrow \mathcal{N}_{C/X}^{*}\rightarrow \iota^{*}\Omega_{X}\rightarrow \Omega_{C}\rightarrow 0
\nonumber \end{equation}
induces
\begin{equation}0\cong H^{0}(C,\Omega_{C})\rightarrow H^{1}(C,\mathcal{N}_{C/X}^{*})\rightarrow
H^{1}(C,\iota^{*}\Omega_{X})\cong H^{1}(C,\iota^{*}TX)\rightarrow H^{1}(C,\Omega_{C})\rightarrow 0,
\nonumber \end{equation}
where the second isomorphism is by the existence of global nonzero holomorphic symplectic form $\sigma\in H^{0}(X,\Omega^{2}_{X})$.
The above sequence establishes the surjective cosection of the obstruction sheaf of $GW$ theory for hyper-K\"ahler manifolds \cite{kiem0}, \cite{kiem}.

As vector space,
\begin{equation}\label{hyper cosection}Ext^{2}(I_{C},I_{C})\cong H^{1}(C,\mathcal{N}_{C/X}^{*})\oplus H^{1}(C,\Omega_{C})
\oplus H^{1}(C,\mathcal{N}_{C/X}^{*})^{*}\oplus H^{1}(C,\Omega_{C})^{*}\oplus H^{2}(X,\mathcal{O}_{X}). \end{equation}
Here the dimension of the trivial factors in the $DT_{4}$ obstruction space is bigger than one.
Taking away all the trivial factors and restrict to the maximal isotropic subspace, the hyper-reduced $DT_{4}$ obstruction space is defined to be
\begin{equation}Ext^{2}_{\emph{hyper-red}}(I_{C},I_{C})\triangleq H^{1}(C,\mathcal{N}_{C/X}^{*}) \nonumber \end{equation}
which coincides with the reduced $GW$ obstruction space \cite{kiem}.

We define the hyper-reduced virtual fundamental class of $\overline{\mathcal{M}}_{c}^{DT_{4}}$ (denoted by $[\overline{\mathcal{M}}_{c}^{DT_{4}}]^{vir}_{hyper-red}$) to be the Poincar\'{e} dual of the Euler class of the
hyper-reduced obstruction bundle.
\begin{theorem}\label{DT=GW2}
Let $X$ be a compact irreducible hyper-K\"ahler four-fold. If $\mathcal{M}_{c}$ with given Chern character $c=(1,0,0,-PD(\beta),-1)$
is smooth and consists of ideal sheaves of smooth connected genus zero imbedded curves only. Assume the $GW$ moduli space
$\overline{\mathcal{M}}_{0,0}(X,\beta)\cong \mathcal{M}_{c}$ and use the natural complex orientation $o(\mathcal{O})$. Then $\overline{\mathcal{M}}_{c}^{DT_{4}}$ exists, $\overline{\mathcal{M}}_{c}^{DT_{4}}\cong\overline{\mathcal{M}}_{0,0}(X,\beta)$
and $[\overline{\mathcal{M}}_{c}^{DT_{4}}]^{vir}=0$.

Furthermore, $[\overline{\mathcal{M}}_{c}^{DT_{4}}]^{vir}_{hyper-red}=[\overline{\mathcal{M}}_{0,0}(X,\beta)]^{vir}_{red}$, where
$[\overline{\mathcal{M}}_{0,0}(X,\beta)]^{vir}_{red}$ is the reduced virtual fundamental class of the $GW$ moduli space \cite{kiem}.
\end{theorem}
A simple application is the following result due to Mukai \cite{mukai}.
\begin{corollary}$\mathbb{P}^{3}$ can't be embedded into any compact irreducible hyper-K\"ahler four-fold. \end{corollary}
\begin{proof} Taking away the trivial factor $H^{2}(X,\mathcal{O}_{X})$
and then restrict to the maximal isotropic subspace, we have
\begin{eqnarray*}Ext^{2}_{red}(I_{C},I_{C})&\triangleq&H^{1}(C,\mathcal{N}_{C/X}) \\
&=& H^{1}(C,\iota^{*}TX) \\
&=& H^{1}(C,\iota^{*}K_{\mathbb{P}^{3}}).
\end{eqnarray*}
Meanwhile, the Euler class of the $GW$ obstruction bundle whose fiber is $H^{1}(C,\iota^{*}K_{\mathbb{P}^{3}})$ can not be trivial by localization calculation \cite{klempand}.

If $\mathbb{P}^{3}\hookrightarrow X$, where $X$ is a compact irreducible hyper-K\"ahler four-fold,
we have a extra trivial factor $H^{1}(C,\Omega_{C})$ in the $DT_{4}$ obstruction space by the above construction which leads to the vanishing of the virtual cycle.
\end{proof}

\subsection{Some remarks on cosection localizations for hyper-K\"ahler four-folds}
In this subsection, we assume $X$ to be a compact hyper-K\"ahler four-fold with holomorphic symplectic two form $\sigma\in H^{0}(X,\Omega_{X}^{2})$.

There is an obvious surjective cosection map of the obstruction sheaf of $\mathcal{M}_{c}$,
\begin{equation}\nu: ob \twoheadrightarrow \mathcal{O}_{\mathcal{M}_{c}},   \nonumber \end{equation}
given by the trace map:
\begin{equation}\nu: Ext^{2}(\mathcal{F},\mathcal{F})\rightarrow H^{2}(X,\mathcal{O}_{X})\cong \mathbb{C} .\nonumber \end{equation}
Then we have the surjective cosection map for $DT_{4}$ obstruction space
\begin{equation}\label{nu+}\nu_{+}:Ext^{2}_{+}(\mathcal{F},\mathcal{F})\twoheadrightarrow H^{2}_{+}(X,\mathcal{O}_{X})\cong \mathbb{R}, \end{equation}
which leads to the vanishing of the virtual fundamental class of (generalized) $DT_{4}$ moduli spaces. \\

If we fix the determinant of the torsion-free sheaf, there is a less obvious cosection map \cite{bflenner}
\begin{equation}\nu_{hyper}: Ext^{2}_{0}(\mathcal{F},\mathcal{F})\twoheadrightarrow H^{4,4}(X), \nonumber \end{equation}
defined to be the composition of
\begin{equation} \xymatrix@1{
Ext^{2}_{0}(\mathcal{F},\mathcal{F})\ar[r]^{\cdot\frac{(At(\mathcal{F}))^{2}}{2}\quad} & Ext^{4}(\mathcal{F},\mathcal{F}\otimes\Omega_{X}^{2})\ar[r]^{\quad tr}
& H^{4}(X,\Omega_{X}^{2})\ar[r]^{\wedge \sigma} & H^{4,4}(X) } . \nonumber \end{equation}
Similar to the case of \cite{maulik}, one can show that $\nu_{hyper}$ is surjective if $ch_{3}(\mathcal{F})\neq0$. But it does not
factor through $Ext^{2}_{+}(\mathcal{F},\mathcal{F})$ and in general does not give a surjective cosection map of the trace-free $DT_{4}$ obstruction space.

Recall that in the subsection of $DT_{4}/GW$ correspondence, we have established a surjective cosection map for the trace-free $DT_{4}$ obstruction space (\ref{hyper cosection}) which turns out to be the same as the cosection map of $GW$ theory for hyper-K\"ahler four-folds.

\subsection{Li-Qin's examples}
Actually, we have examples when $\overline{\mathcal{M}}_{c}=\mathcal{M}_{c}=\mathcal{M}_{c}^{o}$ on certain Calabi-Yau four-folds.
We will study the examples when the rank $2$ bundles always come from non-trivial extensions of two line bundles with specific Chern class.
The construction is due to W. P. Li and Z.Qin \cite{lq}.

Let $X$ be a generic smooth hyperplane section in $W=\mathbb{P}^{1}\times\mathbb{P}^{4}$ of bi-degree $(2,5)$ which is a Calabi-Yau four-fold. Let
\begin{equation}cl=[1+(-1,1)|_{X}]\cdot[1+(\epsilon_{1}+1,\epsilon_{2}-1)|_{X}],\nonumber \end{equation}
\begin{equation}k=(1+\epsilon_{1})\left(\begin{array}{l}6-\epsilon_{2} \\ \quad 4\end{array}\right), \quad \epsilon_{1},\epsilon_{2}=0,1.
\nonumber\end{equation}
Define $\overline{\mathcal{M}}_{c}(L_{r})$ to be the moduli space of Gieseker $L_{r}-$semistable rank-2 torsion-free sheaves
with Chern character $c$ (which can be easily read from the total Chern class $cl$), where $L_{r}=\mathcal{O}_{W}(1,r)|_{X}$.

Then by Theorem 5.7 \cite{lq}, we have
\begin{lemma}(Li-Qin)\cite{lq}\label{lq example} \\
The moduli space of rank two bundles on $X$ with the given Chern class stated above satisfies the following properties:

(i) The moduli space is isomorphic to $\mathbb{P}^{k}$ and consists of all the rank-2 bundles in the nonsplitting extensions
\begin{equation}0\rightarrow \mathcal{O}_{X}(-1,1)\rightarrow E\rightarrow \mathcal{O}_{X}(\epsilon_{1}+1,\epsilon_{2}-1)\rightarrow 0,
\nonumber\end{equation}
when
\begin{equation}\frac{15(2-\epsilon_{2})}{6+5\epsilon_{1}+2\epsilon_{2}}<r<\frac{15(2-\epsilon_{2})}{\epsilon_{1}(1+2\epsilon_{2})}.
\nonumber\end{equation}
(ii) $\overline{\mathcal{M}}_{c}(L_{r})$  is empty when
\begin{equation} 0<r<\frac{15(2-\epsilon_{2})}{6+5\epsilon_{1}+2\epsilon_{2}}.\nonumber\end{equation}
\end{lemma}
By the Hirzebruch-Riemann-Roch theorem, we can get
\begin{equation}\epsilon_{1}=0,\textrm{ }\epsilon_{2}=1 \Rightarrow \chi(E,E)=-6, \textrm{ }\textrm{ }k=4 ,\nonumber\end{equation}
\begin{equation}\textrm{ }\epsilon_{1}=1,\textrm{ }\epsilon_{2}=1 \Rightarrow \chi(E,E)=-16, \textrm{ } k=9, \nonumber\end{equation}
\begin{equation}\textrm{ }\textrm{ }\epsilon_{1}=0,\textrm{ }\epsilon_{2}=0 \Rightarrow \chi(E,E)=-26,\textrm{ } k=14, \nonumber\end{equation}
\begin{equation}\textrm{ }\textrm{ }\epsilon_{1}=1,\textrm{ }\epsilon_{2}=0 \Rightarrow \chi(E,E)=-56,\textrm{ } k=29 ,\nonumber\end{equation}
By Lemma 5.2 \cite{lq}, $k=dim Ext^{1}(E,E)$. By simple computations we have $Ext^{2}(E,E)=0$ in all the above four cases.

Thus $\overline{\mathcal{M}}_{c}^{DT_{4}}$ exists and $\overline{\mathcal{M}}_{c}^{DT_{4}}=\overline{\mathcal{M}}_{c}(L_{r})$ is compact smooth whose
virtual fundamental class is the usual fundamental class of $\overline{\mathcal{M}}_{c}(L_{r})$. \\

Using the $\mu$-map to define corresponding $DT_{4}$ invariants, we need the universal bundle of the moduli space. Here, bundles in the moduli space always
come from extensions of two line bundles, we know the universal family of the moduli space comes from the universal extension \cite{lange}
\begin{equation}0\rightarrow \pi_{1}^{*}L_{2}\otimes \pi_{2}^{*}\mathcal{O}(1)\rightarrow\mathcal{E}\rightarrow\pi_{1}^{*}L_{1}\rightarrow 0,
\nonumber\end{equation}
\begin{equation}L_{1}=\mathcal{O}_{X}(\epsilon_{1}+1,\epsilon_{2}-1), L_{2}=\mathcal{O}_{X}(-1,1), \nonumber\end{equation}
where $\pi_{1}: X\times\overline{\mathcal{M}}_{c}(L_{r})\rightarrow X$,
$\pi_{2}: X\times\overline{\mathcal{M}}_{c}(L_{r})\rightarrow \overline{\mathcal{M}}_{c}(L_{r})$ are the projection maps.
And the Chern class of the universal bundle is
\begin{equation}c(\mathcal{E})=\big(1+\pi_{1}^{*}c_{1}(L_{1})\big)\big(1+\pi_{1}^{*}c_{1}(L_{2})+\pi_{2}^{*}c_{1}(\mathcal{O}(1))\big),
\nonumber\end{equation}
which ensures us to compute all $DT_{4}$ invariants in this example. \\
${}$ \\
\textbf{The wall crossing phenomenon}.
By Lemma \ref{lq example}, when the parameter $r$ is small,
the moduli space is empty and the invariant is zero. When $r$ crosses the critical value
$\frac{15(2-\epsilon_{2})}{\epsilon_{1}(1+2\epsilon_{2})}$, the virtual cycle will be nontrivial and produces nonzero invariants.
Hence, in general the wall-crossing phenomenon exists in $DT_{4}$ theory.

To sum up, we have
\begin{theorem}\label{liqin eg}
Let $X$ be a generic smooth hyperplane section in $W=\mathbb{P}^{1}\times\mathbb{P}^{4}$ of $(2,5)$ type.
Let
\begin{equation}cl=[1+(-1,1)|_{X}]\cdot[1+(\epsilon_{1}+1,\epsilon_{2}-1)|_{X}],
\nonumber
\end{equation}
\begin{equation}k=(1+\epsilon_{1})\left(\begin{array}{l}6-\epsilon_{2} \\ \quad 4\end{array}\right), \quad \epsilon_{1},\epsilon_{2}=0,1.
\nonumber\end{equation}
Denote $\overline{\mathcal{M}}_{c}(L_{r})$ to be the moduli space of Gieseker $L_{r}$-semistable rank-2
torsion-free sheaves with Chern character $c$ (which can be easily read from the total Chern class $cl$), where $L_{r}=\mathcal{O}_{W}(1,r)|_{X}$. \\
$(1)$ If
\begin{equation}\frac{15(2-\epsilon_{2})}{6+5\epsilon_{1}+2\epsilon_{2}}<r<\frac{15(2-\epsilon_{2})}{\epsilon_{1}(1+2\epsilon_{2})},
\nonumber\end{equation}
then $\overline{\mathcal{M}}_{c}^{DT_{4}}$ exists and $\overline{\mathcal{M}}_{c}^{DT_{4}}=\overline{\mathcal{M}}_{c}(L_{r})=\mathbb{P}^{k}$, $[\overline{\mathcal{M}}_{c}^{DT_{4}}]^{vir}=[\mathbb{P}^{k}]$. \\
${}$ \\
$(2)$ If
\begin{equation} 0<r<\frac{15(2-\epsilon_{2})}{6+5\epsilon_{1}+2\epsilon_{2}},\nonumber\end{equation}
then $\overline{\mathcal{M}}_{c}^{DT_{4}}=\emptyset $ and $[\overline{\mathcal{M}}_{c}^{DT_{4}}]^{vir}=0$. \\
\end{theorem}
\begin{remark}From the above example, we can see the $Spin(7)$ instanton moduli space is not spin in general.
\end{remark}

\subsection{Moduli space of ideal sheaves of one point}
In this subsection, we assume $X$ is a compact Calabi-Yau four-fold. And we consider the Hilbert scheme of one point on $X$,
i.e. $\mathcal{M}_{c}=X$. Here we prefer (actually equivalent to ideal sheaves of one point if $Hol(X)=SU(4)$) considering the moduli space of structure sheaf of one point.

By the standard Koszul resolution, we have
\begin{equation}\cdot\cdot\cdot\rightarrow \mathcal{O}\otimes {\wedge}^{2}T_{p}^{*}\rightarrow \mathcal{O}\otimes
 T_{p}^{*}\rightarrow \mathcal{O}\rightarrow \mathcal{O}_{p}\rightarrow 0.  \nonumber\end{equation}
Then
\begin{eqnarray*}Ext^{i}(\mathcal{O}_{p},\mathcal{O}_{p})&=&Ext^{i}(\mathcal{O}\otimes {\wedge}^{\bullet}T_{p}^{*},\mathcal{O}_{p} ) \\
&=& H^{i}(\mathcal{O}\otimes {\wedge}^{\bullet}T_{p}\otimes \mathcal{O}_{p}) \\
&=& {\wedge}^{i}T_{p}\otimes \mathcal{O}_{p}. \nonumber\end{eqnarray*}
The last equality is because the differentials in Koszul complex are 0 at point $p$. The above isomorphism is canonical and
we can identify the obstruction bundle as ${\wedge}^{2}T\mathcal{M}_{c}={\wedge}^{2}TX$.

Now we want to give a representation method to determine the Euler class of the self-dual obstruction bundle.

Because $SU(4)=Spin(6)$, let $V$ be a fundamental representation of $Spin(6)$ such that
\begin{equation} V\otimes_{\mathbb{R}}\mathbb{C}=\wedge^{2}T^{*}X. \nonumber \end{equation}
We take a complex bundle $U$ such that $V$ is its underlying real bundle, then the spinor bundle $S^{+}(V)=\wedge^{even}U\otimes K^{\frac{1}{2}}$,
where $K=\wedge^{3}U^{*}$ and $c_{3}\big(S^{+}(V)\big)=-c_{3}(U)$.

If we identify $T^{*}X=S^{+}(V)$, we get
\begin{equation} e(V)=c_{3}(U)=c_{3}(X). \nonumber \end{equation}
Because $ob_{+}=\wedge^{2}_{+}TX=V^{*}$, we have
\begin{equation}e(ob_{+})=e(V^{*})=-c_{3}(X). \nonumber \end{equation}
\begin{theorem}\label{moduli of one point}
If $Hol(X)=SU(4)$ and $c=(1,0,0,0,-1)$, then $\overline{\mathcal{M}}_{c}^{DT_{4}}$ exists and $\overline{\mathcal{M}}_{c}^{DT_{4}}=X$, $[\overline{\mathcal{M}}_{c}^{DT_{4}}]^{vir}=\pm PD\big(c_{3}(X)\big)$.
\end{theorem}
\begin{remark}If $Hol(X)=Sp(2)$, the real virtual dimension of $\overline{\mathcal{M}}_{c}^{DT_{4}}$ (moduli of ideal sheaves instead of structure sheaves of one point) is $1$
by Lemma {\ref{v.d of ideal sheaves of curves}}. By fixing the determinants of the ideal sheaves, the real virtual dimension becomes $2$ and
the reduced virtual cycle will still be zero.
\end{remark}

\newpage
\section{Appendix}

\subsection{Local Kuranishi models of $\mathcal{M}_{c}^{o}$ }
We will review several different local Kuranishi models of $\mathcal{M}_{c}^{o}$. They are isomorphic and related by some re-parametrizations of the deformation space.  \\
${}$ \\
\textbf{Standard Kuranishi model of $\mathcal{M}_{c}^{o}$ with gauge fixing condition $\overline{\partial}_{A}^{*}a''=0$ }
Let us first recall the standard local Kuranishi model of $\mathcal{M}_{c}^{o}$
using linear gauge fixing condition $\overline{\partial}_{A}^{*}a''=0$
\cite{fm}.
Denote
\begin{equation}\kappa: H^{0,1}(X,EndE)\rightarrow H^{0,2}(X,EndE)  \nonumber \end{equation}
to be the Kuranishi map defined by
\begin{equation}\kappa(\alpha)=\mathbb{H}^{0,2}\big(g^{-1}(\alpha)\wedge g^{-1}(\alpha)\big), \nonumber \end{equation}
where
\begin{equation}g(a'')=a''+\overline{\partial}_{A}^{-1}P_{\overline{\partial}_{A}}(a''\wedge a'').  \nonumber \end{equation}
Here $a''\in Ker\overline{\partial}_{A}^{*}\subseteq\Omega^{0,1}(X,EndE)_{k}$ and $\alpha\in H^{0,1}(X,EndE)$.
\begin{remark} ${}$ \\
1. The map
\begin{equation}\overline{\partial}_{A}: Im(\overline{\partial}_{A}^{*})\rightarrow \Omega^{0,2}(X,EndE)_{k-1}  \nonumber \end{equation}
is an isomorphism onto its image and $\overline{\partial}_{A}^{-1}$ is defined as its inverse. \\
2. By the linear gauge fixing condition $\overline{\partial}_{A}^{*}a''=0$ and the definition of the map in (1),
we know that $\overline{\partial}_{A}^{*}g(a'')=0$ which implies
that the above Kuranishi map $\kappa$ is well defined. \\
3. By the standard Kuranishi theory, we have
\begin{equation} \kappa^{-1}(0)=\left\{ \begin{array}{lll}
    a'' \textrm{ }\big{|} & \|a''\|_{k} < \epsilon'', F^{0,2}(\overline{\partial}_{A}+a'')=0 , \textrm{ } \overline{\partial}_{A}^{*}a''=0      
\nonumber\end{array}\right\}.\end{equation}
By a suitable complex gauge transformation, we have $\kappa^{-1}(0)\cong Q_{A}\cap P^{-1}(0)$ (\ref{QA def}).
\end{remark}
${}$ \\
\textbf{Standard Kuranishi model of $\mathcal{M}_{c}^{o}$ with gauge fixing condition $F\wedge \omega^{3}=0$}
If we use the gauge fixing condition $F\wedge \omega^{3}=0$ instead of $\overline{\partial}_{A}^{*}a''=0$, we have the following local
Kuranishi model of $\mathcal{M}_{c}^{o}$ near $\overline{\partial}_{A}$
\begin{equation}\widetilde{\kappa}: H^{0,1}(X,EndE)\rightarrow H^{0,2}(X,EndE),  \nonumber \end{equation}
which is defined by
\begin{equation}\label{kappa tilta}\widetilde{\kappa}(\alpha)=\mathbb{H}^{0,2}\big(\widetilde{g}^{-1}(\alpha)\wedge \widetilde{g}^{-1}(\alpha)\big),
\end{equation}
where $\alpha\in H^{0,1}(X,EndE) $ and
\begin{equation}\widetilde{g}: \Omega^{0,1}(EndE)_{k}\rightarrow H^{0,1}(EndE)\oplus {\overline{\partial}_{A}^{*}\Omega^{0,1}(EndE)}_{k}\oplus
{\overline{\partial}_{A}^{*}\Omega^{0,2}(EndE)}_{k-1}  \nonumber \end{equation}
satisfies
\begin{equation}\label{g tilta}\widetilde{g}(a'')=\bigg(\mathbb{H}(a''),\overline{\partial}_{A}^{*}a''-\frac{i}{2}\wedge(a'\wedge a''+a''\wedge a'),
\overline{\partial}_{A}^{*}\big(\overline{\partial}_{A}a''+P_{\overline{\partial}_{A}}(a''\wedge a'')
\big)\bigg). \end{equation}
We know that $\widetilde{\kappa}^{-1}(0)=Q_{A}\cap P^{-1}(0)$ by standard Kuranishi theory. \\
${}$ \\
\textbf{Another Kuranishi model of $\mathcal{M}_{c}^{o}$ with gauge fixing condition $F\wedge \omega^{3}=0$}
As appeared in the $DT_{4}$ equations, we have the following Kuranishi model of $\mathcal{M}_{c}^{o}$ with gauge fixing condition $F\wedge \omega^{3}=0$ near $\overline{\partial}_{A}$
\begin{equation}\tilde{\tilde{\kappa}}: H^{0,1}(X,EndE)\rightarrow H^{0,2}(X,EndE),  \nonumber \end{equation}
which is defined by
\begin{equation}\label{kappa double tilta}\tilde{\tilde{\kappa}}(\alpha)=\mathbb{H}^{0,2}\big(q^{-1}(\alpha)\wedge q^{-1}(\alpha)\big),
\end{equation}
where $\alpha\in H^{0,1}(X,EndE) $ and
\begin{equation}q: \Omega^{0,1}(EndE)_{k}\rightarrow H^{0,1}(EndE)\oplus {\overline{\partial}_{A}^{*}\Omega^{0,1}(EndE)}_{k}\oplus
{\overline{\partial}_{A}^{*}\Omega^{0,2}(EndE)}_{k-1}  \nonumber \end{equation}
satisfies
\begin{equation}\label{g double tilta}q(a'')=\bigg(\mathbb{H}(a''),\overline{\partial}_{A}^{*}a''-\frac{i}{2}\wedge(a'\wedge a''+a''\wedge a'),
\overline{\partial}_{A}^{*}\big(\overline{\partial}_{A}a''+P_{\overline{\partial}_{A}}(a''\wedge a'')+*_{4}P_{\overline{\partial}_{A}^{*}}(a''\wedge a'')\big)\bigg). \end{equation}
By Proposition \ref{QA intesect P=0 equals NA}, we know $\tilde{\tilde{\kappa}}^{-1}(0)$ really gives a local model of $\mathcal{M}_{c}^{o}$ near $\overline{\partial}_{A}$.

\subsection{Some remarks on the orientability of the determinant line bundles on the (generalized) $DT_{4}$ moduli spaces}
In this subsection, we prove the orientability of the determinant line bundles on the (generalized) $DT_{4}$ moduli spaces under certain assumptions. The technique is suggested by Zheng Hua which originates from an idea of A. Okounkov. \\

We first assume $\mathcal{M}_{c}^{o}\neq\emptyset$ and $H^{*}(\mathcal{M}_{c}^{o},\mathbb{Z}_{2})$ is finitely generated (which is satisfied if $\mathcal{M}_{c}^{o}=\mathcal{M}_{c}\neq\emptyset$). Hence $\mathcal{M}_{c}^{DT_{4}}=\mathcal{M}_{c}^{o}$ as topological space with finitely generated cohomologies in $\mathbb{Z}_{2}$ coefficient. \\

We denote the index bundle on the $DT_{4}$ moduli space $\mathcal{M}_{c}^{DT_{4}}$ to be $Ind$ such that
\begin{equation}Ind|_{E}=Ext^{1}(E,E)-Ext^{2}_{+}(E,E). \nonumber \end{equation}
After complexifying it, we get
\begin{equation}Ind_{\mathbb{C}}|_{E}=Ext^{1}(E,E)-Ext^{2}(E,E)+Ext^{3}(E,E). \nonumber \end{equation}
Then we can extend the complexified index bundle to $\mathcal{M}_{c}^{DT_{4}}\times \mathcal{M}_{c}^{DT_{4}}$ and
define its determinant line bundle $\mathcal{L}_{\mathbb{C}}=det(Ind_{\mathbb{C}})$. \\

Define an involution map between topological spaces
\begin{equation}\sigma: \mathcal{M}_{c}^{DT_{4}}\times \mathcal{M}_{c}^{DT_{4}}\rightarrow \mathcal{M}_{c}^{DT_{4}}\times \mathcal{M}_{c}^{DT_{4}},  \nonumber \end{equation}
\begin{equation}\sigma([A_{1}],[A_{2}])=([A_{2}],[A_{1}]).
\nonumber \end{equation}
By Serre duality pairing, we have
\begin{equation}\label{skew-sym}\sigma^{*}(\mathcal{L}_{\mathbb{C}})\cong \mathcal{L}_{\mathbb{C}}^{*}.  \end{equation}

We then refer to a result on algebraic topology from Theorem 3.16 of \cite{hatcher}
\begin{lemma}\label{Kunneth formula}
The cross product $H^{*}(X,R)\otimes_{R}H^{*}(Y,R)\rightarrow H^{*}(X\times Y,R)$ is an isomorphism of rings if $X$ and $Y$ are $CW$
complexes and $H^{k}(Y,R)$ is a finitely generated free $R$-module for all $k$.
\end{lemma}
\begin{remark}\label{use Kunneth for}
We have isomorphism
\begin{equation}H^{n}(\mathcal{M}_{c}^{DT_{4}}\times \mathcal{M}_{c}^{DT_{4}},\mathbb{Z}_{2})
\cong \oplus_{i+j=n}H^{i}(\mathcal{M}_{c}^{DT_{4}},\mathbb{Z}_{2})\otimes
H^{j}(\mathcal{M}_{c}^{DT_{4}},\mathbb{Z}_{2}). \nonumber \end{equation}
\end{remark}

Combine the involution map and the above remark, we get
\begin{proposition}\label{w1 square =0}
$w_{2}(\mathcal{L}_{\mathbb{C}}|_{\Delta})=0$, where $\Delta$ is the diagonal of $\mathcal{M}_{c}^{DT_{4}}\times \mathcal{M}_{c}^{DT_{4}}$.
\end{proposition}
\begin{proof}
By definition and Remark \ref{use Kunneth for}, we have $w_{2}(\mathcal{L}_{\mathbb{C}})\in H^{2}(\mathcal{M}_{c}^{DT_{4}}\times \mathcal{M}_{c}^{DT_{4}},\mathbb{Z}_{2})$,
\begin{equation}H^{2}(\mathcal{M}_{c}^{DT_{4}}\times \mathcal{M}_{c}^{DT_{4}},\mathbb{Z}_{2})=
H^{0}(\mathcal{M}_{c}^{DT_{4}},\mathbb{Z}_{2})\otimes H^{2}(\mathcal{M}_{c}^{DT_{4}},\mathbb{Z}_{2})\oplus H^{2}(\mathcal{M}_{c}^{DT_{4}},\mathbb{Z}_{2})\otimes H^{0}(\mathcal{M}_{c}^{DT_{4}},\mathbb{Z}_{2})\nonumber \end{equation}
\begin{equation}\oplus H^{1}(\mathcal{M}_{c}^{DT_{4}},\mathbb{Z}_{2})\otimes H^{1}(\mathcal{M}_{c}^{DT_{4}},\mathbb{Z}_{2}).
\nonumber \end{equation}
Assume $\{a_{i}\}$ is a basis of $H^{0}(\mathcal{M}_{c}^{DT_{4}},\mathbb{Z}_{2})$, $\{b_{i}\}$ is a basis of $H^{1}(\mathcal{M}_{c}^{DT_{4}},\mathbb{Z}_{2})$,
$\{c_{i}\}$ is a basis of $H^{1}(\mathcal{M}_{c}^{DT_{4}},\mathbb{Z}_{2})$ then
\begin{equation}w_{2}(\mathcal{L}_{\mathbb{C}})=\sum_{ij}n_{ij}a_{i}\otimes b_{j}+ \sum_{ij}m_{ij}b_{i}\otimes a_{j}+ \sum_{ij}k_{ij}c_{i}\otimes c_{j}.
\nonumber \end{equation}
Under the action of involution map
\begin{equation}\sigma^{*}\big(w_{2}(\mathcal{L}_{\mathbb{C}})\big)=\sum_{ij}m_{ij}a_{j}\otimes b_{i}+ \sum_{ij}n_{ij}b_{j}\otimes a_{i}+
\sum_{ij}k_{ij}c_{j}\otimes c_{i}.\nonumber \end{equation}
By (\ref{skew-sym}), we have
\begin{equation}\sigma^{*}\big(w_{2}(\mathcal{L}_{\mathbb{C}})\big)=w_{2}(\sigma^{*}\mathcal{L}_{\mathbb{C}})=w_{2}(\mathcal{L}_{\mathbb{C}}) . \nonumber \end{equation}
Hence $m_{ji}\equiv n_{ij}(mod\textrm{ }2)$, $k_{ji}\equiv k_{ij}(mod\textrm{ } 2)$ and when we restrict to the diagonal \textrm{ }
\begin{equation}w_{2}(\mathcal{L}_{\mathbb{C}}|_{\Delta})\equiv\sum_{ij}n_{ij}(a_{i}\cup b_{j}+b_{j}\cup a_{i})\equiv0 (mod\textrm{ } 2).  \nonumber \end{equation}
\end{proof}
By Proposition \ref{w1 square =0}, we know that
\begin{equation}\label{c1 no 2 torsion}c_{1}(\mathcal{L}_{\mathbb{C}}|_{\Delta})=0 (mod \textrm{ }2). \end{equation}
\begin{corollary}\label{c1 equals zero if no 4k torsion}
If $H^{2}(\mathcal{M}_{c}^{DT_{4}},\mathbb{Z})$ $($equivalently $H_{1}(\mathcal{M}_{c}^{DT_{4}},\mathbb{Z})$$)$ does not have torsion of type $\mathbb{Z}_{4k}$, where $k\geq1$. Then $c_{1}(\mathcal{L}_{\mathbb{C}}|_{\Delta})=0$.
\end{corollary}
\begin{proof}
By (\ref{skew-sym}), we know $2c_{1}(\mathcal{L}_{\mathbb{C}}|_{\Delta})=0$. Combined with (\ref{c1 no 2 torsion}) and the assumption, we are done.
\end{proof}
\begin{remark}
Similar to a construction in $K$-theory \cite{atiyah}, we can find a trivial complex bundle with standard quadratic form $(\mathbb{C}^{N},q)$ such that
$(Ind_{\mathbb{C}}|_{\Delta},Q_{Serre})\oplus (\mathbb{C}^{N},q)$ becomes a quadratic bundle over $\mathcal{M}_{c}^{DT_{4}}$. By Corollary \ref{c1 equals zero if no 4k torsion}, the structure group of the quadratic bundle is reduced to $SO(n,\mathbb{C})$. Thus, $w_{1}(Ind|_{\Delta})=0$
if $H_{1}(\mathcal{M}_{c}^{DT_{4}},\mathbb{Z})$ does not have torsion of type $\mathbb{Z}_{4k}$.
\end{remark}
${}$ \\
Then we come to the case of the generalized $DT_{4}$ moduli space. As we know, the generalized $DT_{4}$ moduli space
exists and $\overline{\mathcal{M}}_{c}^{DT_{4}}=\mathcal{M}_{c}$ when the Gieseker moduli space of stable sheaves $\mathcal{M}_{c}$ is smooth. The index bundle and the above involution map can be similarly defined, then we know
\begin{corollary}
If the Gieseker moduli space of stable sheaves $\mathcal{M}_{c}$ is smooth and $H_{1}(\mathcal{M}_{c},\mathbb{Z})$ does not have torsion of type $\mathbb{Z}_{4k}$, where $k\geq1$. Then the index bundle of the generalized $DT_{4}$ moduli space is oriented.
\end{corollary}
\begin{proof}
Because  $\mathcal{M}_{c}$ is smooth, we are reduced to prove the orientability of the self-dual obstruction bundle $ob_{+}$.

By (\ref{skew-sym}), (\ref{c1 no 2 torsion}) and the assumption, we know
\begin{equation}c_{1}(ob_{+}\otimes\mathbb{C})=c_{1}(ob)=0. \nonumber \end{equation}

With the help of Serre duality pairing, the structure group of complex bundle $ob$ is reduced to $SO(n,\mathbb{C})$. Hence the structure group of the corresponding real bundle $ob_{+}$ sits inside $SO(n,\mathbb{R})$.
\end{proof}

\subsection{Seidel-Thomas twists}
In this subsection, we recall the definition of Seidel-Thomas twist \cite{js}.
\begin{definition}
Let $(X,\mathcal{O}_{X}(1))$ be a projective Calabi-Yau $m$-fold with $Hol(X)=SU(m)$. For each $n\in \mathbb{Z}$,
the Seidel-Thomas twist $T_{\mathcal{O}_{X}(-n)}$ by $\mathcal{O}_{X}(-n)$ is the Fourier-Mukai transform from D(X) to D(X) with kernel
\begin{equation}K=cone(\mathcal{O}_{X}(n)\boxtimes\mathcal{O}_{X}(-n))\rightarrow \mathcal{O}_{\Delta}.
\nonumber \end{equation}
\end{definition}
In general, $T_{n}\triangleq T_{\mathcal{O}_{X}(-n)}[-1]$ maps sheaves to complex of sheaves, but for $n\gg 0$ we have
\begin{lemma}\cite{js} \label{seidel thomas twist lemma}
Let $U$ be a finite type $\mathbb{C}$-scheme and $\mathcal{F}_{U}$ is a coherent sheaf on $U\times X$ flat over $U$ i.e.
it is a $U$-family of coherent sheaves on $X$. Then for $n\gg 0$,
$T_{n}(\mathcal{F}_{U})$ is also a $U$-family of coherent sheaves on $X$.
\end{lemma}
Then let us see how the twist can map sheaves to vector bundles. We first recall the definition of homological dimension of a sheaf.
\begin{definition}
For a nonzero coherent sheaf $\mathcal{F}$, the homological dimension $hd(\mathcal{F})$ is the smallest $n\geq0$ for
which there exists an exact sequence in $coh(X)$
\begin{equation}0\rightarrow E_{n}\rightarrow E_{n-1}\cdot\cdot\cdot\rightarrow E_{0}\rightarrow \mathcal{F}\rightarrow 0
\nonumber \end{equation}
with $\{E_{i}\}_{i=0,...,n}$ are vector bundles.
\end{definition}
\begin{lemma}\cite{js}
Let $\mathcal{F}_{U}$, $n\gg0$ be the same as in Lemma \ref{seidel thomas twist lemma}, then for any $u\in U$,
we have $hd(T_{n}(\mathcal{F}_{u}))=max(hd(\mathcal{F}_{u})-1,0)$.
\end{lemma}
\begin{corollary}\cite{js}
Let $U$ be a finite type $\mathbb{C}$-scheme and $\mathcal{F}_{U}$ is a $U$-family of coherent sheaves on $X$.
Then there exists $n_{1},...n_{m}\gg0$ such that for $T_{n_{m}}\circ\cdot\cdot\cdot\circ T_{n_{1}}(\mathcal{F}_{U})$
is a $U$-family of vector bundles on $X$.
\end{corollary}
By the above successive Seidel-Thomas twists, we obtain an isomorphism between $\mathcal{M}_{c}$ and
some component of the moduli space of simple holomorphic bundles.

In general, Gieseker stability condition is lost after the above twist.
However the twist is a derived equivalence which preserves the extension groups. We fix such an isomorphism, then
\begin{equation}Ext^{i}(\mathcal{F},\mathcal{F})\cong Ext^{i}(E_{b},E_{b}), \textrm{ } i=1,2, \nonumber\end{equation}
\begin{equation}\label{seidel thomas twist}\kappa^{-1}(0)\cong \kappa'^{-1}(0), \end{equation}
where $\mathcal{F}\in\mathcal{M}_{c}$ and $\kappa$ is a Kuranishi map near $\mathcal{F}$ ,
$E_{b}$ is the corresponding holomorphic bundle and $\kappa'$ is a Kuranishi map near $E_{b}$.

One possible way to prove the vanishing result $\kappa_{+}=0 \Rightarrow \kappa =0$ in Assumption \ref{assumption on gluing} is to use
Seidel-Thomas twist transform everything to the bundle side and then prove it using gauge theory.

\subsection{A quiver representation of $\mathcal{M}_{c} $ }
Now let's recall a quiver representation of $\mathcal{M}_{c}$ given by \cite{bchr}. \\
\textbf{Set up}. Let $(Y,\mathcal{O}_{Y}(1))$ be a smooth projective variety and
$A\triangleq \bigoplus_{n\geq0}H^{0}(Y,\mathcal{O}_{Y}(n))$. Fix the Chern character $c$ and
denote the Hilbert polynomial of elements in $\mathcal{M}_{c}$ by $\alpha(t)$.
Let $V$ be a finite dimensional complex vector space with grading
\begin{equation}V=\bigoplus_{i=p}^{q}V_{i}, \nonumber \end{equation}
where $dim V_{i}=\alpha(i)$. Denote the dimension vector by $\alpha=(dim V_{p},\cdot\cdot\cdot,dim V_{q})$. \\
${}$ \\
\textbf{A differential graded Lie algebra}. Consider the following dgla $(L^{\cdot},d,[ , ])$
\begin{equation}L^{n}=Hom_{gr}(A^{\bigotimes_{\mathbb{C}}n},End_{\mathbb{C}}V). \nonumber \end{equation}
The differential
\begin{equation}d: L^{n}\rightarrow L^{n+1}  \nonumber \end{equation}
is defined by
\begin{equation}d\mu(a_{1},\cdot\cdot\cdot,a_{n+1})=\sum_{i=1}^{n}(-1)^{n-i}\mu(\cdot\cdot\cdot,a_{i}a_{i+1},\cdot\cdot\cdot).
\nonumber \end{equation}
The bracket
\begin{equation}[,]: L^{m}\times L^{n}\rightarrow L^{m+n} \nonumber \end{equation}
satisfies
\begin{equation}[\mu,\mu^{'}]=\mu\circ\mu^{'}-(-1)^{mn}\mu^{'}\circ\mu, \nonumber \end{equation}
where $\mu\in L^{m}$, $\mu^{'}\in L^{n}$ and $\mu\circ\mu^{'}\in L^{m+n}$ is defined by
\begin{equation}\mu\circ\mu^{'}(a_{1},\cdot\cdot\cdot,a_{m+n})=(-1)^{mn}\mu(a_{1},\cdot\cdot\cdot,a_{m})\circ\mu^{'}(a_{m+1},\cdot\cdot\cdot,a_{m+n}).
\nonumber \end{equation}
The differential $d$ acts as derivation with respect to the bracket i.e. for any $\mu\in L^{m}$, $\mu^{'}\in L^{n}$
\begin{equation}d[\mu,\mu^{'}]=[d\mu,\mu^{'}]+(-1)^{m}[\mu,d\mu^{'}]. \nonumber \end{equation}
Now we consider a gauge group whose Lie algebra is $L^{0}$
\begin{equation}G=\prod_{i=p}^{q}GL(V_{i}),   \nonumber \end{equation}
which acts on $L^{n}$ by conjugation
\begin{equation}(g\cdot\mu)(a_{1},\cdot\cdot\cdot,a_{n})=g\circ\mu(a_{1},\cdot\cdot\cdot,a_{n})\circ g^{-1}. \nonumber \end{equation}
It preserves the differential and acts as automorphism of the differential graded Lie algebra structure on $L$, i.e.
\begin{equation}d(g\cdot\mu)=g\cdot d\mu, \quad\quad g\cdot[\mu,\mu^{'}]=[g\cdot\mu,g\cdot\mu^{'}]. \nonumber \end{equation}
As for the corresponding Maurer-Cartan equation
\begin{equation}d\mu+\frac{1}{2}[\mu,\mu]=0 , \quad\quad \mu\in L^{1}. \nonumber \end{equation}
The zero loci of the above equation consists of elements of $L^{1}$ such that
\begin{equation}\mu(ab)=\mu(a)\circ\mu(b). \nonumber \end{equation}
Hence the Maurer-Cartan elements up to gauge equivalence are the graded $A$- modules up to isomorphism.

To do GIT quotient, we recall the definition of stable $A$-modules.
\begin{definition}\cite{bchr},\cite{king}
We call a $[p,q]$-graded $A$-module $M$ (denote the corresponding point in MC(L) by $\mu$) \textbf{stable}, if for every graded $\mu$-submodule
$0<W<V$ we have
\begin{equation}\frac{dim W_{p}}{dim V_{p}}<\frac{dim W_{q}}{dim V_{q}}.  \nonumber \end{equation}
\end{definition}
We denote $\widetilde{\mathfrak{M}}od_{\alpha\mid_{[p,q]}}^{s}(A)=MC(L)^{s}//\widetilde{G}$ be the GIT quotient of
stable $A$-modules with dimension vector $\alpha\mid_{[p,q]}$ by the reduced gauge group $\widetilde{G}=G/\Delta$,
where $\Delta$ is the diagonal scalar subgroup acting trivially on $L^{\cdot}$. \\
${}$ \\
\textbf{A quiver representation of $\mathcal{M}_{c}$}. Now we associate the above dgla with $\mathcal{M}_{c}$ on $Y$ by the following functor
\begin{equation}\Gamma_{[p,q]}\mathcal{F}=\bigoplus_{i=p}^{q}\Gamma(Y,\mathcal{F}(i)).  \nonumber \end{equation}
\begin{theorem}\cite{bchr} If $\alpha(t)$ is a primitive polynomial, there exists $0\ll p\ll q$ such that
\begin{equation}\Gamma_{[p,q]}: \mathcal{M}_{c}\rightarrow \widetilde{\mathfrak{M}}od_{\alpha\mid_{[p,q]}}^{s}(A)
\nonumber \end{equation}
is an open immersion with image as a union of connected components of $\widetilde{\mathfrak{M}}od_{\alpha\mid_{[p,q]}}^{s}(A)$.
Furthermore we have a closed imbedding
\begin{equation}\label{Ls//G}\mathcal{M}_{c}\hookrightarrow L^{s}//\widetilde{G},  \end{equation}
where $L^{s}//\widetilde{G}$ is a smooth projective scheme of finite dimension.
\end{theorem}
${}$ \\
\textbf{The corresponding moment maps}. For the purpose of gluing local Kuranishi models
\begin{equation}\kappa_{+}=\pi_{+}\circ\kappa:Ext^{1}(\mathcal{F},\mathcal{F})\rightarrow Ext^{2}_{+}(\mathcal{F},\mathcal{F}).
\nonumber \end{equation}
We need a moment map equation abstractly denoted by $\Lambda=0$ to kill the ambiguity of complex automorphism of $\kappa$
which is achieved via a Kempf-Ness type result in this specific case \cite{king}. Thus we get
\begin{equation}\widetilde{\mathfrak{M}}od_{\alpha\mid_{[p,q]}}^{s}(A)=(MC(L)\cap\Lambda^{-1}(0))/U(\alpha). \nonumber \end{equation}

To describe $\Lambda=0$, let us choose metrics on $A$, $\{V_{i}=\Gamma(X,\mathcal{F}(i))\}_{p\leq i\leq q}$ and denote the Lie algebra of $K$ by $\mathfrak{k}$,
where $K$ is the maximal compact subgroup of $G=\prod_{i=p}^{q}GL(V_{i})$ preserving the induced metric on $L^{1}$.
For the purpose of matching the stability with the Gieseker stability, we use the extremal character $\theta_{i}$ \cite{bchr}, where
\begin{equation}\theta_{p}=-dim V_{q}, \quad \theta_{q}=dim V_{p}, \quad \theta_{i,i\neq p,q}=0 . \nonumber \end{equation}
For these given $\theta_{i}$, define the characters of $G$ by
\begin{equation}\chi_{\theta}(g)=\prod_{i=p}^{q}det(g_{i})^{\theta_{i}}, \nonumber \end{equation}
where $g\in G=\prod_{i=p}^{q}GL(V_{i})$ and $g_{i}$ is the $i$-th component of $g$. By King \cite{king},
we have the moment map equation for $\mu\in L^{1}$
\begin{equation}\label{Moment map for Mc}(A\mu,\mu)=d\chi_{\theta}(A)      \end{equation}
with respect to the metric $(,)$ and any $A\in i\mathfrak{k}$. The action of $A$ on $L^{1}$ is given by $(A\mu)_{a}=A_{ta}\mu_{a}-\mu_{a}A_{ha}$,
where $a\in Q_{1}$ is an arrow connecting two vertices $i,j$ and $ha (ta)$ denote the head (tail) vertices of the arrow $a$ respectively. \\
${}$ \\
\textbf{Some facts on dgla}: \\
Locally around $\mu_{0}\in MC(L)$, we have
\begin{equation}d(\mu_{0}+\mu)+\frac{1}{2}[\mu_{0}+\mu,\mu_{0}+\mu]=0 \Leftrightarrow d^{\mu_{0}}(\mu)+\frac{1}{2}[\mu,\mu]=0,
\nonumber \end{equation}
where $d^{\mu_{0}}=d+[\mu_{0},]$ is the twisted differential with respect to $\mu_{0}$. \\

If we choose metrics on $L^{\cdot}$ and consider the complex $(L^{\cdot},d^{\mu_{0}})$
\begin{equation}0\rightarrow L^{0}\rightarrow L^{1}\rightarrow L^{2}\rightarrow\cdot\cdot\cdot,  \nonumber \end{equation}
we get an associated short exact sequence
\begin{equation}0\rightarrow H^{1}(L^{\cdot},d^{\mu_{0}})\rightarrow L^{1}/Im(d^{\mu_{0}})\rightarrow d^{\mu_{0}}L^{1}\rightarrow 0.
\nonumber \end{equation}
Using the metrics, we have
\begin{equation}L^{1}\cong H^{1}(L^{\cdot},d^{\mu_{0}})\oplus d^{\mu_{0}}L^{0}\oplus \delta^{\mu_{0}} L^{2},  \nonumber \end{equation}
where $\delta^{\mu_{0}}: L^{2}\rightarrow L^{1}$ with $\delta^{\mu_{0}}L^{2}\cong d^{\mu_{0}}L^{1}$ can be described as composition of \\
$1$. The projection $L^{2}\rightarrow d^{\mu_{0}}L^{1}$, \\
$2$. $(d^{\mu_{0}})^{-1}: d^{\mu_{0}}L^{1}\rightarrow (d^{\mu_{0}})^{-1}(d^{\mu_{0}}L^{1})\subset L^{1} $. \\

The Hodge decomposition theorem,
\begin{equation} d^{\mu_{0}}\delta^{\mu_{0}}+\delta^{\mu_{0}}d^{\mu_{0}}=I-\mathbb{H}^{\mu_{0}} \nonumber \end{equation}
follows from the standard theory of differential graded Lie algebra \cite{ma}. \\

\begin{lemma}\label{Moment map of dgla}
Around a stable sheaf $\mathcal{F}$ which corresponds to $\mu_{0}\in MC(L^{1})$, the point $\mu_{0}+\mu$ is stable if and only if
\begin{equation}Re(A\mu_{0},\mu)+(A\mu,\mu)=0\nonumber \end{equation}
for any $A\in i\mathfrak{k}$. Furthermore the combination of linearization of the above moment map equation and the $U(n)$
gauge fixing equation $Im(A\mu_{0},\mu)=0$ is equivalent to the linear $G$-gauge fixing equation $\delta^{\mu_{0}}\mu=0$.
\end{lemma}
\begin{proof}
${}$\\
1. By the stability of $\mu_{0}$, we have
\begin{equation}(A\mu_{0},\mu_{0})=d\chi_{\theta}(A) \nonumber \end{equation}
for any $A\in i\mathfrak{k}$. Then
\begin{equation}(A(\mu_{0}+\mu),\mu_{0}+\mu)=d\chi_{\theta}(A) \Leftrightarrow (A\mu_{0},\mu)+(A\mu,\mu_{0})+(A\mu,\mu)=0 .  \nonumber \end{equation}
While $A\in i\mathfrak{k}$, we get
\begin{equation}(A\mu,\mu_{0})=(\mu,A^{*}\mu_{0})=(\mu,A\mu_{0})=\overline{(A\mu_{0},\mu)} . \nonumber \end{equation}
Thus
\begin{equation}(A(\mu_{0}+\mu),\mu_{0}+\mu)=d\chi_{\theta}(A) \Leftrightarrow Re(A\mu_{0},\mu)+(A\mu,\mu)=0.  \nonumber \end{equation}
2. The combination of the two equations is
\begin{equation}(A\mu_{0},\mu)=0
\nonumber \end{equation}
for any $A\in i\mathfrak{k}$. While $\delta^{\mu_{0}}\mu=0$ is equivalent to $(\mu,d^{\mu_{0}}L^{0})=0$
with respect to the inner product and Hodge decomposition. Meanwhile
\begin{equation}(\mu,d^{\mu_{0}}A)=(\mu,[\mu_{0},A])=(\mu,\mu_{0}\circ A-A\circ\mu_{0})=(\mu,A\mu_{0}), \nonumber \end{equation}
by the definition. Thus we finish the proof.
\end{proof}
\begin{lemma}\label{lemma on dgla}
If $H^{\mu_{0}}([\mu,\mu])=0$ and $\mu$ is small with respect to the chosen metric, we have
\begin{equation} d^{\mu_{0}}(\mu+\frac{1}{2}\delta^{\mu_{0}}[\mu,\mu])=0 \Leftrightarrow  d^{\mu_{0}}\mu+\frac{1}{2}[\mu,\mu]=0.
\nonumber \end{equation}
\end{lemma}
\begin{proof}
${}$ \\
$\Leftarrow)$ Trivial by Hodge decomposition. \\
$\Rightarrow)$ We have
\begin{eqnarray*}d^{\mu_{0}}\mu+\frac{1}{2}[\mu,\mu]
&=&d^{\mu_{0}}\mu+\frac{1}{2}(d^{\mu_{0}}\delta^{\mu_{0}}+\delta^{\mu_{0}}d^{\mu_{0}}+\mathbb{H}^{\mu_{0}})[\mu,\mu]\\
&=& \frac{1}{2}\delta^{\mu_{0}}d^{\mu_{0}}[\mu,\mu].
\end{eqnarray*}
Thus we are reduced to show the vanishing of $\delta^{\mu_{0}}d^{\mu_{0}}[\mu,\mu] $,
\begin{eqnarray*}\delta^{\mu_{0}}d^{\mu_{0}}[\mu,\mu]&=&\delta^{\mu_{0}}([d^{\mu_{0}}\mu,\mu]-[\mu,d^{\mu_{0}}\mu]) \\
&=& 2\delta^{\mu_{0}}[d^{\mu_{0}}\mu,\mu] \\
&=& -\delta^{\mu_{0}}[d^{\mu_{0}}\delta^{\mu_{0}}[\mu,\mu],\mu] \\
&=& -\delta^{\mu_{0}}[(Id-\mathbb{H}^{\mu_{0}}-\delta^{\mu_{0}}d^{\mu_{0}})[\mu,\mu],\mu] \\
&=& -\delta^{\mu_{0}}[[\mu,\mu],\mu]+\delta^{\mu_{0}}[\delta^{\mu_{0}}d^{\mu_{0}}[\mu,\mu],\mu]   \\
&=& \delta^{\mu_{0}}[\delta^{\mu_{0}}d^{\mu_{0}}[\mu,\mu],\mu],
\end{eqnarray*}
where the last equality is by the Jacobi identity $[[\mu,\mu],\mu]=0$

Thus we get
\begin{equation}\delta^{\mu_{0}}d^{\mu_{0}}[\mu,\mu]=\delta^{\mu_{0}}[\delta^{\mu_{0}}d^{\mu_{0}}[\mu,\mu],\mu]. \nonumber \end{equation}
Then by Proposition 4.6 of \cite{ma}, we know $\delta^{\mu_{0}}d^{\mu_{0}}[\mu,\mu]=0$ if $\mu$ is small.
\end{proof}
${}$ \\
\textbf{Definition of the candidate local models of generalized $DT_{4}$ moduli spaces}.
Now we define the candidate local model of the generalized $DT_{4}$ moduli space near $\mathcal{F}\in \mathcal{M}_{c}$. We pick $\mu_{0}\in MC(L^{1})$
which corresponds to $\mathcal{F}$.
\begin{definition}\label{local model of gene DT4 moduli space}
The candidate local Kuranishi model of the generalized $DT_{4}$ moduli space near a stable sheaf $(\mathcal{F},\mu_{0})$ is defined to be
\begin{equation}\kappa_{+}=\pi_{+}\kappa: Ext^{1}(\mathcal{F},\mathcal{F})\rightarrow  Ext^{2}_{+}(\mathcal{F},\mathcal{F}), \nonumber \end{equation}
where $\kappa$ is a canonical Kuranishi map for $\mathcal{M}_{c}$ near $\mathcal{F}$ defined by
\begin{equation}\kappa(\alpha)=\mathbb{H}^{\mu_{0}}[f^{-1}(\alpha),f^{-1}(\alpha)]. \nonumber \end{equation}
The map
\begin{equation}f: U_{\mu_{0}}\rightarrow H^{1}(L^{\cdot},d^{\mu_{0}})=Ext^{1}(\mathcal{F},\mathcal{F}) \nonumber \end{equation}
is the harmonic projection map which is a local diffeomorphism and
\begin{equation}
U_{\mu_{0}}=\{d^{\mu_{0}}+\mu \textrm{ }\big{|} \mid \mu\mid <\epsilon , \textrm{  }
\Lambda^{\mu_{0}}=0,  \quad  \delta^{\mu_{0}}_{Im}=0,   \quad d^{\mu_{0}}(\mu+\frac{1}{2}\delta^{\mu_{0}}[\mu,\mu])=0 \},
\nonumber \end{equation}
where $\Lambda^{\mu_{0}}=0$ is the moment map equation (\ref{Moment map for Mc}), $\delta^{\mu_{0}}_{Im}=0$ is
the linear $U(n)$-gauge fixing equation in Lemma \ref{Moment map of dgla}
and $\mathbb{H}^{\mu_{0}}$ is the projection map to $d^{\mu_{0}}$ harmonic subspace $H^{2}(L^{\cdot},d^{\mu_{0}})=Ext^{2}(\mathcal{F},\mathcal{F})$.
$\pi_{+}$ is the projection
\begin{equation}\pi_{+}: Ext^{2}(\mathcal{F},\mathcal{F})\rightarrow Ext^{2}_{+}(\mathcal{F},\mathcal{F})  \nonumber \end{equation}
in Corollary \ref{cutting for sheaves}.
\end{definition}
\begin{remark}
By Lemma \ref{lemma on dgla}, $\kappa^{-1}(0)\subseteq \mathcal{M}_{c}$ as an open subset.
Hence, if we assume $\kappa_{+}=0 \Rightarrow \kappa=0$ in Definition \ref{local model of gene DT4 moduli space}, we know
that $\kappa_{+}^{-1}(0)$ can be identified with $\kappa^{-1}(0)\subseteq \mathcal{M}_{c}$ as sets.
\end{remark}

The Institute of Mathematical Sciences and Department of Mathematics, The Chinese University of Hong Kong, Shatin, Hong Kong \\
\textit{E-mail address}: \texttt{ylcao@math.cuhk.edu.hk}


\end{document}